\documentclass{amsart}
\usepackage{amssymb, amsmath, mathrsfs, verbatim, url}
\usepackage[all,cmtip]{xy}
\usepackage{mathabx}

\theoremstyle{plain}
\newtheorem{theorem}{Theorem}[section]
\newtheorem{lemma}[theorem]{Lemma}
\newtheorem{proposition}[theorem]{Proposition}
\newtheorem{corollary}[theorem]{Corollary}

\theoremstyle{definition}
\newtheorem{definition}[theorem]{Definition}
\newtheorem{question}[theorem]{Question}
\newtheorem{conjecture}[theorem]{Conjecture}

\newcommand{\re}{\upharpoonright}
\newcommand{\id}{\mathsf{id}}
\newcommand{\supr}{\mathsf{sup}}
\newcommand{\bool}{\mathsf{b}}
\newcommand{\Cof}{\mathsf{Cof}}
\newcommand{\Fin}{\mathsf{Fin}}
\newcommand{\Ne}{\mathsf{N}}
\newcommand{\cl}{\mathsf{cl}}
\newcommand{\bD}{\mathbf{\Delta}}
\newcommand{\bG}{\mathbf{\Gamma}}
\newcommand{\bGc}{\widecheck{\mathbf{\Gamma}}}
\newcommand{\bL}{\mathbf{\Lambda}}
\newcommand{\bLc}{\widecheck{\mathbf{\Lambda}}}
\newcommand{\bS}{\mathbf{\Sigma}}
\newcommand{\bP}{\mathbf{\Pi}}

\newcommand{\Borel}{\mathsf{Borel}}
\newcommand{\ZFC}{\mathsf{ZFC}}
\newcommand{\ZF}{\mathsf{ZF}}
\newcommand{\AC}{\mathsf{AC}}
\newcommand{\DC}{\mathsf{DC}}
\newcommand{\AD}{\mathsf{AD}}
\newcommand{\PD}{\mathsf{PD}}
\newcommand{\Det}{\mathsf{Det}}
\newcommand{\Ga}{\mathsf{G}}
\newcommand{\Diff}{\mathsf{D}}
\newcommand{\Diffc}{\widecheck{\mathsf{D}}}
\newcommand{\NSD}{\mathsf{NSD}}
\newcommand{\NSDS}{\mathsf{NSD}_{\mathbf{\Sigma}}}
\newcommand{\NSDxi}{\mathsf{NSD}^{(\xi)}}
\newcommand{\NSDxiS}{\mathsf{NSD}_{\mathbf{\Sigma}}^{(\xi)}}
\newcommand{\SD}{\mathsf{SD}}
\newcommand{\SDc}{\widecheck{\mathsf{SD}}}
\newcommand{\SU}{\mathsf{SU}}
\newcommand{\SUc}{\widecheck{\mathsf{SU}}}
\newcommand{\PU}{\mathsf{PU}}
\newcommand{\wc}{\!\!\downarrow\,}
\newcommand{\Aa}{\mathcal{A}}
\newcommand{\BB}{\mathcal{B}}
\newcommand{\FF}{\mathcal{F}}
\newcommand{\HH}{\mathcal{H}}
\newcommand{\II}{\mathcal{I}}
\newcommand{\PP}{\mathcal{P}}
\newcommand{\RR}{\mathcal{R}}
\newcommand{\Ss}{\mathcal{S}}
\newcommand{\TT}{\mathcal{T}}
\newcommand{\UU}{\mathcal{U}}
\newcommand{\XX}{\mathcal{X}}
\newcommand{\PPP}{\mathbb{P}}
\newcommand{\QQQ}{\mathbb{Q}}
\newcommand{\RRR}{\mathbb{R}}

\begin{document}

\title[Classification of the zero-dimensional homogeneous spaces]{A complete classification of the zero-dimensional homogeneous spaces under determinacy}

\author{Andrea Medini}
\address{Institut f\"{u}r Diskrete Mathematik und Geometrie
\newline\indent Technische Universit\"{a}t Wien
\newline\indent  Wiedner Hauptstra\ss e 8-10/104
\newline\indent 1040 Vienna, Austria}
\email{andrea.medini@tuwien.ac.at}
\urladdr{http://www.dmg.tuwien.ac.at/medini/}



\thanks{The author acknowledges the support of the FWF grants P 35655-N and P 35588-N}

\date{October 22, 2025}

\begin{abstract}
All spaces are assumed to be separable and metrizable. We give a complete classification of the zero-dimensional homogeneous spaces, under the Axiom of Determinacy. This classification is expressed in terms of topological complexity (in the sense of Wadge theory) and Baire category. In the same spirit, we also give a complete classification of the filters on $\omega$ up to homeomorphism. As byproducts, we obtain purely topological characterizations of the semifilters and filters on $\omega$. The Borel versions of these results are in almost all cases due to Fons van Engelen. Along the way, we obtain Wadge-theoretic results of independent interest, especially regarding closure properties.
\end{abstract}

\maketitle

\tableofcontents

Throughout this article, unless we specify otherwise, we will be working in the theory $\ZF+\DC$, that is, the usual axioms of Zermelo-Fraenkel (without the Axiom of Choice) plus the principle of Dependent Choices (see \S\ref{subsection_preliminaries_determinacy} for more details). By \emph{space} we will always mean separable metrizable topological space.

\section{Introduction}\label{section_introduction}

\subsection{History}\label{subsection_introduction_history}

We begin with the notion that motivates the research contained in this article. Intuitively, a space is homogeneous if all of its points ``look alike.''

\begin{definition}\label{definition_homogeneous}
A space $X$ is \emph{homogeneous} if for every $(x,y)\in X\times X$ there exists a homeomorphism $h:X\longrightarrow X$ such that $h(x)=y$.
\end{definition}

Homogeneity is a classical notion in topology, which has been studied in depth (see the survey \cite{arhangelskii_van_mill}). Every topological group is homogeneous. However, even among the zero-dimensional Borel spaces, the converse implication does not hold (see Corollary \ref{corollary_semifilter_homogeneous} and the proof of Proposition \ref{proposition_counterexample_semifilter}).

\begin{definition}\label{definition_wadge_class}
Given a space $Z$, we will say that $\bG$ is a \emph{Wadge class} in $Z$ if there exists $A\subseteq Z$ such that
$$
\bG=\{f^{-1}[A]:f:Z\longrightarrow Z\text{ is a continuous function}\}.
$$
\end{definition}

The systematic study of these classes, founded by William Wadge in his doctoral thesis \cite{wadge_thesis} (see also \cite{wadge_cabal}), is known as Wadge theory, and it has become a classical topic in descriptive set theory (see \cite[\S21.E]{kechris}). Under the Axiom of Determinacy, the collection of all Wadge classes on a zero-dimensional Polish space~$Z$, ordered by inclusion, constitutes a well-behaved hierarchy that is similar to, but much finer than the well-known Borel hierarchy (and not limited to sets of low complexity). Furthermore, as we will discuss in \S\ref{subsection_homogeneous_complexity}, when $\bG$ is a \emph{good} Wadge class in $\omega^\omega$ (see Definition \ref{definition_good}), it will make sense to talk about spaces of \emph{exact complexity} $\bG$. This concept will feature prominently in the statements (and proofs) of our main results.

In his remarkable doctoral thesis \cite{van_engelen_thesis} (see also \cite{van_engelen_ambiguous} and \cite{van_engelen_homogeneous}), Fons van Engelen pioneered the application of Wadge theory to topology by giving a complete classification of the zero-dimensional homogeneous Borel spaces. This classification can be split into three cases. The trivial case is that of the locally compact spaces, since the only zero-dimensional homogeneous ones are the non-empty discrete spaces, $2^\omega$, and $\omega\times 2^\omega$ (see Proposition \ref{proposition_locally_compact}). The next case is that of the low complexity spaces (more precisely, of complexity below $\Delta(\Diff_\omega(\bS^0_2))$) that are not locally compact. The simplest examples of this kind are $\QQQ$ and $\QQQ\times 2^\omega$. This case can be viewed as a unified generalization of a series of classical results that characterize certain ``famous'' spaces (these results are neatly collected in \cite[Theorem 2.4]{van_engelen_ambiguous}). The final case is that of the high complexity spaces (more precisely, of complexity above $\Delta(\Diff_\omega(\bS^0_2))$). It is only in this case that Wadge theory becomes relevant, ultimately because of the closure properties that it provides. For example, in order to construct a homogeneous space of a given complexity $\bG$, we will start with an arbitrary set of complexity $\bG$, then perform certain operations on it to ensure homogeneity: suitable closure properties will guarantee that these operations preserve complexity (see the proof of Theorem \ref{theorem_meager_semifilter}).

\subsection{Our results}\label{subsection_introduction_results}

The aim of this article is to extend van Engelen's classification beyond the Borel realm, under suitable determinacy assumptions. Naturally, our focus will be on the high complexity case, as the low complexity cases have already been exhaustively treated elsewhere (see \cite{van_engelen_ambiguous} or \cite[Chapters 2 and 3]{van_engelen_thesis}). In particular, under the Axiom of Determinacy, our results will yield a classification of \emph{all} zero-dimensional homogeneous spaces. Not surprisingly, this classification can be split into a uniqueness result and an existence result, outlined as follows (see Theorem \ref{theorem_classification_homogeneous} for a more precise formulation), where we assume that all spaces are zero-dimensional and $\bG$ is a good Wadge class in $\omega^\omega$. This assumption does not result in any loss of generality, since every zero-dimensional homogeneous space of high enough complexity has exact complexity $\bG$ for some good Wadge class $\bG$ in $\omega^\omega$ (see Theorem \ref{theorem_homogeneous_spaces_have_good_complexity}). We also remark that every homogeneous space is either meager or Baire (see Proposition \ref{proposition_homogeneous_dichotomy}). The uniqueness theorem was essentially proved in \cite{carroy_medini_muller_homogeneous} (see \S\ref{subsection_homogeneous_uniqueness} for more details).

We hope that the reader will not be too irritated by the fact that we have not yet defined most of the relevant notions, as it would not be feasible to do that here. Even so, we believe that stating our main results early will be useful in guiding the reader through this article. We will only say that $\ell(\bG)$ and $t(\bG)$ respectively denote the \emph{level} and \emph{type} of $\bG$. We also remark that right before the bibliography we give a list of symbols and terminology, which indicates the subsections in which they are defined.

\smallskip

\begin{center}
\textbf{The uniqueness theorem}\\
\smallskip
Up to homeomorphism, there exist at most two homogeneous spaces of exact complexity $\bG$: at most one of them is Baire, and at most one of them is meager.
\end{center}

\newpage

\begin{center}
\textbf{The existence theorem}\\
\smallskip
\smallskip
\smallskip
\begin{tabular}{c|c|c}
& $\bG$ has the & $\bG$ does not have the\\
& separation property & separation property\\\hline
$\ell(\bG)\geq 2$ & \multicolumn{2}{c}{There exist both Baire and meager}\\
or & \multicolumn{2}{c}{homogeneous spaces}\\
 $\ell(\bG)=1$ and $t(\bG)=3$ & \multicolumn{2}{c}{of exact complexity $\bG$}\\\hline
& There exist only Baire &  There exist only meager\\
$\ell(\bG)=1$ and $t(\bG)\in\{1,2\}$ & homogeneous spaces & homogeneous spaces\\
& of exact complexity $\bG$ & of exact complexity $\bG$
\end{tabular}
\end{center}

\bigskip

It will naturally follow from our proof of the above classification result that every zero-dimensional homogeneous space (with the trivial exception of locally compact spaces) is homeomorphic to a semifilter on $\omega$ (see Theorem \ref{theorem_characterization_semifilters}). Furthermore, building on the above classification, we will be able to determine precisely for which good Wadge classes $\bG$ in $\omega^\omega$ there exists a filter on $\omega$ of exact complexity $\bG$ (see Theorem \ref{theorem_classification_filters}). As a corollary, we will obtain a purely topological characterization of the filters on $\omega$ (see Theorem \ref{theorem_characterization_filters}). The Borel versions of all the results concerning filters were obtained by van Engelen in \cite{van_engelen_ideals}, and our generalizations rely heavily on his ideas.

Along the way, we will obtain several Wadge-theoretic results of independent interest, culminating with the closure properties of \S\ref{subsection_wadge_closure_main}. In this regard, the treatment of type given in \S\ref{section_wadge_type}, the preservation results of \S\ref{subsection_wadge_closure_preservation_type} and \S\ref{subsection_wadge_closure_preservation_separation}, and the theory of complexity described in \S\ref{subsection_homogeneous_complexity} also seem worth mentioning. We remark that, as in \cite{carroy_medini_muller_constructing}, we will generally state our results for uncountable zero-dimensional Polish spaces instead of just $\omega^\omega$ (which is usually the only case considered in the literature). Furthermore, we will give level-by-level statements with respect to the determinacy assumptions (see \S\ref{subsection_preliminaries_determinacy}).

In many cases, the generalization of a Borel result under suitable determinacy assumptions is trivial to achieve (simply play the same game, while the determinacy assumptions will guarantee that the proof applies to a wider class of payoff sets). The reason why this is not the case in the present context is that van Engelen's results are fundamentally tied to Louveau's analysis of the Borel Wadge classes from \cite{louveau_article}, and extending this very fine analysis to all Wadge classes does not seem to be feasible.\footnote{\,Partial results in this direction have been obtained in \cite{fournier}, but they are limited to very low levels of the projective hierarchy.} Luckily, the coarser analysis given in \cite{louveau_book} (as developed and generalized in \cite{carroy_medini_muller_constructing}) provides an adequate foundation for our purposes. As a most welcome side-effect, our results will not only generalize van Engelen's, but often have simpler proofs and simpler statements.

While being comfortable with the content of \cite{carroy_medini_muller_constructing} would undoubtedly be helpful, we tried to make the presentation accessible to readers that are not familiar with this article. More specifically, in \S\ref{section_wadge_fundamental}, we collected all the fundamental Wadge-theoretic results that will be needed later on. In addition to being convenient for the reader, this will allow us to incorporate more general versions of known results, with a view towards our intended applications. Furthermore, giving precise statements (and rigorous proofs) of results such as Corollary \ref{corollary_expansion_level}, Corollary \ref{corollary_new_from_old}, Theorem \ref{theorem_level_sup_hausdorff} and Theorem \ref{theorem_separation_generalized} seems to be valuable beyond the specific scope of this paper.

\subsection{Organization of the paper}\label{subsection_introduction_organization}

In \S\ref{section_preliminaries}, we will establish much of the terminology and notation to be employed in the remainder of this article (with the exception of Wadge theory). We will give basic facts concerning determinacy, the Baire property, filters, semifilters and homogeneity. Furthermore, we will discuss how these notions interact with each other (see for example the diagram in \S\ref{subsection_preliminaries_homogeneity}). As we explained above, \S\ref{section_wadge_fundamental} is essentially a concise introduction to Wadge theory, which heavily references \cite{carroy_medini_muller_constructing} for the proofs.

In \S\ref{section_wadge_type}, we will extend Louveau's notion of type beyond the Borel context, introduce separated unions, then use them to characterize the classes of small type (see Theorem \ref{theorem_characterization_small_type}). These results will be applied in \S\ref{section_wadge_closure}, in order to obtain certain closure properties. For our purposes, the most important results of this kind are those regarding good Wadge classes (see Corollary \ref{corollary_closure_good}) and those that link the level, the type and the separation property with closure under separated unions (see Theorems \ref{theorem_closure_main_separation} and \ref{theorem_closure_main_type}).

In \S\ref{section_homogeneous}, we will employ the above results to finally obtain the existence (and non-existence) results. In conjunction with the results of \cite{carroy_medini_muller_homogeneous}, this will yield the desired classification of the zero-dimensional homogeneous spaces. In particular, the results on closure under separated unions will allow us to construct the required homogeneous spaces (see Corollaries \ref{corollary_meager_semifilter} and \ref{corollary_baire_semifilter}). This section begins with a theory of complexity for zero-dimensional spaces, based on the closure properties of good Wadge classes, which will enable us to give natural-sounding statements of our classification results (in fact, we have already done so above). This seems to be another worthwhile improvement with respect to van Engelen's presentation.

In \S\ref{section_filters}, we will obtain the classification and characterization of the filters on $\omega$ mentioned above. For this purpose, we will employ the notion of closure under squares and prove a further existence result (see Theorem \ref{theorem_existence_filter}). Finally, we will discuss the relation between filters and topological groups, including a conjecture that would yield a complete classification of the zero-dimensional topological groups up to homeomorphism (see Conjecture \ref{conjecture_generalized_van_engelen}).

\section{Preliminaries}\label{section_preliminaries}

\subsection{Terminology and notation}\label{subsection_preliminaries_terminology}

Our reference for basic set theory will be \cite{jech}. Let $\Omega$ be a set. We will denote by $\PP(\Omega)$ the collection consisting of all subsets of $\Omega$. When $\Omega$ is countably infinite, we will identify every element of $\PP(\Omega)$ with its characteristic function, so that every subset of $\PP(\Omega)$ will inherit the subspace topology from $2^\Omega$. We will denote by $\Fin(\Omega)$ the collection of all finite subsets of $\Omega$, and by $\Cof(\Omega)$ the collection of all cofinite subsets of $\Omega$. We will use the notation $\Omega^{<\omega}=\bigcup_{n<\omega}\Omega^n$. Recall that $\sigma:\Omega^{<\omega}\longrightarrow\Omega^{<\omega}$ is \emph{monotone} if $\sigma(s)\subseteq\sigma(t)$ whenever $s,t\in \Omega^{<\omega}$ and $s\subseteq t$. Given $s\in\Omega^{<\omega}$, we will use the notation $\Ne_s=\{z\in\Omega^\omega:s\subseteq z\}$.

Let $f:X\longrightarrow Y$. Given $S\subseteq X$, we will use the notation $f[S]=\{f(x):x\in S\}$. Given $S\subseteq Y$, we will use the notation $f^{-1}[S]=\{x\in X:f(x)\in S\}$. When $X=Y$ and $f(x)=x$ for every $x\in X$, we will use the notation $\id_X=f$ for the \emph{identity} on $X$. Recall that $f$ is a \emph{retraction} if $f$ is continuous, $Y\subseteq X$ and $f\re Y=\id_Y$.

Our reference for topology will be \cite{engelking}. A \emph{base} for a space $X$ is a collection $\BB$ consisting of non-empty open subsets of $X$ such that for every $x\in X$ and every non-empty open neighborhood $U$ of $x$ there exists $V\in\BB$ such that $x\in V\subseteq U$. A \emph{local base} for a space $X$ at a point $x\in X$ is a collection $\BB$ consisting of open neighborhoods of $x$ such that for every open neighborhood $U$ of $x$ there exists $V\in\BB$ such that $V\subseteq U$. A \emph{$\pi$-base} $\BB$ for a space is a collection $\BB$ consisting of non-empty open sets such that for every non-empty open set $U$ there exists $V\in\BB$ such that $V\subseteq U$.

A subset of a space is \emph{clopen} if it is closed and open. A space is \emph{zero-dimensional} if it is non-empty and it has a base consisting of clopen sets.\footnote{\,The empty space has dimension $-1$ (see \cite[\S7.1]{engelking}).} When we say that a space $X$ is a \emph{topological group}, we really mean that there exists a topological group whose underlying space is $X$. A subset $S$ of a space $X$ is \emph{meager} if there exist closed nowhere dense subsets $C_n$ of $X$ for $n\in\omega$ such that $S\subseteq\bigcup_{n\in\omega}C_n$. A subset $S$ of a space $X$ is \emph{comeager} if $X\setminus S$ is meager. A space is \emph{meager} if it is a meager subset of itself. A space $X$ is \emph{Baire} if every non-empty open subset of $X$ is non-meager.

Let $X$ and $Y$ be spaces. We will write $X\approx Y$ to mean that $X$ and $Y$ are homeomorphic. An \emph{embedding} is a function $j:X\longrightarrow Y$ such that $j:X\longrightarrow j[X]$ is a homeomorphism. We will say that an embedding is \emph{dense} if it has dense range. A function $f:X\longrightarrow Y$ is \emph{closed} if $f[C]$ is a closed subset of $Y$ whenever $C$ is a closed subset of $X$. Similarly define an \emph{open} function. The straightforward proof of the following result is left to the reader.

\begin{lemma}\label{lemma_closed_function}
Let $Z$ and $W$ be spaces, and let $f:Z\longrightarrow W$ be a closed function. Assume that $X\subseteq Z$ and $Y\subseteq W$ are such that $f[X]\subseteq Y$ and $f[Z\setminus X]\subseteq W\setminus Y$. Then $f\re X:X\longrightarrow Y$ is closed.
\end{lemma}

Our reference for descriptive set theory will be \cite{kechris}. Given a space $Z$, we will denote by $\Borel(Z)$ the collection of all Borel subsets of $Z$, and by $\bS^0_1(Z)$, $\bP^0_1(Z)$ and $\bD^0_1(Z)$ respectively the collections of all open, closed and clopen subsets of~$Z$. A space $Z$ is a \emph{Borel space} if there exists a Polish space $W$ and an embedding $j:Z\longrightarrow W$ such that $j[Z]\in\Borel(W)$. In particular, every Polish space is a Borel space. By proceeding as in the proof of \cite[Proposition 4.2]{medini_zdomskyy_between}, it is easy to show that a space $Z$ is Borel iff $j[Z]\in\Borel(W)$ for every Polish space $W$ and every embedding $j:Z\longrightarrow W$. Given $1\leq\xi<\omega_1$ and spaces $Z$ and $W$, a function $f:Z\longrightarrow W$ is \emph{$\bS^0_\xi$-measurable} if $f^{-1}[U]\in\bS^0_\xi(Z)$ for every $U\in\bS^0_1(W)$. A function $f:Z\longrightarrow W$ is \emph{Borel} if $f^{-1}[U]\in\Borel(Z)$ for every $U\in\bS^0_1(W)$.

The classes defined below constitute the so-called \emph{difference hierarchy} (or \emph{small Borel sets}). For a detailed treatment, see \cite[\S22.E]{kechris} or \cite[Chapter 3]{van_engelen_thesis}. Here, we will only mention that the classes $\Diff_\eta(\bS^0_\xi(Z))$ are among the simplest concrete examples of Wadge classes (see \S\ref{subsection_wadge_fundamental_basics}).

\begin{definition}[Kuratowski]\label{definition_differences}
Let $Z$ be a space, let $1\leq\eta<\omega_1$, and let $1\leq\xi<\omega_1$. Given a sequence of sets $(A_\mu:\mu<\eta)$, define
$$
\Diff_\eta(A_\mu:\mu<\eta)= \left\{
\begin{array}{ll}
\bigcup\{A_\mu\setminus\bigcup_{\zeta<\mu}A_\zeta:\mu<\eta\text{ and }\mu\text{ is odd}\} & \text{if }\eta\text{ is even,}\\
\bigcup\{A_\mu\setminus\bigcup_{\zeta<\mu}A_\zeta:\mu<\eta\text{ and }\mu\text{ is even}\} & \text{if }\eta\text{ is odd.}
\end{array}
\right.
$$
Define $\Diff_\eta(\bS^0_\xi(Z))$ by declaring $A\in\Diff_\eta(\bS^0_\xi(Z))$ if there exist $A_\mu\in\bS^0_\xi(Z)$ for $\mu<\eta$ such that $A=\Diff_\eta(A_\mu:\mu<\eta)$.
\end{definition}

We conclude this subsection by stating a classical result that will serve as a fundamental technical tool in \S\ref{section_wadge_closure} (see \cite[Theorem 22.16]{kechris}).\footnote{\,The conclusion of this theorem could be stated more succinctly by saying that $\bS^0_\xi(X)$ has the \emph{generalized reduction property} (see \cite[Definition 22.14]{kechris}).}

\begin{theorem}\label{theorem_reduction}
Let $X$ be a zero-dimensional space, and let $1\leq\xi<\omega_1$. If $U_n\in\bS^0_\xi(X)$ for $n\in\omega$ then there exist $U'_n\in\bS^0_\xi(X)$ for $n\in\omega$ such that the following conditions are satisfied:
\begin{itemize}
\item $U'_n\subseteq U_n$ for each $n$,
\item $U'_m\cap U'_n=\varnothing$ whenever $m\neq n$,
\item $\bigcup_{n\in\omega}U'_n=\bigcup_{n\in\omega}U_n$.
\end{itemize}
\end{theorem}

\subsection{Determinacy and nice topological pointclasses}\label{subsection_preliminaries_determinacy}

For more extensive introductions to the topic of games, we refer to \cite[\S20]{kechris} or \cite[Chapter 33]{jech}. Here, we only want to give the definitions of strategy and determinacy. Given a set $\Omega$ and $S\subseteq\Omega^\omega$, a \emph{play} of the \emph{game} $\Ga=\Ga(\Omega,S)$ is described by the diagram
\begin{center}
\begin{tabular}{c|ccccl}
I & $x_0$ & & $x_1$ & & $\cdots$\\\hline
II & & $y_0$ & & $y_1$ & $\cdots$,
\end{tabular}
\end{center}
where $x_n,y_n\in\Omega$ for $n\in\omega$. We will say that Player I \emph{wins} this play of $\Ga$ if $(x_0,y_0,x_1,y_1,\ldots)\in S$. Player II \emph{wins} if Player I does not win. The set $S$ is called the \emph{payoff set}.

A \emph{strategy} for Player I on $\Omega$ is a monotone $\sigma:\Omega^{<\omega}\longrightarrow\Omega^{<\omega}$ such that $|\sigma(s)|=|s|+1$ for every $s\in\Omega^{<\omega}$. We will say that a strategy $\sigma$ for Player I on $\Omega$ is \emph{winning} in $\Ga$ if setting $x\re n+1=\sigma(y\re n)$ for $n\in\omega$ yields a play in which Player I wins for every $y\in\Omega^\omega$, where $x=(x_0,x_1,\ldots)$ and $y=(y_0,y_1,\ldots)$. Similarly, a \emph{strategy} for Player II on $\Omega$ is a monotone $\tau:\Omega^{<\omega}\longrightarrow\Omega^{<\omega}$ such that $|\tau(s)|=|s|$ for all $s\in\Omega^{<\omega}$, and $\tau$ is \emph{winning} in $\Ga$ if setting $y\re n=\tau(x\re n)$ for $n\in\omega$ yields a play in which Player II wins for every $x\in\Omega^\omega$. We will say that the game $\Ga$ (or simply the set $S$) is \emph{determined} if (exactly) one of the players has a winning strategy. For the purposes of this article, only the case in which $\Omega$ is countable will be relevant. Given $\bS\subseteq\PP(\omega^\omega)$, we will write $\Det(\bS)$ to mean that every element of $\bS$ is determined. The assumption $\Det(\PP(\omega^\omega))$ is known as the \emph{Axiom of Determinacy} ($\AD$). The assumption $\Det(\bigcup_{1\leq n<\omega}\bS^1_n(\omega^\omega))$ is known as the axiom of \emph{Projective Determinacy} ($\PD$).

We will use $\AC$ to denote the Axiom of Choice, and $\AC_\omega$ to denote the Countable Axiom of Choice. Recall that the principle of \emph{Dependent Choices} ($\DC$) states that if $R$ is a binary relation on a non-empty set $\Omega$ such that for every $x\in\Omega$ there exists $y\in\Omega$ such that $(y,x)\in R$, then there exists a sequence $(x_0,x_1,\ldots)\in\Omega^\omega$ such that $(x_{n+1},x_n)\in R$ for every $n\in\omega$. This principle is what is needed to carry out recursive constructions of length $\omega$. Furthermore, it is easy to see that the implications $\AC\rightarrow\DC\rightarrow\AC_\omega$ hold (see \cite[Exercise 5.7]{jech}).

As the determinacy assumptions get stronger, the set-theoretic universe becomes more regular. It follows that $\AC$ is incompatible with $\AD$, since $\AC$ enables one to construct pathological sets (see for example \cite[Lemma 33.1]{jech}). For this reason, as we mentioned at the beginning of the article, we will be working in $\ZF+\DC$.\footnote{\,As eloquently put by Kanamori, while comparing $\DC$ to a weaker axiom, ``it soon became \emph{de rigueur} to assume the stronger $\DC$ in much of the investigation of $\AD$'' (see \cite[page 378]{kanamori}).} It is a famous theorem of Martin that $\Det(\Borel(\omega^\omega))$ holds in $\ZF+\AC_\omega$ (this was originally proved in \cite{martin_determinacy}, but see also \cite[Remark (2) on page 307]{martin_inductive}). However, Harrington showed that $\Det(\bS^1_1(\omega^\omega))$ has large cardinal strength (see \cite{harrington}). For the consistency of $\ZF+\DC+\AD$, see \cite{neeman} and \cite[Proposition 11.13]{kanamori}.

Most results that use determinacy assumptions follow a consistent pattern: they will hold for Borel sets without any additional assumptions (by Martin's theorem), for projective sets under $\PD$, and for all sets under $\AD$. The following is \cite[Definition 3.1]{carroy_medini_muller_constructing}, and it provides a convenient expositional tool to simultaneously state the various versions of our results.\footnote{\,Notice that the term ``function'' in the following definition is an abuse of terminology, as each topological pointclass is a proper class. Therefore, every theorem in this paper that mentions these pointclasses is strictly speaking an infinite scheme (one theorem for each topological pointclass).} This is what we mean when we speak about ``level-by-level'' results. Given a set $Z$ and $\bS\subseteq\PP(Z)$, we will denote by $\bool\bS$ the smallest boolean subalgebra of $\PP(Z)$ that contains $\bS$.

\begin{definition}[Carroy, Medini, M\"uller]\label{definition_pointclass}
A function $\bS$ is a \emph{topological pointclass} if it satisfies the following requirements:
\begin{itemize}
\item The domain of $\bS$ is the class of all spaces,
\item $\bS(Z)\subseteq\PP(Z)$ for every space $Z$,
\item If $f:Z\longrightarrow W$ is continuous and $B\in\bS(W)$ then $f^{-1}[B]\in\bS(Z)$.
\end{itemize}
We will say that a topological pointclass $\bS$ is \emph{nice} if it satisfies the following additional requirements:
\begin{itemize}
\item $\Borel(Z)\subseteq\bS(Z)$ for every space $Z$,
\item $\bool\bS(Z)=\bS(Z)$ for every space $Z$,
\item If $f:Z\longrightarrow W$ is a Borel function and $B\in\bS(W)$ then $f^{-1}[B]\in\bS(Z)$,
\item For every space $Z$, if $j[Z]\in\bS(W)$ for some Borel space $W$ and embedding $j:Z\longrightarrow W$, then $i[Z]\in\bS(W)$ for every Borel space $W$ and embedding $i:Z\longrightarrow W$.
\end{itemize}
\end{definition}

For the purposes of this paper, the following are the intended examples of nice topological pointclasses.\footnote{\,For a more detailed discussion, see \cite[\S3]{carroy_medini_muller_constructing}.} Recall that $\bS^1_n(Z)$ for $1\leq n<\omega$ can be defined for an arbitrary space $Z$ by declaring $A\in\bS^1_n(Z)$ if there exists a Polish space $W$ and an embedding $j:Z\longrightarrow W$ such that $j[A]=\widetilde{A}\cap j[Z]$ for some $\widetilde{A}\in\bS^1_n(Z)$.
\begin{itemize}
 \item $\bS(Z)=\Borel(Z)$ for every space $Z$,
 \item $\bS(Z)=\bool\bS^1_n(Z)$ for every space $Z$, where $1\leq n<\omega$,
 \item $\bS(Z)=\bigcup_{1\leq n<\omega}\bS^1_n(Z)$ for every space $Z$,
 \item $\bS(Z)=\PP(Z)$ for every space $Z$.
\end{itemize}

We conclude this subsection by giving precise formulations of well-known results concerning the Baire property under suitable determinacy assumption. We remark that Corollary \ref{corollary_baire_dense_subspace} will be needed in the proof of Theorem \ref{theorem_uniqueness_embeddings}. Recall that a subset $A$ of a space $Z$ has the \emph{Baire property} if there exists an open subset $U$ of $Z$ such that $A\Delta U$ is a meager subset of $Z$.

\begin{theorem}\label{theorem_baire_property_baire_space}
Let $\bS$ be a nice topological pointclass, and assume that $\Det(\bS(\omega^\omega))$ holds. Then every element of $\bS(\omega^\omega)$ has the Baire property in $\omega^\omega$.
\end{theorem}
\begin{proof}
Set $\Omega=\omega^{<\omega}\setminus\{\varnothing\}$. Given $s_n\in\Omega$ for $n\in\omega$, where each $s_n\in\omega^{k_n+1}$ for suitable $k_n\in\omega$, define
$$
\varphi(s_0,s_1,\ldots)=(s_0(0),\ldots,s_0(k_0),s_1(0)\ldots,s_1(k_1),\ldots),
$$
and observe that $\varphi:\Omega^\omega\longrightarrow\omega^\omega$  is continuous. In order to prove that $A$ has the Baire property in $\omega^\omega$, by \cite[Exercises 8.35 and 8.36]{kechris}, it will be enough to show that $\Ga(\Omega,\varphi^{-1}[(\omega^\omega\setminus A)\cup U])$ is determined for every open subset $U$ of $\omega^\omega$, where $\Ga$ is the game introduced at the beginning of this subsection. This is straightforward to check, using the fact that $\bS$ is a nice topological pointclass and our determinacy assumptions.
\end{proof}

\begin{corollary}\label{corollary_baire_property_polish}
Let $\bS$ be a nice topological pointclass, and assume that $\Det(\bS(\omega^\omega))$ holds. Let $Z$ be a Polish space. Then every element of $\bS(Z)$ has the Baire property in $Z$.
\end{corollary}
\begin{proof}
The desired result is trivial if $Z$ is empty, so assume that $Z$ is non-empty. Then, by \cite[Exercise 7.14]{kechris}, we can fix an open continuous surjection $f:\omega^\omega\longrightarrow Z$. Pick $A\in\bS(Z)$. Notice that $f^{-1}[A]\in\bS(\omega^\omega)$ because $\bS$ is a topological pointclass, hence $f^{-1}[A]$ has the Baire property in $\omega^\omega$ by Theorem \ref{theorem_baire_property_baire_space}. This means that $f^{-1}[A]\,\Delta\, U$ is meager in $\omega^\omega$ for some open subset $U$ of $\omega^\omega$. Since
$$
A\,\Delta\, f[U]= f[f^{-1}[A]]\,\Delta\, f[U]\subseteq f[f^{-1}[A]\,\Delta\, U],
$$
it will be enough to show that the right-hand side is meager in $Z$. This follows from the fact that open continuous surjections map meager sets to meager sets.
\end{proof}

\begin{corollary}\label{corollary_baire_dense_subspace}
Let $\bS$ be a nice topological pointclass, and assume that $\Det(\bS(\omega^\omega))$ holds. Let $Z$ be a Polish space, and let $X\in\bS(Z)$ be a dense Baire subspace of $Z$. Then $X$ is comeager in $Z$.
\end{corollary}
\begin{proof}
First, observe that $X$ has the Baire property in $Z$ by Corollary \ref{corollary_baire_property_polish}. This means that there exist $G\in\bP^0_2(Z)$ and a meager subset $M$ of $Z$ such that $X=G\cup M$ (see \cite[Proposition 8.23]{kechris}). Notice that $M$ is also meager in $X$ by density. Since $X$ is a Baire space, it follows that $G$ is dense in $X$, hence in $Z$.
\end{proof}

\subsection{Filters and semifilters}\label{subsection_preliminaries_filters}

Throughout this subsection, we will assume that $\Omega$ is a given countably infinite set. We will say that $\XX\subseteq\PP(\Omega)$ is \emph{closed under finite modifications} if $|x\Delta y|<\omega$ and $x\in\XX$ implies $y\in\XX$ for all $x,y\subseteq\Omega$. A collection $\XX$ is \emph{closed under finite intersections} (respectively \emph{closed under finite unions}) if $x\cap y\in\XX$ (respectively $x\cup y\in\XX$) for all $x,y\in\XX$. A collection $\XX\subseteq\PP(\Omega)$ is \emph{upward-closed} (respectively \emph{downward-closed}) if $y\supseteq x\in\XX$ (respectively $y\subseteq x\in\XX$) implies $y\in\XX$ for all $x,y\in\PP(\Omega)$.

\begin{definition}\label{definition_semifilter}
A \emph{semifilter} on $\Omega$ is a collection $\Ss\subseteq\PP(\Omega)$ that satisfies the following conditions:
\begin{itemize}
\item $\varnothing\notin\Ss$ and $\Omega\in\Ss$,
\item $\Ss$ is closed under finite modifications,
\item $\Ss$ is upward-closed.
\end{itemize}
A \emph{filter} on $\Omega$ is a semifilter $\FF$ on $\Omega$ that satisfies the following additional condition:
\begin{itemize}
\item $\FF$ is closed under finite intersections.
\end{itemize}
A \emph{semiideal} on $\Omega$ is a collection $\RR\subseteq\PP(\Omega)$ that satisfies the following conditions:
\begin{itemize}
\item $\varnothing\in\RR$ and $\Omega\notin\RR$,
\item $\RR$ is closed under finite modifications,
\item $\RR$ is downward-closed.
\end{itemize}
An \emph{ideal} on $\Omega$ is a semiideal $\II$ on $\Omega$ that satisfies the following additional condition:
\begin{itemize}
\item $\II$ is closed under finite unions.
\end{itemize}
When the set $\Omega$ is not mentioned, we will assume that $\Omega=\omega$.
\end{definition}

A collection $\XX\subseteq\PP(\Omega)$ has the \emph{finite intersection property} (respectively the \emph{finite union property}) if $\bigcap F\notin\Fin(\Omega)$ (respectively $\bigcup F\notin\Cof(\Omega)$) whenever $F\subseteq\XX$ is finite and non-empty. One can easily check that a collection $\XX\subseteq\PP(\Omega)$ has the finite intersection property (respectively the finite union property) iff it can be extended to a filter (respectively an ideal) on $\Omega$.

Next, we will point out how semifilters (respectively filters) are indistinguishable from semiideals (respectively ideals) from the topological point of view. To be more precise, given $F\subseteq\omega$, $x\in 2^\omega$ and $n\in\omega$, define 
$$
\left.
\begin{array}{lcl}
& & h_F(x)(n)= \left\{
\begin{array}{ll}
1-x(n) & \textrm{if }n\in F,\\
x(n) & \textrm{if }n\in\omega\setminus F.
\end{array}
\right.
\end{array}
\right.
$$
It is clear that each $h_F:2^\omega\longrightarrow 2^\omega$ is a homeomorphism. Throughout this article, we will denote by $c=h_\omega$ the complement homeomorphism. Given any $\XX\subseteq\PP(\omega)$, it is trivial to check that $\XX$ is a semifilter (respectively a semiideal) iff $c[\XX]$ is a semiideal (respectively a semifilter). Similarly, one sees that $\XX$ is a filter (respectively an ideal) iff $c[\XX]$ is an ideal (respectively a filter). Since $\XX\approx c[\XX]$, this means that every result about the topology of semifilters (respectively filters) immediately translates to a result about semiideals (respectively ideals), and viceversa.

As an application of this principle, one can see that every filter is a topological group. In fact, it is straightforward to check that every ideal is a topological subgroup of $2^\omega$ with the operation of coordinatewise addition modulo $2$, and any space that is homeomorphic to a topological group is a topological group. In particular, every filter is homogeneous. As Corollary \ref{corollary_semifilter_homogeneous} will show, this is true of semifilters as well, but the proof is considerably more involved.

Another notable topological property of filters is given by the following result from \cite{medini_zdomskyy_filters}, which will be useful in \S\ref{subsection_filters_classification}.

\begin{theorem}[Medini, Zdomskyy]\label{theorem_filters_square}
Let $\FF$ be a filter. Then $\FF\times\FF\approx\FF$.
\end{theorem}

We conclude this subsection with some well-known facts about the Baire category of the above combinatorial objects under suitable determinacy assumptions.

\begin{proposition}\label{proposition_finite_modifications_meager_comeager}
Let $\bS$ be a nice topological pointclass, and assume that $\Det(\bS(\omega^\omega))$ holds. Let $\XX\in\bS(2^\omega)$ be closed under finite modifications. Then $\XX$ is either meager in $2^\omega$ or comeager in $2^\omega$.
\end{proposition}
\begin{proof}
Assume that $\XX$ is non-meager. Then, since $\XX$ has the Baire property in $2^\omega$ by Corollary \ref{corollary_baire_property_polish}, there exist $n\in\omega$ and $s\in 2^n$ such that $\XX\cap\Ne_s$ is comeager in $\Ne_s$ (see \cite[Proposition 8.26]{kechris}). Notice that
$$
2^\omega\setminus\XX=\bigcup\{h_F[\Ne_s\setminus\XX]:F\subseteq n\}
$$
because $\XX$ is closed under finite modifications. Since the right-hand side is a finite union of meager subsets of $2^\omega$, it follows that $\XX$ is comeager, as desired.
\end{proof}

\begin{corollary}\label{corollary_filters_meager}
Let $\bS$ be a nice topological pointclass, and assume that $\Det(\bS(\omega^\omega))$ holds. Let $\FF\in\bS(2^\omega)$ be a filter. Then $\FF$ is meager in $2^\omega$.
\end{corollary}
\begin{proof}
Assume, in order to get a contradiction, that $\FF$ is non-meager in $2^\omega$. Then $\FF$ is comeager in $2^\omega$ by Proposition \ref{proposition_finite_modifications_meager_comeager}. Since $c$ is a homeomorphism and $2^\omega$ is a non-meager space, it follows that there exists $x\in\FF$ such that $\omega\setminus x=c(x)\in\FF$, which contradicts the definition of filter.
\end{proof}

\subsection{Homogeneity}\label{subsection_preliminaries_homogeneity}

We have already introduced homogeneity in Definition \ref{definition_homogeneous}. The following notion is both a useful technical tool and an independent object of study (see the introduction to \cite[Chapter 3]{medini_thesis} for a brief survey of this topic).

\begin{definition}\label{definition_strongly_homogeneous}
A space $X$ is \emph{strongly homogeneous} (or \emph{h-homogeneous}) if $X$ is zero-dimensional and every non-empty clopen subspace of $X$ is homeomorphic to~$X$.
\end{definition}

Familiar examples of strongly homogeneous spaces are given by the rationals $\QQQ$, the Cantor set $2^\omega$, and Baire space $\omega^\omega$. This can be seen by using their characterizations (see \cite[Exercise 7.12, Theorem 7.4 and Theorem 7.7]{kechris} respectively). The reason for the modifier ``strongly'' is given by Corollary \ref{corollary_strongly_homogeneous}, which also explains why we included zero-dimensionality in the definition.\footnote{\,To see that the assumption of zero-dimensionality is indispensable, consider the subspace
$$
X=(\RRR\times\{0\})\cup\{(q_n,2^{-n}):n\in\omega\}
$$
of $\RRR\times\RRR$, where $\QQQ=\{q_n:n\in\omega\}$ is an enumeration. Using \cite[Theorem 18]{medini_products}, one can show that every non-empty clopen subspace of $X^\omega$ is homeomorphic to $X^\omega$. On the other hand, it is easy to realize that $X^\omega$ is not homogeneous.}

In fact, even when confined to the zero-dimensional realm, the relationship between homogeneity and strong homogeneity remains rather delicate. The discrete spaces of size at least two and $\omega\times 2^\omega$ are obviously homogeneous but not strongly homogeneous. However, as Proposition \ref{proposition_locally_compact} and Theorem \ref{theorem_characterization_semifilters} will show, it is consistent that these are the only zero-dimensional counterexamples (this was first proved in \cite{carroy_medini_muller_homogeneous}).

The following simple proposition justifies referring to the locally compact case of the classification as the trivial case.

\begin{proposition}\label{proposition_locally_compact}
Let $X$ be a zero-dimensional locally compact space. Then the following conditions are equivalent:
\begin{itemize}
\item $X$ is homogeneous	,
\item $X$ is discrete, $X\approx 2^\omega$, or $X\approx\omega\times 2^\omega$.
\end{itemize}
\end{proposition}
\begin{proof}
Clearly, all the spaces mentioned above are homogeneous. In order to prove the other implication, assume that $X$ is homogeneous but not discrete. Notice that $X$ must be crowded by homogeneity. Since $X$ is locally compact and zero-dimensional, it is possible to partition $X$ into compact clopen sets. Using \cite[Theorem 7.4]{kechris}, it follows that $X\approx 2^\omega$ (if the partition is finite) or $X\approx\omega\times 2^\omega$	(if the partition is infinite).
\end{proof}

The following ingenious result is an important tool for constructing homogeneous spaces, and it first appeared as \cite[Lemma 2.1]{van_mill_homogeneous}. Corollary \ref{corollary_finite_modifications_homogeneous} is essentially the same as \cite[Lemma 2]{medini_van_mill_zdomskyy}.

\begin{theorem}[van Mill]\label{theorem_homogeneity}
Let $X$ be a space, let $x,y\in X$, let $U_0\supseteq U_1\supseteq\cdots$ be clopen subsets of $X$ such that $\{U_n:n\in\omega\}$ is a local base for $X$ at $x$, and let $V_0\supseteq V_1\supseteq\cdots$ be clopen subsets of $X$ such that $\{V_n:n\in\omega\}$ is a local base for $X$ at $y$. Assume that $U_n\approx V_n$ for each $n$. Then there exists a homeomorphism $h:X\longrightarrow X$ such that $h(x)=y$.
\end{theorem}

\begin{corollary}\label{corollary_strongly_homogeneous}
Every strongly homogeneous space is homogeneous.
\end{corollary}

\begin{corollary}[Medini, van Mill, Zdomskyy]\label{corollary_finite_modifications_homogeneous}
Every subspace of $2^\omega$ that is closed under finite modifications is homogeneous.	
\end{corollary}
\begin{proof}
Let $\XX$ be a subspace of $2^\omega$ that is closed under finite modifications. Pick $x,y\in\XX$. We will show that there exists a homeomorphism $h:X\longrightarrow X$ such that $h(x)=y$. Set $U_n=\Ne_{x\re n}\cap\XX$ and $V_n=\Ne_{y\re n}\cap\XX$ for $n\in\omega$. By Theorem \ref{theorem_homogeneity}, it will be enough to show that $U_n\approx V_n$ for each $n$.

So pick $n\in\omega$, and set $F=\{i\in n:x(i)\neq y(i)\}$. Consider the homeomorphism $h_F$ defined in \S\ref{subsection_preliminaries_filters}, and observe that $h_F[\Ne_{x\re n}]=\Ne_{y\re n}$. Using the fact that $\XX$ is closed under finite modifications, one sees that $h_F[\XX]=\XX$. It follows that $h_F[U_n]=V_n$, hence $U_n\approx V_n$.
\end{proof}

\begin{corollary}\label{corollary_semifilter_homogeneous}
Every semifilter is homogeneous.	
\end{corollary}

At this point, it seems fitting to mention the following proposition and question (they first appeared as \cite[Proposition 13.6 and Question 13.5]{medini_semifilters} respectively).

\begin{proposition}\label{proposition_filter_strongly_homogeneous}
Every filter is strongly homogeneous.	
\end{proposition}

\begin{question}\label{question_semifilter_strongly_homogeneous}
Is every semifilter strongly homogeneous?	
\end{question}

The following diagram illustrates the connections between the homogeneity-type properties and the combinatorial properties that we have discussed so far. With the possible exception of the implication mentioned by Question \ref{question_semifilter_strongly_homogeneous}, the implications depicted below (and their obvious consequences) are the only ones that are provable in $\ZF+\DC$ alone. The desired counterexamples are given by Proposition \ref{proposition_counterexample_semifilter}, by the proof of \cite[Theorem 1.2]{carroy_medini_muller_homogeneous} (which was essentially obtained by van Douwen in \cite{van_douwen}), and by \cite[Proposition 8.3]{medini_cdh}.\footnote{\,A few clarifications are in order, so that trivial counterexamples will be avoided. First, we are restricting the attention to zero-dimensional spaces that are not locally compact. Second, by filter (respectively semifilter), we mean homeomorphic to a filter (respectively homeomorphic to a semifilter). Third, regarding the relationship between filters and topological groups, we refer to \S\ref{subsection_filters_vs_groups} for a more nuanced discussion.}

\bigskip

\begin{center}
$
\xymatrix{
& \mathrm{Filter}\ar@/_1pc/[ld] \ar@{->}[d] \ar@/^1pc/[rd] &\\
  \mathrm{Topological\text{ }group} \ar@/_1pc/[rd] & \mathrm{Semifilter} \ar@{->}[d] & \mathrm{Strongly\text{ }homogeneous} \ar@/^1pc/[ld]\\
&\mathrm{Homogeneous}&\\
}
$
\end{center}

\bigskip

Next, we prove an intuitive dichotomy regarding homogeneous spaces, which can be assumed to be folklore (but see \cite[Lemma 3.1]{fitzpatrick_zhou} for a more general result).

\begin{proposition}\label{proposition_homogeneous_dichotomy}
If $X$ is a homogeneous space then $X$ is either meager or Baire.	
\end{proposition}
\begin{proof}
Assume that $X$ is a homogeneous space that is not Baire. Fix a non-empty meager open subset $U$ of $X$. Notice that
$$
\UU=\{h[U]:h:X\longrightarrow X\text{ is a homeomorphism}\}
$$
is an open cover of $X$ by homogeneity, and that $\UU$ consists of meager subsets of $X$. By considering a countable subcover of $\UU$, one sees that $X$ is meager.
\end{proof}

We conclude this subsection with a technical result that will be needed in the proof of Corollary \ref{corollary_h-homogeneously_embedded}. It is a slight variation on \cite[Theorem 2.4]{terada}.

\begin{lemma}[Terada]\label{lemma_terada}
Let $X$ be a zero-dimensional space. Assume that $X$ has a $\pi$-base $\BB$ consisting of clopen sets such that $U\approx X$ for every $U\in\BB$. Then $X$ is strongly homogeneous. 	
\end{lemma}
\begin{proof}
It is easy to realize that if $X$ has an isolated point then $|X|=1$, so assume that $X$ is crowded. If $X$ is compact then $X\approx 2^\omega$ by \cite[Theorem 7.4]{kechris}. Otherwise, proceed as in the proof of \cite[Theorem 2.4]{terada} (see also \cite[Appendix A]{medini_products} or \cite[Appendix B]{medini_thesis}).
\end{proof}

\section{Wadge theory: fundamental notions and results}\label{section_wadge_fundamental}

\subsection{The basics of Wadge theory}\label{subsection_wadge_fundamental_basics}

We begin with the most fundamental definitions and results of Wadge theory. For the proofs, see \cite[\S4]{carroy_medini_muller_constructing}. Assume that a set $Z$ and $\bG\subseteq\PP(Z)$ are given. We will use the notation $\bGc=\{Z\setminus A:A\in\bG\}$ for the \emph{dual} of $\bG$, and say that $\bG$ is \emph{selfdual} if $\bG=\bGc$. Also set $\Delta(\bG)=\bG\cap\bGc$.

\begin{definition}[Wadge]\label{definition_wadge_reducibility}
Let $Z$ be a space, and let $A,B\subseteq Z$. We will write $A\leq B$ if there exists a continuous function $f:Z\longrightarrow Z$ such that $A=f^{-1}[B]$. In this case, we will say that $A$ is \emph{Wadge-reducible} to $B$, and that $f$ \emph{witnesses} the reduction.\footnote{\,Wadge-reduction is usually denoted by $\leq_\mathsf{W}$, which allows to distinguish it from other types of reduction (such as Lipschitz-reduction). Since we will not consider any other type of reduction in this article, we decided to simplify the notation.} We will write $A<B$ if $A\leq B$ and $B\not\leq A$. We will say that $A$ is \emph{selfdual} if $A\leq Z\setminus A$. We will use the notation
$$
A\wc=\{B\subseteq Z:B\leq A\}.
$$
We will say that $\bG\subseteq\PP(Z)$ is \emph{continuously closed} if $C\wc\subseteq\bG$ for every $C\subseteq Z$.
\end{definition}

We have already introduced Wadge classes in Definition \ref{definition_wadge_class}. Notice that, using the notation of Definition \ref{definition_wadge_reducibility}, the Wadge classes in a space $Z$ are precisely the sets of the form $A\wc$ for some $A\subseteq Z$.

\begin{definition}\label{definition_non-selfdual}
Given a space $Z$, define 
$$
\NSD(Z)=\{\bG:\bG\textrm{ is a non-selfdual Wadge class in }Z\}.
$$
Also set $\NSDS(Z)=\{\bG\in\NSD(Z):\bG\subseteq\bS(Z)\}$ whenever $\bS$ is a topological pointclass.
\end{definition}

\begin{lemma}[Wadge]\label{lemma_wadge}
Let $\bS$ be a nice topological pointclass, and assume that $\Det(\bS(\omega^\omega))$ holds. Let $Z$ be a zero-dimensional Polish space, and let $A,B\in\bS(Z)$. Then either $A\leq B$ or $Z\setminus B\leq A$.
\end{lemma}

\begin{theorem}[Martin, Monk]\label{theorem_well-founded}
Let $\bS$ be a nice topological pointclass, and assume that $\Det(\bS(\omega^\omega))$ holds. Let $Z$ be a zero-dimensional Polish space. Then the relation $\leq$ on $\bS(Z)$ is well-founded.
\end{theorem}

Next, we state an elementary result, which shows that clopen sets are ``neutral sets'' with respect to Wadge-reduction (the straightforward proof is left to the reader). In \S\ref{section_wadge_closure}, we will discuss much more sophisticated closure properties.

\begin{lemma}\label{lemma_closure_clopen}
Let $Z$ be a space, and let $\bG$ be a Wadge class in $Z$.
\begin{itemize}
\item Assume that $\bG\neq\{Z\}$. Then $A\cap U\in\bG$ whenever $A\in\bG$ and $U\in\bD^0_1(Z)$.
\item Assume that $\bG\neq\{\varnothing\}$. Then $A\cup U\in\bG$ whenever $A\in\bG$ and $U\in\bD^0_1(Z)$.
\end{itemize}
\end{lemma}

We conclude this subsection with the complete analysis of the non-selfdual Wadge classes below $\bD^0_2$ (see \cite[Theorem 11.2]{carroy_medini_muller_constructing}).

\begin{theorem}\label{theorem_complete_analysis_delta^0_2}
Let $Z$ be an uncountable zero-dimensional Polish space, and let $\bG\subseteq\bD^0_2(Z)$ be such that $\bG\neq\{\varnothing\}$ and $\bG\neq\{Z\}$. Then the following conditions are equivalent:
\begin{itemize}
\item $\bG\in\NSD(Z)$,
\item There exists $1\leq\eta<\omega_1$ such that $\bG=\Diff_\eta(\bS^0_1(Z))$ or $\bG=\widecheck{\Diff}_\eta(\bS^0_1(Z))$.
\end{itemize}
\end{theorem}

\subsection{Relativization}\label{subsection_wadge_fundamental_relativization}

When one tries to give a systematic exposition of Wadge theory, it soon becomes apparent that it would be very useful to be able to say when $A$ and $B$ belong to ``the same'' Wadge class, even when $A\subseteq Z$ and $B\subseteq W$ for distinct ambient spaces $Z$ and $W$. In the non-selfdual case, this problem can be solved by using Wadge classes in $\omega^\omega$ to parametrize Wadge classes in arbitrary zero-dimensional Polish spaces. This is essentially due to Louveau and Saint Raymond (see \cite[Theorem 4.2]{louveau_saint_raymond_strength}), but we will follow the more systematic treatment given in \cite{carroy_medini_muller_constructing}. In particular, the following definition and three results establish the foundations of this method (see \cite[\S6]{carroy_medini_muller_constructing} for the proofs).

\begin{definition}[Louveau, Saint Raymond]\label{definition_relativization}
Let $Z$ be a space, and let $\bG\subseteq\PP(\omega^\omega)$. We will use the notation
$$
\bG(Z)=\{A\subseteq Z:f^{-1}[A]\in\bG\text{ for every continuous }f:\omega^\omega\longrightarrow Z\}.
$$	
\end{definition}

\begin{lemma}\label{lemma_relativization_exists_unique}
Let $\bS$ be a nice topological pointclass, and assume that $\Det(\bS(\omega^\omega))$ holds. Let $Z$ be a zero-dimensional Polish space, and let $\bL\in\NSDS(Z)$. Then there exists a unique $\bG\in\NSDS(\omega^\omega)$ such that $\bG(Z)=\bL$.
\end{lemma}

\begin{lemma}\label{lemma_relativization_basic}
Let $Z$ and $W$ be spaces, and let $\bG\subseteq\PP(\omega^\omega)$.
\begin{itemize}
\item If $f:Z\longrightarrow W$ is continuous and $B\in\bG(W)$ then $f^{-1}[B]\in\bG(Z)$.
\item If $h:Z\longrightarrow W$ is a homeomorphism then $A\in\bG(Z)$ iff $h[A]\in\bG(W)$.
\item $\widecheck{\bG(Z)}=\bGc(Z)$.
\item If $\bG$ is continuously closed then $\bG(\omega^\omega)=\bG$.
\end{itemize}
\end{lemma}

\begin{lemma}\label{lemma_relativization_subspace}
Let $\bS$ be a nice topological pointclass, and assume that $\Det(\bS(\omega^\omega))$ holds. Let $Z$ and $W$ be zero-dimensional Borel spaces such that $W\subseteq Z$, and let $\bG\in\NSDS(\omega^\omega)$. Then
$$
\bG(W)=\{A\cap W:A\in\bG(Z)\}.
$$
\end{lemma}

Until now, we have never assumed the uncountability of the ambient spaces. As the following two results show, the situation gets particularly pleasant when this assumption is satisfied (see \cite[\S7 and Lemma 20.2]{carroy_medini_muller_constructing} for the proofs).

\begin{theorem}\label{theorem_relativization_uncountable}
Let $\bS$ be a nice topological pointclass, and assume that $\Det(\bS(\omega^\omega))$ holds. Let $Z$ be an uncountable zero-dimensional Polish space. Then
$$
\NSDS(Z)=\{\bG(Z):\bG\in\NSDS(\omega^\omega)\}.
$$
\end{theorem}

\begin{theorem}\label{theorem_order_isomorphism}
Let $\bS$ be a nice topological pointclass, and assume that $\Det(\bS(\omega^\omega))$ holds. Let $Z$ and $W$ be uncountable zero-dimensional Borel spaces, and let $\bG,\bL\in\NSDS(\omega^\omega)$. Then
$$
\bG(Z)\subseteq\bL(Z)\text{ iff }\bG(W)\subseteq\bL(W).
$$
\end{theorem}

\subsection{Level}\label{subsection_wadge_fundamental_level}

We begin with some preliminaries concerning the following useful notion.\footnote{\,According to Wadge, this notion had already been identified by ``the classical descriptive set theorists'' (see \cite[page 187]{wadge_cabal}).} Partitioned unions will also play a major role in \S\ref{section_wadge_type}.

\begin{definition}\label{definition_pu}
Let $Z$ be a space, let $\bG\subseteq\PP(Z)$, and let $\xi<\omega_1$. Define $\PU_\xi(\bG)$ to be the collection of all sets of the form
$$
\bigcup_{n\in\omega}(A_n\cap V_n),
$$
where each $A_n\in\bG$, the $V_n\in\bD_{1+\xi}^0(Z)$ are such that $V_m\cap V_n=\varnothing$ whenever $m\neq n$, and $\bigcup_{n\in\omega}V_n=Z$. A set in this form is called a \emph{partitioned union} of sets in $\bG$.
\end{definition}

The following proposition, whose straightforward proof is left to the reader, collects the most basic facts about partitioned unions.

\begin{proposition}\label{proposition_pu_basic}
Let $Z$ be a space, let $\bG\subseteq\PP(Z)$, and let $\xi<\omega_1$. Then
\begin{itemize}
\item If $\bG$ is continuously closed then $\PU_\xi(\bG)$ is continuously closed,	
\item $\bG\subseteq\PU_\eta(\bG)\subseteq\PU_\xi(\bG)$ whenever $\eta\leq\xi$,
\item $\PU_0(\bG)=\bG$ whenever $\bG$ is a Wadge class in $Z$,
\item $\widecheck{\PU}_\xi(\bG)=\PU_\xi(\bGc)$,
\item $\PU_\xi(\PU_\xi(\bG))=\PU_\xi(\bG)$.
\end{itemize}
\end{proposition}

The following notion was essentially introduced in \cite{louveau_article}. Our approach will be closer to that of \cite{louveau_saint_raymond_level} and \cite{louveau_book}, as developed in \cite{carroy_medini_muller_constructing}.

\begin{definition}[Louveau and Saint Raymond]\label{definition_level}
Let $Z$ be a space, let $\bG\subseteq\PP(Z)$, and let $\xi<\omega_1$. Define
\begin{itemize}
\item $\ell(\bG)\geq\xi$ if $\PU_\xi(\bG)=\bG$,
\item $\ell(\bG)=\xi$ if $\ell(\bG)\geq\xi$ and $\ell(\bG)\not\geq\xi+1$,
\item $\ell(\bG)=\omega_1$ if $\ell(\bG)\geq\eta$ for every $\eta<\omega_1$.
\end{itemize}
We refer to $\ell(\bG)$ as the \emph{level} of $\bG$.
\end{definition}

As immediate consequences of Proposition \ref{proposition_pu_basic}, observe that $\ell(\bG)\geq 0$ for every Wadge class $\bG$, and that $\ell(\bG)=\ell(\bGc)$ for an arbitrary $\bG$. The following statements provide concrete examples, where $Z$ is an uncountable Polish space (the last two can be proved using \cite[Theorem 22.4 and Proposition 37.1]{kechris} respectively):
\begin{itemize}
\item $\ell(\{\varnothing\})=\ell(\{Z\})=\omega_1$,
\item $\ell(\bS^0_{1+\xi}(Z))=\ell(\bP^0_{1+\xi}(Z))=\xi$ whenever $\xi<\omega_1$,
\item $\ell(\bS^1_n(Z))=\ell(\bP^1_n(Z))=\omega_1$ whenever $1\leq n<\omega$.
\end{itemize}

The following result (whose straightforward proof is left to the reader) gives the first hint that classes of higher level tend to have better closure properties.

\begin{lemma}\label{lemma_closure_level}
Let $Z$ be a space, let $\bG\subseteq\PP(Z)$ be such that $\varnothing\in\bG$, and let $\xi<\omega_1$. Assume that $\ell(\bG)\geq\xi$. Then $A\cap V\in\bG$ whenever $A\in\bG$ and $V\in\bD^0_{1+\xi}(Z)$.
\end{lemma}

Next, we will show that the level is well-behaved with respect to relativization (see \cite[Corollary 16.2]{carroy_medini_muller_constructing}).

\begin{theorem}\label{theorem_level_relativization}
Let $\bS$ be a nice topological pointclass, and assume that $\Det(\bS(\omega^\omega))$ holds. Let $Z$ and $W$ be uncountable zero-dimensional Polish spaces, let $\xi<\omega_1$, and let $\bG\in\NSDS(\omega^\omega)$. Then
$$
\ell(\bG(Z))\geq\xi\text{ iff }\ell(\bG(W))\geq\xi.
$$
\end{theorem}

At this point, it is not clear whether every non-selfdual Wadge class $\bG$ even \emph{has} a level.\footnote{\,For example, it is conceivable that $\PU_n(\bG)=\bG$ for all $n<\omega$, while $\PU_\omega(\bG)\neq\bG$.} That this is actually the case is the content of the next theorem (see \cite[Corollary 17.2]{carroy_medini_muller_constructing}), whose proof heavily relies on methods from \cite{louveau_book}.

\begin{theorem}[Carroy, Medini, M\"{u}ller]\label{theorem_every_class_has_a_level}
Let $\bS$ be a nice topological pointclass, and assume that $\Det(\bS(\omega^\omega))$ holds. Let $Z$ be a zero-dimensional Polish space, and let $\bG\in\NSDS(Z)$. Then there exists $\xi\leq\omega_1$ such that $\ell(\bG)=\xi$.
\end{theorem}

\subsection{Expansions}\label{subsection_wadge_fundamental_expansions}

The following notion can be traced back to \cite[Chapter IV]{wadge_thesis}, and it is inspired by work of Kuratowski (see \cite[\S20 and \S21]{wadge_cabal}). Once again, our approach will be closer to that of \cite{louveau_book}, as developed in \cite{carroy_medini_muller_constructing}.

\begin{definition}\label{definition_expansion}
Let $Z$ be a space, and let $\xi<\omega_1$. Given $\bG\subseteq\PP(Z)$, define
$$
\bG^{(\xi)}=\{f^{-1}[A]:A\in\bG\text{ and }f:Z\longrightarrow Z\text{ is $\bS^0_{1+\xi}$-measurable}\}.
$$
We will refer to $\bG^{(\xi)}$ as an \emph{expansion} of $\bG$.
\end{definition}

The following proposition (whose straightforward proof is left to the reader) collects some basic facts about expansions.

\begin{proposition}\label{proposition_expansion_basic}
Let $Z$ be a space, let $\bG\subseteq\PP(Z)$, and let $\xi<\omega_1$.
\begin{itemize}
\item $\bG^{(\xi)}$ is continuously closed.
\item $\bG\subseteq\bG^{(\eta)}\subseteq\bG^{(\xi)}$ whenever $\eta\leq\xi$.
\item $\bG^{(0)}=\bG$ whenever $\bG$ is continuously closed.
\item $\widecheck{\bG^{(\xi)}}=\bGc^{(\xi)}$.
\end{itemize}
\end{proposition}

The next result, known as the Expansion Theorem, clarifies the crucial connection between level and expansions. It is an immediate consequence of \cite[Theorems 16.1 and 22.2]{carroy_medini_muller_constructing}.

\begin{theorem}\label{theorem_expansion}
Let $\bS$ be a nice topological pointclass, and assume that $\Det(\bS(\omega^\omega))$ holds. Let $Z$ be an uncountable zero-dimensional Polish space, and let $\xi<\omega_1$. Then the following conditions are equivalent:
\begin{itemize}
\item $\bG\in\NSDS(Z)$ and $\ell(\bG)\geq\xi$,
\item $\bG =\bL^{(\xi)}$ for some $\bL\in\NSDS(Z)$.
\end{itemize}
\end{theorem}

\begin{corollary}\label{corollary_expansion_order_isomorphism}
Let $\bS$ be a nice topological pointclass, and assume that $\Det(\bS(\omega^\omega))$ holds. Let $Z$ be an uncountable zero-dimensional space, let $\bG,\bL\in\NSDS(Z)$, and let $\xi<\omega_1$. Then
$$
\bG\subseteq\bL\text{ iff }\bG^{(\xi)}\subseteq\bL^{(\xi)}.
$$
\end{corollary}
\begin{proof}
The left-to-right implication is clear by the definition of expansion. In order to prove the other implication, assume that $\bG\nsubseteq\bL$. Then $\bLc\subseteq\bG$ by Lemma  \ref{lemma_wadge}, which implies $\bLc^{(\xi)}\subseteq\bG^{(\xi)}$. So, if we had $\bG^{(\xi)}\subseteq\bL^{(\xi)}$, it would follow that $\bL^{(\xi)}$ is self-dual, contradicting Theorem \ref{theorem_expansion}. Therefore $\bG^{(\xi)}\nsubseteq\bL^{(\xi)}$, as desired.
\end{proof}

The following result shows that expansions ``add up'' when they are composed (see \cite[Theorem 18.2]{carroy_medini_muller_constructing}). As a corollary, we will be able to determine the exact level of an expansion.

\begin{theorem}\label{theorem_expansion_composition}
Let $\bS$ be a nice topological pointclass, and assume that $\Det(\bS(\omega^\omega))$ holds. Let $Z$ be an uncountable zero-dimensional Polish space, and let $\xi,\eta<\omega_1$. Then
$$
(\bG^{(\eta)})^{(\xi)}=\bG^{(\xi+\eta)}
$$
whenever $\bG\in\NSDS(Z)$.
\end{theorem}

\begin{corollary}\label{corollary_expansion_level}
Let $\bS$ be a nice topological pointclass, and assume that $\Det(\bS(\omega^\omega))$ holds. Let $Z$ be an uncountable zero-dimensional Polish space, let $\xi<\omega_1$, and let $\bG\in\NSDS(Z)$. Assume that $\ell(\bG)=\eta<\omega_1$. Then
$$
\ell(\bG^{(\xi)})=\xi+\eta.
$$	
\end{corollary}
\begin{proof}
By Theorem \ref{theorem_expansion}, we can fix $\bL\in\NSDS(Z)$ such that $\bL^{(\eta)}=\bG$. Notice that $\bL^{(\xi+\eta)}=\bG^{(\xi)}$ by Theorem \ref{theorem_expansion_composition}, hence $\ell(\bG^{(\xi)})\geq\xi+\eta$ by Theorem \ref{theorem_expansion}. Now assume, in order to get a contradiction, that $\ell(\bG^{(\xi)})\geq\xi+\eta+1$. By Theorem \ref{theorem_expansion}, we can fix $\bL\in\NSDS(Z)$ such that $\bL^{(\xi+\eta+1)}=\bG^{(\xi)}$. Notice that
$$
\bG^{(\xi)}=\bL^{(\xi+\eta+1)}=(\bL^{(\eta+1)})^{(\xi)},
$$
where the second equality holds by Theorem \ref{theorem_expansion_composition}. It follows from Corollary \ref{corollary_expansion_order_isomorphism} that $\bG=\bL^{(\eta+1)}$, which contradicts the fact that $\ell(\bG)\ngeq\eta+1$ by Theorem \ref{theorem_expansion}.
\end{proof}

We conclude this subsection with three technical results. Lemmas \ref{lemma_expansion_relativization_measurable_function} and \ref{lemma_expansion_relativization_move_xi} show that expansions interact in the expected way with the machinery of relativization (see \cite[Lemmas 14.1 and 14.2]{carroy_medini_muller_constructing} respectively). Lemma \ref{lemma_expansion_bijection} is a variation on the theme of Kuratowski's Transfer Theorem (see \cite[\S12]{carroy_medini_muller_constructing}), and it can be proved like \cite[Lemma 14.3]{carroy_medini_muller_constructing}.

\begin{lemma}\label{lemma_expansion_relativization_measurable_function}
Let $\bS$ be a nice topological pointclass, and assume that $\Det(\bS(\omega^\omega))$ holds. Let $Z$ and $W$ be zero-dimensional Polish spaces, let $\xi<\omega_1$, and let $\bG\in\NSDS(\omega^\omega)$. If $f:Z\longrightarrow W$ is $\bS^0_{1+\xi}$-measurable then $f^{-1}[A]\in\bG^{(\xi)}(Z)$ for every $A\in\bG(W)$.
\end{lemma}

\begin{lemma}\label{lemma_expansion_relativization_move_xi}
Let $\bS$ be a nice topological pointclass, and assume that $\Det(\bS(\omega^\omega))$ holds. Let $Z$ be an uncountable zero-dimensional Polish space, let $\xi<\omega_1$, and let $\bG\in\NSDS(\omega^\omega)$. Then $\bG^{(\xi)}(Z)=\bG(Z)^{(\xi)}$.
\end{lemma}

\begin{lemma}\label{lemma_expansion_bijection}
Let $Z$ be a zero-dimensional Polish space, let $\xi<\omega_1$, let $\bG_n\subseteq\PP(\omega^\omega)$ for $n\in\omega$, and let $\Aa_n\subseteq\bG_n(Z)^{(\xi)}$ for $n\in\omega$ be countable. Then there exist a zero-dimensional Polish space $W$ and a $\bS^0_{1+\xi}$-measurable bijection $f:Z\longrightarrow W$ such that $f[A]\in\bG_n(W)$ for every $n\in\omega$ and $A\in\Aa_n$.
\end{lemma}

\subsection{Hausdorff operations and universal sets}\label{subsection_wadge_fundamental_hausdorff}

The main purpose of this subsection is to obtain Theorems \ref{theorem_wadge_implies_universal} and \ref{theorem_universal_implies_wadge}, which explain the relationship between universal sets and non-selfdual Wadge classes. This results will be useful in the proofs of Corollary \ref{corollary_new_from_old} and Lemma \ref{lemma_universal_pair}. One of the fundamental tools to prove results of this kind is given by the notion of Hausdorff operation.

\begin{definition}
Let $D\subseteq\PP(\omega)$, and let $Z$ be a set. Define
$$
\HH_D(A_0,A_1,\ldots)=\{x\in Z:\{n\in\omega:x\in A_n\}\in D\}
$$
whenever $A_0,A_1,\ldots\subseteq Z$. Functions of this form are called \emph{Hausdorff operations} (or \emph{$\omega$-ary Boolean operations}).
\end{definition}

\begin{definition}
Let $D\subseteq\PP(\omega)$. Define
$$
\bG_D(Z)=\{\HH_D(A_0,A_1,\ldots):A_n\in\bS^0_1(Z)\text{ for }n\in\omega\}
$$
for every space $Z$. We will refer to collections in this form as \emph{Hausdorff classes}.
\end{definition}

A basic fact about Hausdorff operations is that they include all operations obtained by combinining unions, intersections and complements (see \cite[Proposition 8.2]{carroy_medini_muller_constructing}). Furthermore, the composition of Hausdorff operations is still a Hausdorff operation (see \cite[Proposition 8.3]{carroy_medini_muller_constructing}). As an application, one can use induction to prove that each $\bS^0_\xi(Z)$ is a Hausdorff class. In fact, as the following result shows, something much more general holds (see \cite[Theorem 22.2]{carroy_medini_muller_constructing}). The case $Z=\omega^\omega$ and $\bS=\PP$ is originally due to Van Wesep (see \cite[Theorem 5.3.1]{van_wesep_thesis}).

\begin{theorem}\label{theorem_van_wesep_hausdorff}
Let $\bS$ be a nice topological pointclass, and assume that $\Det(\bS(\omega^\omega))$ holds. Let $Z$ be an uncountable zero-dimensional Polish space. Then the following conditions are equivalent:
\begin{itemize}
\item $\bG\in\NSDS(Z)$,
\item $\bG=\bG_D(Z)$ for some $D\subseteq\PP(\omega)$ and $\bG\subseteq\bS(Z)$.
\end{itemize}
\end{theorem}

Next, we come to the classical notion of a universal set (see for example \cite[Definition 22.2]{kechris}).

\begin{definition}\label{definition_universal}
Let $\bP$ be a topological pointclass, and let $W$ and $Z$ be spaces. We will say that $S$ is a \emph{$W$-universal set} for $\bP(Z)$ if the following conditions hold:
\begin{itemize}
\item $S\in\bP(W\times Z)$,
\item $\bP(Z)=\{S_y:y\in W\}$,
\end{itemize}
where we denote by $S_y=\{x\in Z:(y,x)\in S\}$ the \emph{vertical section} above $y\in W$.
\end{definition}

The following special kind of topological pointclass is particularly useful when dealing with universal sets.

\begin{definition}\label{definition_relativizable}
We will say that a function $\bP$ is a \emph{relativizable pointclass} if the following conditions hold:
\begin{itemize}
\item $\bP$ is a topological pointclass,
\item $\bP(W)=\{A\cap W:A\in\bP(Z)\}$ whenever $Z$ and $W$ are zero-dimensional Borel spaces such that $W\subseteq Z$.
\end{itemize}
\end{definition}

In the present context, the relevant examples of relativizable pointclasses are obtained as follows:
\begin{itemize}
\item Given $D\subseteq\PP(\omega)$, set $\bP(Z)=\bG_D(Z)$ for every space $Z$,
\item Given $\bG\in\NSD(\omega^\omega)$, set $\bP(Z)=\bG(Z)$ for every space $Z$.
\end{itemize}
In the first case, this is the content of \cite[Lemma 9.5]{carroy_medini_muller_constructing}. In the second case (under suitable determinacy assumptions), this follows from Lemmas \ref{lemma_relativization_basic} and \ref{lemma_relativization_subspace}. In fact, the reason for the restriction ``zero-dimensional Borel'' in Definition \ref{definition_relativizable} is that Lemma \ref{lemma_relativization_subspace} only applies to such spaces.

Next, we will see that non-selfdual Wadge classes have universal sets. We will need the following auxiliary result (see \cite[Proposition 10.2]{carroy_medini_muller_constructing} for the proof).

\begin{proposition}\label{proposition_hausdorff_cantor-universal}
Let $D\subseteq\PP(\omega)$, and let $Z$ be a space. Then there exists a $2^\omega$-universal set for $\bG_D(Z)$.
\end{proposition}

\begin{theorem}\label{theorem_wadge_implies_universal}
Let $\bS$ be a nice topological pointclass, and assume that $\Det(\bS(\omega^\omega))$ holds. Let $\bG\in\NSDS(\omega^\omega)$, and let $Z$ be a zero-dimensional Borel space. Then $\bG(Z)$ has a $2^\omega$-universal set.
\end{theorem}
\begin{proof}
By Theorem \ref{theorem_van_wesep_hausdorff}, we can fix $D\subseteq\PP(\omega)$ such that $\bG=\bG_D(\omega^\omega)$. Observe that if $W$ is a zero-dimensional Borel space, which we assume without loss of generality to be a subspace of $\omega^\omega$, then 
$$
\bG(W)=\{A\cap W:A\in\bG(\omega^\omega)\}=\{A\cap W:A\in\bG_D(\omega^\omega)\}=\bG_D(W)
$$
by relativizability. In other words, the same $D$ works for every zero-dimensional Borel space. To conclude the proof, simply apply Proposition \ref{proposition_hausdorff_cantor-universal}.
\end{proof}

Finally, we will obtain a converse to Theorem \ref{theorem_wadge_implies_universal}. We will need the following two auxiliary results, which correspond to \cite[Corollary 10.3 and Lemma 10.4]{carroy_medini_muller_constructing}, except that they are formulated in the more general framework described above.

\begin{lemma}\label{lemma_cantor-universal_implies_self-universal}
Let $\bP$ be a relativizable topological pointclass, and let $Z$ be an uncountable zero-dimensional Borel space. If $\bP(Z)$ has a $2^\omega$-universal set then $\bP(Z)$ has a $Z$-universal set.
\end{lemma}

\begin{lemma}\label{lemma_self-universal_implies_non-selfdual}
Let $\bP$ be a topological pointclass, and let $Z$ be a space. If $\bP(Z)$ has a $Z$-universal set then $\bP(Z)$ is non-selfdual.
\end{lemma}

\begin{theorem}\label{theorem_universal_implies_wadge}
Let $\bP$ be a relativizable pointclass, and let $Z$ be an uncountable zero-dimensional Borel space. If $\bP(Z)$ has a $2^\omega$-universal set then $\bP(Z)\in\NSD(Z)$.
\end{theorem}
\begin{proof}
Assume that $\bP(Z)$ has a $2^\omega$-universal set. The fact that $\bP(Z)$ is non-selfdual follows from Lemmas \ref{lemma_cantor-universal_implies_self-universal} and \ref{lemma_self-universal_implies_non-selfdual}. To see that $\bP(Z)$ is a Wadge class, proceed as in the proof of \cite[Theorem 7.5]{carroy_medini_muller_homogeneous}.
\end{proof}

\subsection{New Wadge classes from old}\label{subsection_wadge_fundamental_new_from_old}

Many examples of Wadge classes are obtained by combining elements of known classes using a fixed set-theoretic operation (consider for example the difference hierarchy). In fact, combining sets that belong to a given sequence of non-selfdual Wadge classes according to a given Hausdorff operation always yields a non-selfdual Wadge class (this follows from Corollary \ref{corollary_new_from_old} by ignoring the sets $U_{2k+1}$). However, the operation of separated unions (to be introduced in Definition \ref{definition_su}) will require in addition that certain sets are pairwise disjoint. This is why, in this subsection, the statements of our results look more complicated than one might expect.

We begin by showing that combining relativizable pointclasses according to a fixed Hausdorff operation preserves the property of having a universal set. Notice that Theorem \ref{theorem_new_from_old} does not require any determinacy assumptions. The desired result on combining Wadge class will then follow easily, thanks to Theorems \ref{theorem_wadge_implies_universal} and \ref{theorem_universal_implies_wadge}.

\begin{theorem}\label{theorem_new_from_old}
Let $D\subseteq\PP(\omega)$, let $\bP_n$ for $n\in\omega$ be relativizable topological pointclasses, and let $1\leq\xi<\omega_1$. Assume that each $\bP_n$ has a $2^\omega$-universal set. Set
\begin{multline}
\bP(Z)=\{\HH_D(A_0,U_1,A_2,U_3\ldots):A_{2k}\in\bP_k(Z)\text{ for every }k\in\omega\text{,}\nonumber\\
U_{2k+1}\in\bS^0_\xi(Z)\text{ for every }k\in\omega\text{, and }U_{2i+1}\cap U_{2j+1}\text{ whenever }i\neq j\}
\end{multline}
for every space $Z$. Then $\bP$ is a relativizable topological pointclass and $\bP(Z)$ has a $2^\omega$-universal set for every zero-dimensional space $Z$.
\end{theorem}
\begin{proof}
The fact that $\bP$ is a relativizable pointclass is straightforward to verify using the assumption that each $\bP_n$ is a relativizable pointclass and Theorem \ref{theorem_reduction}. Pick a zero-dimensional space $Z$. Fix a homeomorphism $h:2^\omega\longrightarrow (2^\omega)^{\omega}$. Set $(y)_n=\pi_n(h(y))$ for $y\in 2^\omega$ and $n\in\omega$, where $\pi_n:(2^\omega)^{\omega}\longrightarrow 2^\omega$ denotes the projection on the $n$-th coordinate. By assumption, for every $k\in\omega$ we can fix a $2^\omega$-universal set $S_{2k}$ for $\bP_k(Z)$. Define
$$
S'_{2k}=\{(y,x)\in 2^\omega\times Z:((y)_{2k},x)\in S_{2k}\}
$$
for $k\in\omega$, and observe that each $S'_{2k}\in\bP_k(2^\omega\times Z)$.

By \cite[Theorem 22.3]{kechris}, we can fix a $2^\omega$-universal set $U$ for $\bS^0_\xi(Z)$. Define
$$
U_{2k+1}=\{(y,x)\in 2^\omega\times Z:((y)_{2k+1},x)\in U\}
$$
for $k\in\omega$, and observe that each $U_{2k+1}\in\bS^0_\xi(2^\omega\times Z)$. By Theorem \ref{theorem_reduction}, there exist $U'_{2k+1}\in\bS^0_\xi(2^\omega\times Z)$ for $k\in\omega$ such that the following conditions are satisfied:
\begin{itemize}
\item $U'_{2k+1}\subseteq U_{2k+1}$ for each $k$,
\item $U'_{2j+1}\cap U'_{2k+1}=\varnothing$ whenever $j\neq k$,
\item $\bigcup_{k\in\omega}U'_{2k+1}=\bigcup_{k\in\omega}U_{2k+1}$.
\end{itemize}
Notice that, as $y\in 2^\omega$ varies, the sequence of vertical sections
$$
((U'_{2k+1})_y:k\in\omega)
$$
will yield precisely the sequences consisting of pairwise disjoint elements of $\bS^0_\xi(Z)$. Furthermore, changing the values of $(y)_n$ for even $n$ will not affect the above sequence.

Finally, define
$$
T=\HH_D(S'_0,U'_0,S'_1,U'_1,\ldots).
$$
Using the properties that we just mentioned, it is straightforward to check that $T$ is a $2^\omega$-universal set for $\bP(Z)$.
\end{proof}

\begin{corollary}\label{corollary_new_from_old}
Let $\bS$ be a nice topological pointclass, and assume that $\Det(\bS(\omega^\omega))$ holds. Let $D\subseteq\PP(\omega)$, let $Z$ be an uncountable zero-dimensional Polish space, and let $1\leq\xi<\omega_1$. If $\bG_n\in\NSDS(Z)$ for $n\in\omega$ and
\begin{multline}
\bG=\{\HH_D(A_0,U_1,A_2,U_3\ldots):A_{2k}\in\bG_k\text{ for every }k\in\omega\text{,}\nonumber\\
U_{2k+1}\in\bS^0_\xi(Z)\text{ for every }k\in\omega\text{, and }U_{2i+1}\cap U_{2j+1}\text{ whenever }i\neq j\}
\end{multline}
then $\bG\in\NSD(Z)$.
\end{corollary}
\begin{proof}
Pick $\bG_n(Z)\in\NSDS(Z)$ for $n\in\omega$, where each $\bG_n\in\NSDS(\omega^\omega)$. Define relativizable pointclasses $\bP_n$ for $n\in\omega$ by setting $\bP_n(W)=\bG_n(W)$ for every space~$W$, and observe that each $\bP_n(Z)$ has a $2^\omega$-universal set by Theorem \ref{theorem_wadge_implies_universal}. It follows that $\bP(Z)$ has a $2^\omega$-universal set, where $\bP$ is defined as in the statement of Theorem \ref{theorem_new_from_old}. Since $\bG=\bP(Z)$, where $\bG$ is defined as in the statement of this corollary, the desired conclusion follows from Theorem \ref{theorem_universal_implies_wadge}.
\end{proof}

\subsection{Clarifying level and expansions}\label{subsection_wadge_fundamental_clarifying}

In this subsection, we will show how to view level and expansions through the lens of Hausdorff operations. While the only result from this section that will be needed in the rest of the article is Proposition \ref{proposition_expansion_differences}, the alternative descriptions given here are helpful in clarifying these two important concepts.

\begin{definition}\label{definition_hausdorff_expansions}
Let $D\subseteq\PP(\omega)$, and let $\xi<\omega_1$. Define
$$
\bG^{(\xi)}_D(Z)=\{\HH_D(A_0,A_1,\ldots):A_n\in\bS^0_{1+\xi}(Z)\text{ for }n\in\omega\}
$$
for every space $Z$.
\end{definition}

The following result (see \cite[Lemma 13.9]{carroy_medini_muller_constructing}) shows that Definition \ref{definition_hausdorff_expansions} is the counterpart of expansions in the present context.

\begin{lemma}\label{lemma_hausdorff_expansions}
Let $D\subseteq\PP(\omega)$, let $Z$ be an uncountable zero-dimensional Polish space, and let $\xi<\omega_1$. Then $\bG_D(Z)^{(\xi)}=\bG^{(\xi)}_D(Z)$.
\end{lemma}

As a first application, we will give some concrete examples of expansions. This will be useful in \S\ref{subsection_wadge_closure_good}.

\begin{proposition}\label{proposition_expansion_differences}
Let $Z$ be an uncountable zero-dimensional Polish space, let $1\leq\eta<\omega_1$, and let $\xi<\omega_1$. Set $\bG=\Diff_\eta(\bS^0_1(Z))$. Then $\bG^{(\xi)}=\Diff_\eta(\bS^0_{1+\xi}(Z))$.
\end{proposition}
\begin{proof}
Simply apply Lemma \ref{lemma_hausdorff_expansions} for the appropriate choice of $D\subseteq\PP(\omega)$.
\end{proof}

We conclude this subsection with a characterization of the level that is perhaps more intuitive than the official Definition \ref{definition_expansion}. The Borel version of Theorem \ref{theorem_level_sup_hausdorff} originally appeared as part of \cite[Th\'{e}or\`{e}me 8]{louveau_saint_raymond_level}.

\begin{theorem}\label{theorem_level_sup_hausdorff}
Let $\bS$ be a nice topological pointclass, and assume that $\Det(\bS(\omega^\omega))$ holds. Let $Z$ be an uncountable zero-dimensional Polish space, and let $\bG\in\NSDS(Z)$. Then
$$
\ell(\bG)=\supr\{\xi<\omega_1:\bG=\bG^{(\xi)}_D(Z)\text{ for some }D\subseteq\PP(\omega)\}.
$$
\end{theorem}
\begin{proof}
Set $I=\{\xi<\omega_1:\bG=\bG^{(\xi)}_D(Z)\text{ for some }D\subseteq\PP(\omega)\}$. First assume that $\ell(\bG)=\xi<\omega_1$. Let $\xi'\in I$. By Theorem \ref{theorem_expansion}, Theorem \ref{theorem_van_wesep_hausdorff} and Lemma \ref{lemma_hausdorff_expansions}, we must have $\xi\geq\xi'$. Therefore, the inequality $\geq$ holds. On the other hand, a similar argument shows that $\xi\in I$, hence the inequality $\leq$ holds. The case $\ell(\bG)=\omega_1$ is analogous, but slightly easier.
\end{proof}

\subsection{The separation property and the ``pair of socks'' problem}\label{subsection_wadge_fundamental_separation}

Much like the socks in a pair, a non-selfdual Wadge class $\bG$ and its dual $\bGc$ do not appear to be distinguishable in any obvious way. Luckily, work of Van Wesep and Steel provides an elegant solution to this problem, at least when the ambient space is $\omega^\omega$. The key notion is the (first) separation property, already mentioned in the introduction and formally defined below.\footnote{\,The second separation property will only be needed in this subsection, and it will not appear anywhere else in this article. According to \cite[page 92]{van_wesep_thesis} and \cite[page 77]{van_wesep_separation}, both separation properties are due to Lusin.} More precisely, they showed that exactly one of $\bG$ and $\bGc$ has the separation property. We will show that their result holds for uncountable zero-dimensional Polish spaces (see Theorem \ref{theorem_separation_generalized}).

\begin{definition}[Lusin]
Let $Z$ be a set, and let $\bG\subseteq\PP(Z)$.
\begin{itemize}
\item We will say that $\bG$ has the \emph{first separation property} (or simply the \emph{separation property}) if for all $A,B\in\bG$ such that $A\cap B=\varnothing$ there exists $C\in\Delta(\bG)$ such that $A\subseteq C\subseteq Z\setminus B$.
\item We will say that $\bG$ has the \emph{second separation property} if for all $A,B\in\bG$ there exist $A',B'\in\bGc$ such that $A\setminus B\subseteq A'$, $B\setminus A\subseteq B'$, and $A'\cap B'=\varnothing$.
\end{itemize}
\end{definition}

We begin by showing that both of the separation properties are preserved by relativization.

\begin{lemma}\label{lemma_separation_relativization}
Let $\bS$ be a nice topological pointclass, and assume that $\Det(\bS(\omega^\omega))$ holds. Let $Z$ be a zero-dimensional Polish space, and let $\bG\in\NSDS(\omega^\omega)$.
\begin{enumerate}
\item If $\bG$ has the first separation property then $\bG(Z)$ has the first separation property.
\item If $\bG$ has the second separation property then $\bG(Z)$ has the second separation property.
\end{enumerate}
\end{lemma}

\begin{proof}
By \cite[Theorem 7.8]{kechris}, we can assume without loss of generality that $Z$ is a closed subspace of $\omega^\omega$. Using \cite[Proposition 2.8]{kechris}, we can fix a retraction $\rho:\omega^\omega\longrightarrow Z$. In order to prove $(1)$, assume that $\bG$ has the separation property. Pick disjoint $A,B\in\bG(Z)$, and observe that $\rho^{-1}[A],\rho^{-1}[B]\in\bG$ by Lemma \ref{lemma_relativization_basic}. Since $\bG$ has the separation property, there exists $C\in\Delta(\bG)$ such that $\rho^{-1}[A]\subseteq C\subseteq\omega^\omega\setminus\rho^{-1}[B]$. Notice that $C\cap Z\in\Delta(\bG(Z))$ by Lemma \ref{lemma_relativization_subspace}. Since clearly $A\subseteq C\cap Z\subseteq Z\setminus B$, this concludes the proof of $(1)$. The proof of $(2)$ is similar.
\end{proof}

Next, we will state the original results of Van Wesep and Steel. Theorem \ref{theorem_separation_van_wesep} first appeared as \cite[Theorem 4.2.2]{van_wesep_thesis} (see also \cite[Theorem 2]{van_wesep_separation} or \cite[Theorem 5.3]{van_wesep_cabal}), while Theorem \ref{theorem_separation_steel} is the main result of \cite{steel}.\footnote{\,The original statements of these results simply assumed $\AD$, but it is straightforward to check that their proofs actually yield the more precise versions given here.}

\begin{theorem}[Van Wesep]\label{theorem_separation_van_wesep}
Let $\bS$ be a nice topological pointclass, and assume that $\Det(\bS(\omega^\omega))$ holds. Let $\bG\in\NSDS(\omega^\omega)$. Then least one of $\bG$ and $\bGc$ has the second separation property.
\end{theorem}

\begin{theorem}[Steel]\label{theorem_separation_steel}
Let $\bS$ be a nice topological pointclass, and assume that $\Det(\bS(\omega^\omega))$ holds. Let $\bG\in\NSDS(\omega^\omega)$. Then at least one of $\bG$ and $\bGc$ has the first separation property.
\end{theorem}

Before obtaining the desired generalization, we will need two more auxiliary results, which are essentially due to Van Wesep. In fact, Lemmas \ref{lemma_universal_pair} and \ref{lemma_separation_first_vs_second} are ``extracted'' from the proof of \cite[Theorem 4.2.1]{van_wesep_thesis} (see also \cite[Theorem 1]{van_wesep_separation} or \cite[Theorem 5.2]{van_wesep_cabal}).

\begin{lemma}\label{lemma_universal_pair}
Let $\bS$ be a nice topological pointclass, and assume that $\Det(\bS(\omega^\omega))$ holds. Let $Z$ be an uncountable zero-dimensional Polish space, and let $\bG\in\NSDS(\omega^\omega)$. Then there exists $(A_0,A_1)\in\bG(Z\times Z)\times\bG(Z\times Z)$ such that for every $(B_0,B_1)\in\bG(Z)\times\bG(Z)$ there exists $z\in Z$ such that $(A_0)_z=B_0$ and $(A_1)_z=B_1$.
\end{lemma}
\begin{proof}
By Theorem \ref{theorem_wadge_implies_universal}, we can fix a $2^\omega$-universal set $S$ for $\bG(Z)$. Make the following definitions:
\begin{itemize}
\item $A'_0=\{(x,y,z)\in 2^\omega\times 2^\omega\times Z:(x,z)\in S\}$,
\item $A'_1=\{(x,y,z)\in 2^\omega\times 2^\omega\times Z:(y,z)\in S\}$.
\end{itemize}
Using Lemma \ref{lemma_relativization_basic}, one sees that $A'_0,A'_1\in\bG(2^\omega\times 2^\omega\times Z)$. Furthermore, it is easy to realize that for every $(B_0,B_1)\in\bG(Z)\times\bG(Z)$ there exists $(x,y)\in 2^\omega\times 2^\omega$ such that $(A'_0)_{(x,y)}=B_0$ and $(A'_1)_{(x,y)}=B_1$.

By \cite[Corollary 6.5]{kechris}, we can fix a subspace $K$ of $Z$ such that $K\approx 2^\omega$. Let $h:2^\omega\times 2^\omega\longrightarrow K$ be a homeomorphism. Define $j:2^\omega\times 2^\omega\times Z\longrightarrow Z\times Z$ by setting $j(x,y,z)=(h(x,y),z)$ for $(x,y,z)\in 2^\omega\times 2^\omega\times Z$, and observe that $j$ is an embedding. By Lemma \ref{lemma_relativization_subspace}, there exist $A_0,A_1\in\bG(Z\times Z)$ such that $A_0\cap (K\times Z)=j[A'_0]$ and $A_1\cap (K\times Z)=j[A'_1]$. It is straightforward to check that $A_0$ and $A_1$ are as desired.
\end{proof}

\begin{lemma}\label{lemma_separation_first_vs_second}
Let $\bS$ be a nice topological pointclass, and assume that $\Det(\bS(\omega^\omega))$ holds. Let $Z$ be an uncountable zero-dimensional Polish space, and let $\bG\in\NSDS(\omega^\omega)$. If $\bG$ has the second separation property then $\bGc(Z)$ does not have the first separation property.
\end{lemma}
\begin{proof}
Assume that $\bG$ has the second separation property. By Lemma \ref{lemma_universal_pair}, we can fix $(A_0,A_1)\in\bG(Z\times Z)\times\bG(Z\times Z)$ such that for every $(B_0,B_1)\in\bG(Z)\times\bG(Z)$ there exists $z\in Z$ such that $(A_0)_z=B_0$ and $(A_1)_z=B_1$. Since $\bG(Z\times Z)$ has the second separation property by Lemma \ref{lemma_separation_relativization}, there exist $A'_0,A'_1\in\bGc(Z\times Z)$ such that $A_0\setminus A_1\subseteq A'_0$, $A_1\setminus A_0\subseteq A'_1$, and $A'_0\cap A'_1=\varnothing$. Set $B_0=\{z\in Z:(z,z)\in A'_0\}$ and $B_1=\{z\in Z:(z,z)\in A'_1\}$. Observe that $B_0,B_1\in\bGc(Z)$ by Lemma \ref{lemma_relativization_basic}, and that $B_0\cap B_1=\varnothing$ because $A'_0\cap A'_1=\varnothing$.

Assume, in order to get a contradiction, that $\bGc(Z)$ has the first separation property. Then there exists $C\in\Delta(\bG(Z))$ such that $B_0\subseteq C\subseteq Z\setminus B_1$. Fix $z\in Z$ such that $(A_0)_z=Z\setminus C$ and $(A_1)_z=C$. It is straightforward to check that $z\in C$ iff $z\notin C$, which is a contradiction.
\end{proof}

We are finally in a position to completely solve the ``pair of socks'' problem.

\begin{theorem}\label{theorem_separation_generalized}
Let $\bS$ be a nice topological pointclass, and assume that $\Det(\bS(\omega^\omega))$ holds. Let $Z$ be an uncountable zero-dimensional Polish space. If $\bG\in\NSDS(Z)$ then exactly one of $\bG$ and $\bGc$ has the separation property.
\end{theorem}
\begin{proof}
Let $\bG(Z)\in\NSDS(Z)$, where $\bG\in\NSDS(\omega^\omega)$. It follows immediately from Lemma \ref{lemma_separation_first_vs_second} and Theorem \ref{theorem_separation_van_wesep} that at most one of $\bG(Z)$ and $\bGc(Z)$ has the separation property. On the other hand, by Lemma \ref{lemma_separation_relativization} and Theorem \ref{theorem_separation_steel}, at least one of $\bG(Z)$ and $\bGc(Z)$ has the separation property.
\end{proof}

As a first application, we will show that expansions preserve the separation property. This result will be useful in the proof of Corollary \ref{corollary_meager_semifilter}. Other applications of Theorem \ref{theorem_separation_generalized} will be given in the proofs of Corollaries \ref{corollary_baire_semifilter} and \ref{corollary_existence_negative_baire}.

\begin{lemma}\label{lemma_expansion_separation}
Let $\bS$ be a nice topological pointclass, and assume that $\Det(\bS(\omega^\omega))$ holds. Let $Z$ be an uncountable zero-dimensional Polish space, and let $\xi<\omega_1$. If $\bG\in\NSDS(Z)$ has the separation property then $\bG^{(\xi)}$ has the separation property.
\end{lemma}
\begin{proof}
Let $\bG(Z)\in\NSDS(Z)$, where $\bG\in\NSDS(\omega^\omega)$, and assume that $\bG(Z)$ has the separation property. Pick disjoint $A,B\in\bG(Z)^{(\xi)}$. By Lemma \ref{lemma_expansion_bijection}, we can fix a zero-dimensional Polish space $W$ and a $\bS^0_{1+\xi}$-measurable bijection $f:Z\longrightarrow W$ such that $f[A],f[B]\in\bG(W)$. Observe that $\bG$ must have the separation property, otherwise Lemma \ref{lemma_separation_relativization} and Theorem \ref{theorem_separation_generalized} would easily yield a contradiction. Therefore $\bG(W)$ has the separation property by Lemma \ref{lemma_separation_relativization}. This means that we can find $C\in\Delta(\bG(W))$ such that $f[A]\subseteq C\subseteq W\setminus f[B]$. Notice that $f^{-1}[C]\in\Delta\left(\bG(Z)^{(\xi)}\right)$ by Lemmas \ref{lemma_expansion_relativization_measurable_function} and \ref{lemma_expansion_relativization_move_xi}. Since clearly $A\subseteq f^{-1}[C]\subseteq Z\setminus B$, this concludes the proof.
\end{proof}

\subsection{Separated differences}\label{subsection_wadge_fundamental_sd}

The following notion was essentially introduced in \cite{louveau_article}, but (as in \cite{carroy_medini_muller_constructing}) we will follow the simplified approach given in \cite{louveau_book}.

\begin{definition}[Louveau]\label{definition_sd}
Let $Z$ be a space, let $1\leq\eta<\omega_1$, let $U_{\mu,n},A_{\mu,n}\subseteq Z$ for $\mu<\eta$ and $n\in\omega$, and let $A^\ast\subseteq Z$. Define
\begin{multline}
\SD_\eta((U_{\mu,n}:\mu<\eta,n\in\omega),(A_{\mu,n}:\mu<\eta,n\in\omega),A^\ast)=\\\nonumber=\bigcup_{\substack{\mu<\eta\\n\in\omega}}\left(A_{\mu,n}\cap U_{\mu,n}\setminus\bigcup_{\substack{\mu'<\mu\\m\in\omega}}U_{\mu',m}\right)\cup \left(A^\ast\setminus\bigcup_{\substack{\mu<\eta\\n\in\omega}}U_{\mu,n}\right).
\end{multline}
Given $\bD,\bG^\ast\subseteq\PP(Z)$, define $\SD_\eta(\bD,\bG^\ast)$ as the collection of all sets in the above form, where each $U_{\mu,n}\in\bS^0_1(Z)$ and $U_{\mu,m}\cap U_{\mu,n}=\varnothing$ whenever $m\neq n$, each $A_{\mu,n}\in\bD$, and $A^\ast\in\bG^\ast$. Sets in this form are known as \emph{separated differences}.
\end{definition}

Next, we state two simple but useful results (see \cite[Lemmas 19.2 and 19.3]{carroy_medini_muller_constructing} respectively). We remark that Lemma \ref{lemma_sd_differences} gives the simplest examples of Wadge classes that can be obtained as separated differences.

\begin{lemma}\label{lemma_sd_check}
Let $Z$ be a space, let $1\leq\eta<\omega_1$, and let $\bD,\bG\subseteq\PP(Z)$. Then
$$
\widecheck{\SD}_\eta(\bD,\bG)=\SD_\eta(\widecheck{\bD},\widecheck{\bG}).
$$
\end{lemma}

\begin{lemma}\label{lemma_sd_differences}
Let $Z$ be a space, let $1\leq\eta<\omega_1$, and let $\bD=\{\varnothing\}\cup\{Z\}$. Then
$$
\SD_\eta(\bD,\{\varnothing\})=\Diff_\eta(\bS^0_1(Z))\text{ and }\SD_\eta(\bD,\{Z\})=\widecheck{\Diff}_\eta(\bS^0_1(Z)).
$$
\end{lemma}

Separated differences are important because they make it possible to give concrete descriptions of the non-selfdual Wadge classes of level $0$. This is made precise by the following result, which is essentially due to Louveau (it can be easily derived using \cite[Proposition 19.4 and Theorem 22.2]{carroy_medini_muller_constructing}). See also the first paragraph of \S\ref{subsection_wadge_type_su} to clarify the meaning of $\bD$.

\begin{theorem}\label{theorem_sd_main}
Let $\bS$ be a nice topological pointclass, and assume that $\Det(\bS(\omega^\omega))$ holds. Let $Z$ be an uncountable zero-dimensional Polish space, and let $\bG\subseteq\bS(Z)$. Then the following conditions are equivalent:
\begin{itemize}
\item $\bG\in\NSD(Z)$ and $\ell(\bG)=0$,
\item There exist $1\leq\eta<\omega_1$, $\bG_n\in\NSD(Z)$ for $n\in\omega$ and $\bG^\ast\in\NSD(Z)$ such that each $\ell(\bG_n)\geq 1$, $\bG^\ast\subseteq\bD$ and $\bG=\SD_\eta(\bD,\bG^\ast)$, where $\bD=\bigcup_{n\in\omega}(\bG_n\cup\bGc_n)$.
\end{itemize}
\end{theorem}

\section{Wadge theory: type}\label{section_wadge_type}

\subsection{Definition and basic properties}\label{subsection_wadge_type_basics}

Given a partial order $(\PPP,\preccurlyeq)$ and $p\in\PPP$, recall the following definitions:
\begin{itemize}
\item $p$ is \emph{minimal} if there exists no $q\in\PPP$ such that $q\prec p$,
\item $p$ is a \emph{successor} if there exists $q\in\PPP$ such that $q\prec p$ and there exists  no $r\in\PPP$ such that $q\prec r\prec p$,
\item $p$ is a \emph{countable limit} if $p$ is not a successor and there exist $q_n\in\PPP$ for $n\in\omega$ such that $q_0\prec q_1\prec\cdots\prec p$ and for every $r\prec p$ there exists $n\in\omega$ such that $r\prec q_n$.
\end{itemize}

The following notion is essentially due to Louveau (see \cite[Definition 1.7]{louveau_article} and the remark that precedes it), although his ``official'' definition is of a (rather complicated) recursive nature, and it is limited to the Borel context. This is one of the fundamental notions needed to state our classification results.\footnote{\,It would have been more elegant to simply define $t(\bG)$ as the cofinality of $\bG$ in $\NSDxi(Z)$. However, we preferred to remain consistent with the original definition.}

\begin{definition}\label{definition_type}
Let $Z$ be a space, and let $\xi\leq\omega_1$. Define
$$
\NSDxi(Z)=\{\bG\in\NSD(Z):\ell(\bG)\geq\xi\}.
$$
Given a topological pointclass $\bS$, also set $\NSDxiS(Z)=\{\bG\in\NSDS(Z):\ell(\bG)\geq\xi\}$. Given $\bG\in\NSD(Z)$ such that $\ell(\bG)=\xi$, view $\NSDxi(Z)$ as a partial order under $\subseteq$, then define the \emph{type} of $\bG$ as
$$
\left.
\begin{array}{lcl}
& & t(\bG)=\left\{
\begin{array}{ll}
0 & \text{if }\bG\text{ is minimal in $\NSDxi(Z)$,}\\
1 & \text{if }\bG\text{ is a successor in $\NSDxi(Z)$,}\\
2 & \text{if }\bG\text{ is a countable limit in $\NSDxi(Z)$,}\\
3 & \text{otherwise.}\\
\end{array}
\right.
\end{array}
\right.
$$
\end{definition}

Notice that Theorem \ref{theorem_every_class_has_a_level} will guarantee that the type is defined under suitable determinacy assumptions. The following proposition (whose trivial proof is left to the reader) collects some basic properties of the type.

\begin{proposition}\label{proposition_type_basic}
Let $Z$ be a space, and let $\bG\in\NSD(Z)$. 
\begin{itemize}
\item $t(\bG)=0$ iff $\bG=\{\varnothing\}$ or $\bG=\{Z\}$,
\item $t(\bG)=t(\bGc)$,
\item If $\ell(\bG)=0$ then
$$
\left.
\begin{array}{lcl}
& & t(\bG)=\left\{
\begin{array}{ll}
1 & \text{if }\bG\text{ is a successor in $\NSD(Z)$,}\\
2 & \text{if }\bG\text{ is a countable limit in $\NSD(Z)$,}\\
3 & \text{otherwise.}\\
\end{array}
\right.
\end{array}
\right.
$$
\end{itemize}
\end{proposition}

We conclude this subsection by showing that expansions leave the type unchanged (see Corollary \ref{corollary_type_transfer_expansion} for the precise statement).

\begin{theorem}\label{theorem_type_order_isomorphism}
Let $\bS$ be a nice topological pointclass, and assume that $\Det(\bS(\omega^\omega))$ holds. Let $Z$ be an uncountable zero-dimensional Polish space, and let $\xi<\omega_1$. Then the function $\Phi_\xi:\NSDS(Z)\longrightarrow\NSDxiS(Z)$ defined by setting $\Phi_\xi(\bG)=\bG^{(\xi)}$ for $\bG\in\NSDS(Z)$ is an order-isomorphism with respect to $\subseteq$.
\end{theorem}
\begin{proof}
Using Theorem \ref{theorem_expansion}, one sees that $\Phi_\xi$ is well-defined and surjective. The fact that $\Phi_\xi$ is an order-isomorphism is then given by Corollary \ref{corollary_expansion_order_isomorphism}.
\end{proof}

\begin{corollary}\label{corollary_type_transfer_expansion}
Let $\bS$ be a nice topological pointclass, and assume that $\Det(\bS(\omega^\omega))$ holds. Let $Z$ be an uncountable zero-dimensional Polish space, let $\xi<\omega_1$, and let $\bG\in\NSDS(Z)$. Then $t(\bG)=t(\bG^{(\xi)})$.
\end{corollary}
\begin{proof}
By Theorem \ref{theorem_every_class_has_a_level}, there exists $\eta\leq\omega_1$ such that $\ell(\bG)=\eta$. First assume that $\eta<\omega_1$. By Theorem \ref{theorem_expansion}, there exists $\bL\in\NSDS(Z)$ such that $\bL^{(\eta)}=\bG$. Notice that $\Phi_{\xi +\eta}\circ\Phi_\eta^{-1}:\NSD^{(\eta)}_\bS(Z)\longrightarrow\NSD^{(\xi+\eta)}_\bS(Z)$ is an order-isomorphism with respect to $\subseteq$ by Theorem \ref{theorem_type_order_isomorphism}. Also observe that
$$
(\Phi_{\xi +\eta}\circ\Phi_\eta^{-1})(\bG)=\Phi_{\xi +\eta}(\bL)=\bL^{(\xi+\eta)}=(\bL^{(\eta)})^{(\xi)}=\bG^{(\xi)},
$$
where the third equality holds by Theorem \ref{theorem_expansion_composition}. Since Corollary \ref{corollary_expansion_level} guarantees that $\ell(\bG^{(\xi)})=\xi+\eta$, the desired conclusion immediately follows.

Finally, assume that $\eta=\omega_1$. We will prove that $\bG=\bG^{(\xi)}$, which clearly suffices. By Theorem \ref{theorem_expansion}, there exists $\bL\in\NSDS(Z)$ such that $\bL^{(\xi\cdot\omega)}=\bG$. Then
$$
\bG^{(\xi)}=(\bL^{(\xi\cdot\omega)})^{(\xi)}=\bL^{(\xi+\xi\cdot\omega)}=\bL^{(\xi\cdot (1+\omega))}=\bL^{(\xi\cdot\omega)}=\bG,
$$
where the second equality holds by Theorem \ref{theorem_expansion_composition}.
\end{proof}

\subsection{Separated unions}\label{subsection_wadge_type_su}

The following notion will ultimately allow us to explicitly describe the non-selfdual Wadge classes that are successors or countable limits (see \S\ref{subsection_wadge_type_small}).\footnote{\,According to Wadge, this notion had already been identified by ``the classical descriptive set theorists'' (see \cite[page 187]{wadge_cabal}).} For some insight on the relationship between separated unions and separated differences, see Lemma \ref{lemma_sd_versus_su}.

\begin{definition}\label{definition_su}
Let $Z$ be a space, let $\bG\subseteq\PP(Z)$, and let $\xi<\omega_1$. Define $\SU_\xi(\bG)$ to be the collection of all sets of the form
$$
\bigcup_{n\in\omega}(A_n\cap U_n),
$$
where each $A_n\in\bG$ and the $U_n\in\bS_{1+\xi}^0(Z)$ are pairwise disjoint. A set in this form is called a \emph{separated union} of sets in $\bG$.
\end{definition}

As a first step, we will show that this operation produces new non-selfdual Wadge classes from old ones. 

\begin{theorem}\label{theorem_su_nsd}
Let $\bS$ be a nice topological pointclass, and assume that $\Det(\bS(\omega^\omega))$ holds. Let $Z$ be an uncountable zero-dimensional Polish space, let $\bG_n\in\NSDS(Z)$ for $n\in\omega$, and let $\xi<\omega_1$. Set
$$
\bG=\SU_\xi\left(\bigcup_{n\in\omega}\bG_n\right).
$$
Then $\bG\in\NSD(Z)$.\footnote{\,However, it is not necessarily true that $\bG\subseteq\bS(Z)$.}
\end{theorem}
\begin{proof}
Choose $D\subseteq\PP(\omega)$ such that
$$
\HH_D(A_0,U_1,A_2,U_3,\ldots)=\bigcup_{k\in\omega}(A_{2k}\cap U_{2k+1})
$$
for all subsets $A_0,U_1,A_2,U_3,\ldots$ of the ambient set. Also fix a function $\sigma:\omega\longrightarrow\omega$ such that $\sigma^{-1}(n)$ is infinite for every $n\in\omega$. To obtain the desired result, simply apply Corollary \ref{corollary_new_from_old} to the classes $\bG_{\sigma(n)}$.
\end{proof}

Next, we will show that separated unions are connected to partitioned unions in a rather pleasing way.

\begin{lemma}\label{lemma_su_versus_pu}
Let $Z$ be a zero-dimensional space, let $\xi <\omega_1$, and let $\bD\subseteq\PP(Z)$ be selfdual. Then
$$
\SU_\xi(\bD)\cap\SUc_\xi(\bD)=\PU_\xi(\bD).
$$
\end{lemma}
\begin{proof}
The inclusion $\supseteq$ is trivial. In order to prove the inclusion $\subseteq$, pick $A\in\SU_\xi(\bD)$ such that $Z\setminus A\in\SU_\xi(\bD)$. Pick $U_n\in\bS^0_{1+\xi}(Z)$ and $A_n\in\bD$ for $n\in\omega$ such that
$$
A=\bigcup_{k\in\omega}(A_{2k}\cap U_{2k})\text{ and }Z\setminus A=\bigcup_{k\in\omega}(A_{2k+1}\cap U_{2k+1}).
$$
Observe that $\bigcup_{n\in\omega}U_n=Z$. Therefore, by Theorem \ref{theorem_reduction}, there exist $U'_n\in\bS^0_{1+\xi}(Z)$ for $n\in\omega$ such that the following conditions hold:
\begin{itemize}
	\item $U'_n\subseteq U_n$ for every $n\in\omega$,
	\item $U'_m\cap U'_n=\varnothing$ whenever $m\neq n$,
	\item $\bigcup_{n\in\omega}U'_n=Z$.
\end{itemize}
It is easy to realize that
$$
A=\bigcup_{k\in\omega}(A_{2k}\cap U'_{2k})\cup\bigcup_{k\in\omega}((Z\setminus A_{2k+1})\cap U'_{2k+1}).
$$
Since $\bD$ is selfdual, this shows that $A\in\PU_\xi(\bD)$.
\end{proof}

Recall that, by Theorem \ref{theorem_separation_generalized}, exactly one Wadge classes in every non-selfdual pair has the separation property. The following result tells us which one, at least in certain cases.

\begin{lemma}\label{lemma_su_dual_has_separation}
Let $Z$ be a zero-dimensional space, let $\xi<\omega_1$, and let $\bD\subseteq\PP(Z)$ be selfdual. Set $\bG=\SU_\xi(\bD)$. Then $\bGc$ has the separation property.
\end{lemma}
\begin{proof}
Pick $A_0,A_1\in\bGc$ such that $A_0\cap A_1=\varnothing$. Fix $B_n\in\bD$ and $U_n\in\bS^0_{1+\xi}(Z)$ for $n\in\omega$ such that the following conditions hold:
\begin{itemize}
\item $U_{2k+i}\cap U_{2j+i}=\varnothing$ whenever $k\neq j$ and $i<2$,
\item $A_i=Z\setminus\bigcup_{k\in\omega}(B_{2k+i}\cap U_{2k+i})$.
\end{itemize}
Using the fact that $A_0\cap A_1=\varnothing$, one sees that $\bigcup_{n\in\omega}U_n=Z$. So, by Theorem \ref{theorem_reduction}, there exist $U'_n\in\bS^0_{1+\xi}(Z)$ for $n\in\omega$ such that the following conditions hold:
\begin{itemize}
\item $U'_n\subseteq U_n$ for each $n$,
\item $U'_m\cap U'_n=\varnothing$ whenever $m\neq n$,
\item $\bigcup_{n\in\omega}U'_n=Z$.
\end{itemize}
Define
$$
C=\bigcup_{k\in\omega}((Z\setminus B_{2k})\cap U'_{2k})\cup\bigcup_{k\in\omega}(B_{2k+1}\cap U'_{2k+1}).
$$
Notice that $C\in\Delta(\bG)=\Delta(\bGc)$ because $\bD$ is selfdual. Therefore, to conclude the proof, it remains to show that $A_0\subseteq C\subseteq Z\setminus A_1$. We will only show that $A_0\subseteq C$, since $A_1\subseteq Z\setminus C$ can be verified using a similar argument. Pick $x\in A_0$. First assume that $x\in U'_{2k}$ for some $k$. Then $x\notin B_{2k}$, which clearly implies $x\in C$. Now assume that $x\in U'_{2k+1}$ for some $k$. Then $x\in B_{2k+1}$, otherwise the fact that $A_0\cap A_1=\varnothing$ would be contradicted. It follows that $x\in C$ in this case as well.
\end{proof}

\subsection{Characterizing small type}\label{subsection_wadge_type_small}

The aim of this subsection is to explicitly describe the non-selfdual Wadge classes of small type (that is, type $1$ or $2$), at least in the case of countable level (see Theorem \ref{theorem_characterization_small_type}). We will often consider $\bD=\bigcup_{n\in\omega}(\bG_n\cup\bGc_n)$, where the $\bG_n$ are intended to be the predecessors of a certain class $\bG$. We remark that this is simply a convenient way to unify the cases in which $\bG$ is a successor (all $\bG_n$ are the same) and a countable limit ($\bG_0\subsetneq\bG_1\subsetneq\cdots$).

\begin{lemma}\label{lemma_describe_successors_and_countable_limits}
Let $\bS$ be a nice topological pointclass, and assume that $\Det(\bS(\omega^\omega))$ holds. Let $Z$ be an uncountable zero-dimensional Polish space, and let $\bG_n\in\NSDS(Z)$ for $n\in\omega$. Set $\bD=\bigcup_{n\in\omega}(\bG_n\cup\bGc_n)$, and assume that $\SU_0(\bD)\subseteq\bS(Z)$. Then the following conditions are equivalent:
\begin{enumerate}
\item $\bG=\SU_0(\bD)$ or $\bG=\SUc_0(\bD)$,
\item $\bG\in\NSDS(Z)$, $\bD\subseteq\bG$, and for every $\bG'\in\NSDS(Z)$ such that $\bD\subseteq\bG'$ either $\bG\subseteq\bG'$ or $\bGc\subseteq\bG'$.
\end{enumerate}
\end{lemma}
\begin{proof}
First we will prove that $(1)\rightarrow (2)$. Assume that $(1)$ holds. Observe that $\bG,\bGc\in\NSDS(Z)$ by Theorem \ref{theorem_su_nsd} plus the assumption that $\SU_0(\bD)\subseteq\bS(Z)$. Furthermore, it is clear that $\bD\subseteq\bG$. Now pick $\bG'\in\NSDS(Z)$ such that $\bD\subseteq\bG'$. Assume, in order to get a contradiction, that $\bG\nsubseteq\bG'$ and $\bGc\nsubseteq\bG'$. It follows from Lemma \ref{lemma_wadge} that $\bG'\subseteq\bGc$ and $\bG'\subseteq\bG$, hence $\bG'\subseteq\PU_0(\bD)$ by Lemma \ref{lemma_su_versus_pu}. On the other hand, Proposition \ref{proposition_pu_basic} shows that $\PU_0(\bD)\subseteq\PU_0(\bG')=\bG'$. In conclusion, we see that $\bG'=\PU_0(\bD)$, which contradicts the fact that $\bG'$ is non-selfdual.

Now assume that $(2)$ holds. Using the implication $(1)\rightarrow (2)$, it is easy to see that both of the following conditions hold:
\begin{itemize}
\item $\SU_0(\bD)\subseteq\bG$ or $\SUc_0(\bD)\subseteq\bG$,
\item $\bG\subseteq\SU_0(\bD)$ or $\bG\subseteq\SUc_0(\bD)$.
\end{itemize}
On the other hand, if both of the inclusions $\SU_0(\bD)\subseteq\bG$ and $\bG\subseteq\SUc_0(\bD)$ were true, then the fact that $\SU_0(\bD)$ is non-selfdual would be contradicted. Similarly, the inclusions $\SUc_0(\bD)\subseteq\bG$ and $\bG\subseteq\SU_0(\bD)$ cannot both hold. At this point, it is clear that $(1)$ holds.
\end{proof}

Before giving the promised characterization, we will need one more preliminary result, which illustrates how separated unions interact with expansions.

\begin{lemma}\label{lemma_su_expansion} Let $\bS$ be a nice topological pointclass, and assume that $\Det(\bS(\omega^\omega))$ holds. Let $Z$ be an uncountable zero-dimensional Polish space, and let $\xi<\omega_1$. If $\bG_n\in\NSDS(Z)$ for $n\in\omega$ and $\bG=\SU_0(\bigcup_{n\in\omega}\bG_n\cup\bGc_n)\subseteq\bS(Z)$ then
$$
\bG^{(\xi)}=\SU_\xi\left(\bigcup_{n\in\omega}\left(\bG_n^{(\xi)}\cup\bGc_n^{(\xi)}\right)\right).
$$
\end{lemma}
\begin{proof}
Pick $\bG_n(Z)\in\NSDS(Z)$ for $n\in\omega$, where each $\bG_n\in\NSDS(\omega^\omega)$. Set
$$
\bG=\SU_0\left(\bigcup_{n\in\omega}\bG_n\cup\bGc_n\right).
$$
Given a space $W$, we will use the notation $\bD(W)=\bigcup_{n\in\omega}(\bG_n(W)\cup\bGc_n(W))$. Assume that $\SU_0(\bD(Z))\subseteq\bS(Z)$.

\noindent\textbf{Claim 1.} $\bG\subseteq\bS(\omega^\omega)$.

\noindent\textit{Proof.} Fix an embedding $j:\omega^\omega\longrightarrow Z$. Pick $A\in\bG$. Using Theorem \ref{theorem_reduction} and Lemma \ref{lemma_relativization_subspace}, it is possible to find $\widetilde{A}\in\SU_0(\bD(Z))$ in $Z$ such that $j[A]=\widetilde{A}\cap j[\omega^\omega]$. Since $\widetilde{A}\in\bS(Z)$ by assumption, it follows that $A=j^{-1}[\widetilde{A}]\in\bS(\omega^\omega)$. $\blacksquare$

\noindent\textbf{Claim 2.} $\bG(W)=\SU_0(\bD(W))$ whenever $W$ is an uncountable Polish space.

\noindent\textit{Proof.} Pick an uncountable Polish space $W$. Define $\Phi:\NSDS(\omega^\omega)\longrightarrow\NSDS(W)$ by setting $\Phi(\bL)=\bL(W)$ for $\bL\in\NSDS(\omega^\omega)$, and observe that $\Phi$ is an order-isomorphism with respect to $\subseteq$ by Theorem \ref{theorem_order_isomorphism}. Furthermore, Lemma \ref{lemma_relativization_basic} shows that $\widecheck{\Phi(\bL)}=\Phi(\bLc)$ for every $\bL\in\NSDS(\omega^\omega)$. By Claim 1, it is possible to apply Lemma \ref{lemma_describe_successors_and_countable_limits}, which gives a characterization that can be expressed using only order-theoretic terms plus the operation of taking the dual class. It follows that either $\Phi(\bG)=\SU_0(\bD(W))$ or $\Phi(\bG)=\SUc_0(\bD(W))$. However, using Lemma \ref{lemma_separation_relativization}, Theorem \ref{theorem_separation_generalized} and Lemma \ref{lemma_su_dual_has_separation}, it is straighforward to check that the second case would lead to a contradiction. $\blacksquare$

Now we have all the necessary tools to show that
$$
\bG(Z)^{(\xi)}=\SU_\xi\left(\bigcup_{n\in\omega}\left(\bG_n(Z)^{(\xi)}\cup\bGc_n(Z)^{(\xi)}\right)\right),
$$
which will conclude the proof by Claim 2. The inclusion $\subseteq$ follows easily from Claim 2 and the definition of expansion. In order to prove the other inclusion, pick
$$
A=\bigcup_{k\in\omega}(A_k\cap U_k),
$$
where each $A_k$ belongs to some $\bG_n(Z)^{(\xi)}$ or $\bGc_n(Z)^{(\xi)}$, and the $U_k\in\bS^0_{1+\xi}(Z)$ are pairwise disjoint. By Lemma \ref{lemma_expansion_bijection}, we can fix a zero-dimensional Polish space $W$ and a $\bS^0_{1+\xi}$-measurable bijection $f:Z\longrightarrow W$ such that each $f[A_k]$ belongs to some $\bG_n(W)$ or $\bGc_n(W)$, and each $f[U_k]\in\bS^0_1(W)$. Using Claim 2, Lemma \ref{lemma_expansion_relativization_measurable_function} and Lemma \ref{lemma_expansion_relativization_move_xi}, one sees that
$$
A=f^{-1}\left[\bigcup_{k\in\omega}(f[A_k]\cap f[U_k])\right]\in\bG^{(\xi)}(Z)=\bG(Z)^{(\xi)},
$$
as desired.
\end{proof}

\begin{theorem}\label{theorem_characterization_small_type}
Let $\bS$ be a nice topological pointclass, and assume that $\Det(\bS(\omega^\omega))$ holds. Let $Z$ be an uncountable zero-dimensional Polish space, and let $\bG\in\NSDS(Z)$ be such that $\ell(\bG)=\xi<\omega_1$. Then the following conditions are equivalent:
\begin{itemize}
\item $t(\bG)\in\{1,2\}$,
\item There exist $\bG_n\in\NSD(Z)$ for $n\in\omega$ such that each $\ell(\bG_n)\geq\xi$ and either $\bG=\SU_\xi(\bD)$ or $\bG=\SUc_\xi(\bD)$, where $\bD=\bigcup_{n\in\omega}(\bG_n\cup\bGc_n)$.
\end{itemize}
\end{theorem}
\begin{proof}
The case $\xi=0$ can be easily obtained as a consequence of Lemma \ref{lemma_describe_successors_and_countable_limits}. The general case will then follow, using the methods of \S\ref{subsection_wadge_fundamental_expansions} in conjunction with Corollary \ref{corollary_type_transfer_expansion} and Lemma \ref{lemma_su_expansion}.
\end{proof}

\section{Wadge theory: closure and preservation}\label{section_wadge_closure}

\subsection{Closure properties: preliminaries}\label{subsection_wadge_closure_preliminaries}

As will become increasingly clear, what we need from Wadge theory are closure properties. These properties are spelled out in the following definition, while the subsequent lemmas will be useful later on, towards obtaining the more sophisticated closure properties of \S\ref{subsection_wadge_closure_main}.

\begin{definition}
Let $Z$ be a space, let $\bG\subseteq\PP(Z)$, and let $\xi<\omega_1$.
\begin{itemize}
\item $\bG$ is \emph{closed under intersections with $\bP_{1+\xi}$ sets} if $A\cap B\in\bG$ whenever $A\in\bG$ and $B\in\bP_{1+\xi}(Z)$.
\item $\bG$ is \emph{closed under unions with $\bS_{1+\xi}$ sets} if $A\cup B\in\bG$ whenever $A\in\bG$ and $B\in\bS_{1+\xi}(Z)$.
\item $\bG$ is \emph{closed under $\SU_\xi$} if $\SU_\xi(\bG)=\bG$.
\end{itemize}
\end{definition}

The next four results show that these closure properties are well-behaved with respect to relativization and expansions.

\begin{lemma}\label{lemma_preservation_basic_closure_under_relativization}
Let $\bS$ be a nice topological pointclass, and assume that $\Det(\bS(\omega^\omega))$ holds. Let $Z$ and $W$ be uncountable zero-dimensional Borel spaces, let $\xi<\omega_1$, and let $\bG\in\NSDS(\omega^\omega)$. Then the following conditions are equivalent:
\begin{itemize}
\item $\bG(Z)$ is closed under intersections with $\bP_{1+\xi}$ sets (respectively under unions with $\bS_{1+\xi}$ sets),
\item $\bG(W)$ is closed under intersections with $\bP_{1+\xi}$ sets (respectively under unions with $\bS_{1+\xi}$ sets).
\end{itemize}
\end{lemma}
\begin{proof}
We will only prove one of the implications, since the other cases are perfectly analogous. So assume that $\bG(Z)$ is closed under intersections with $\bP_{1+\xi}$ sets. Without loss of generality, assume that $W$ is a subspace of $Z$. Pick $A\in\bG(W)$ and $B\in\bP_{1+\xi}(W)$. By Lemma \ref{lemma_relativization_subspace}, there exist $\widetilde{A}\in\bG(Z)$ and $\widetilde{B}\in\bP_{1+\xi}(Z)$ such that $\widetilde{A}\cap W=A$ and $\widetilde{B}\cap W= B$. Notice that $\widetilde{A}\cap\widetilde{B}\in\bG(Z)$ by assumption. Therefore $A\cap B=(\widetilde{A}\cap\widetilde{B})\cap W\in\bG(W)$ by Lemma \ref{lemma_relativization_subspace}.
\end{proof}

\begin{lemma}\label{lemma_preservation_basic_closure_under_expansion}
Let $\bS$ be a nice topological pointclass, and assume that $\Det(\bS(\omega^\omega))$ holds. Let $Z$ be an uncountable zero-dimensional Polish space, and let $\xi<\omega_1$. If $\bG\in\NSDS(Z)$ is closed under intersections with $\bP^0_1$ sets (respectively under unions with $\bS^0_1$ sets) then $\bG^{(\xi)}$ is closed under intersections with $\bP^0_{1+\xi}$ sets (respectively under unions with $\bS_{1+\xi}$ sets).
\end{lemma}
\begin{proof}
We will only prove one of the implications, since the other case is perfectly analogous. Pick $\bG(Z)\in\NSDS(Z)$ that is closed under intersections with $\bP^0_1$ sets, where $\bG\in\NSDS(\omega^\omega)$. Let $A\in\bG(Z)^{(\xi)}$ and $B\in\bP^0_{1+\xi}(Z)$. By Lemma \ref{lemma_expansion_bijection}, it is possible to fix a zero-dimensional Polish space $W$ and a $\bS^0_{1+\xi}$-measurable bijection $f:Z\longrightarrow W$ such that $f[A]\in\bG(W)$ and $f[B]\in\bP^0_1(W)$. Observe that $f[A]\cap f[B]\in\bG(W)$ by Lemma \ref{lemma_preservation_basic_closure_under_relativization}. Therefore
$$
A\cap B=f^{-1}[f[A]\cap f[B]]\in\bG^{(\xi)}(Z)=\bG(Z)^{(\xi)}
$$
by Lemmas \ref{lemma_expansion_relativization_measurable_function} and \ref{lemma_expansion_relativization_move_xi}, as desired.
\end{proof}

\begin{lemma}\label{lemma_preservation_closure_su_under_relativization}
Let $\bS$ be a nice topological pointclass, and assume that $\Det(\bS(\omega^\omega))$ holds. Let $Z$ and $W$ be uncountable zero-dimensional Borel spaces, let $\xi<\omega_1$, and let $\bG\in\NSDS(\omega^\omega)$. Then the following conditions are equivalent:
\begin{itemize}
\item $\bG(Z)$ is closed under $\SU_\xi$,
\item $\bG(W)$ is closed under $\SU_\xi$.
\end{itemize}
\end{lemma}
\begin{proof}
We will only prove one of the implications, since the other case is perfectly analogous. So assume that $\bG(Z)$ is closed under $\SU_\xi$. Without loss of generality, assume that $W$ is a subspace of $Z$. Pick $A_n\in\bG(W)$ and pairwise disjoint $U_n\in\bS_{1+\xi}(W)$ for $n\in\omega$. Using Theorem \ref{theorem_reduction}, it is possible to find pairwise disjoint $\widetilde{U}_n\in\bS_{1+\xi}(Z)$ for $n\in\omega$ such that each $\widetilde{U}_n\cap W=U_n$. On the other hand, by Lemma \ref{lemma_relativization_subspace}, there exist $\widetilde{A}_n\in\bG(Z)$ for $n\in\omega$ such that each $\widetilde{A}_n\cap W=A_n$. To conclude the proof, observe that
$$
\bigcup_{n\in\omega}(A_n\cap U_n)=W\cap\bigcup_{n\in\omega}(\widetilde{A}_n\cap\widetilde{U}_n)\in\bG(W),
$$
by Lemma \ref{lemma_relativization_subspace}, since $\bG(Z)$ is closed under $\SU_\xi$.
\end{proof}

\begin{lemma}\label{lemma_preservation_closure_su_under_expansion}
Let $\bS$ be a nice topological pointclass, and assume that $\Det(\bS(\omega^\omega))$ holds. Let $Z$ be an uncountable zero-dimensional Polish space, and let $\xi<\omega_1$. If $\bG\in\NSDS(Z)$ is closed under $\SU_0$ then $\bG^{(\xi)}$ is closed under $\SU_\xi$.
\end{lemma}
\begin{proof}
Let $\bG(Z)\in\NSDS(Z)$ be closed under $\SU_0$, where $\bG\in\NSDS(\omega^\omega)$. Pick $A_n\in\bG(Z)^{(\xi)}$ and pairwise disjoint $U_n\in\bS^0_{1+\xi}(Z)$ for $n\in\omega$. By Lemma \ref{lemma_expansion_bijection}, we can fix a zero-dimensional Polish space $W$ and a $\bS^0_{1+\xi}$-measurable bijection $f:Z\longrightarrow W$ such that $f[A_n]\in\bG(W)$ and $f[U_n]\in\bS^0_1(W)$ for each $n$. Notice that $\bG(W)$ is closed under $\SU_0$ by Lemma \ref{lemma_preservation_closure_su_under_relativization}. Therefore
$$
\bigcup_{n\in\omega}(A_n\cap U_n)=f^{-1}\left[\bigcup_{n\in\omega}(f[A_n]\cap f[U_n])\right]\in\bG^{(\xi)}(Z)=\bG(Z)^{(\xi)}
$$
by Lemmas \ref{lemma_expansion_relativization_measurable_function} and \ref{lemma_expansion_relativization_move_xi}, as desired.
\end{proof}

We conclude this subsection with two elementary results in the same vein (the straightforward proof of Lemma \ref{lemma_su_closed_under_su} is left to the reader).

\begin{lemma}\label{lemma_su_closed_under_su}
Let $Z$ be a space, let $\xi<\omega_1$, and let $\bD\subseteq\PP(Z)$. Set $\bG=\SU_\xi(\bD)$. Then $\bG$ is closed under $\SU_\xi$.
\end{lemma}

\begin{lemma}\label{lemma_preservation_su_under_sd}
Let $Z$ be a space, let $1\leq\eta<\omega_1$, and let $\bD,\bG^\ast\subseteq\PP(Z)$. Assume that $\bG^\ast$ is closed under $\SU_0$. Set $\bG=\SD_\eta(\bD,\bG^\ast)$. Then $\bG$ is closed under $\SU_0$.
\end{lemma}
\begin{proof}
Pick $A\in\SU_0(\bG)$. Fix $A_n\in\bG$ and pairwise disjoint $U_n\in\bS^0_1(Z)$ for $n\in\omega$ such that $A=\bigcup_{n\in\omega}(A_n\cap U_n)$. Given $n\in\omega$, fix $A_{n,\mu,k}$, $V_{n,\mu,k}$ and $A_n^\ast$ for $\mu<\eta$ and $k\in\omega$ as in Definition \ref{definition_sd} such that
$$
A_n=\SD_\eta((V_{n,\mu,k}:\mu<\eta,k\in\omega),(A_{n,\mu,k}:\mu<\eta,k\in\omega),A_n^\ast).
$$
Set $V'_{n,\mu,k}=V_{n,\mu,k}\cap U_n$ for $(n,k)\in\omega\times\omega$ and $\mu<\eta$. Notice that $V'_{n,\mu,k}\cap V'_{m,\mu,j}=\varnothing$ whenever $(n,k)\neq(m,j)$. Also set $A^\ast=\bigcup_{n\in\omega}(A_n^\ast\cap U_n)$, and observe that $A^\ast\in\bG^\ast$. It is straightforward to check that
$$
A=\SD_\eta((V'_{n,\mu,k}:\mu<\eta,(n,k)\in\omega\times\omega),(A_{n,\mu,k}:\mu<\eta,(n,k)\in\omega\times\omega),A^\ast),
$$
which concludes the proof.
\end{proof}

\subsection{Closure properties: good Wadge classes}\label{subsection_wadge_closure_good}

The following key notion (see \cite[Definition 12.1]{carroy_medini_muller_homogeneous}) is essentially due to van Engelen, although he did not give it a name.\footnote{\,One important difference is that van Engelen's treatment of this notion is fundamentally tied to Louveau's classification of the Borel Wadge classes from \cite{louveau_article}, hence it is limited to the Borel context.} Good Wadge classes will enable us to define a well-behaved notion of exact topological complexity (see \S\ref{subsection_homogeneous_complexity}). Furthermore, the closure properties given by Corollary \ref{corollary_closure_good} will be applied directly in the proofs of Theorems \ref{theorem_meager_semifilter} and \ref{theorem_existence_filter}.

\begin{definition}[Carroy, Medini, M\"uller]\label{definition_good}
Let $Z$ be a space, and let $\bG$ be a Wadge class in $Z$. We will say that $\bG$ is \emph{good} if the following conditions are satisfied:
\begin{enumerate}
\item $\bG$ is non-selfdual,
\item $\Delta(\Diff_\omega(\bS^0_2(Z)))\subseteq\bG$,
\item $\ell(\bG)\geq 1$.
\end{enumerate}
\end{definition}

Using Proposition \ref{proposition_expansion_differences} and Theorem \ref{theorem_expansion}, one can easily show that $\Diff_\eta(\bS^0_2(Z))$ and $\Diffc_\eta(\bS^0_2(Z))$ are examples of good Wadge classes whenever $Z$ is an uncountable zero-dimensional Polish space and $\omega\leq\eta <\omega_1$. The next proposition shows that, for all practical purposes, the requirement $(2)$ can be substituted by a weaker requirement.

\begin{proposition}\label{proposition_good_equivalent}
Let $\bS$ be a nice topological pointclass, and assume that $\Det(\bS(\omega^\omega))$ holds. Let $Z$ be an uncountable zero-dimensional Polish space, and let $\bG\in\NSDS(Z)$ be such that $\ell(\bG)\geq 1$. Then the following conditions are equivalent:
\begin{itemize}
\item $\Delta(\Diff_\omega(\bS^0_2(Z)))\subseteq\bG$,
\item $\Diff_n(\bS^0_2(Z))\subseteq\bG$ for every $n\in\omega$.
\end{itemize}
\end{proposition}
\begin{proof}
In order to prove the non-trivial implication, assume that $\Diff_n(\bS^0_2(Z))\subseteq\bG$ for every $n\in\omega$. By Theorem \ref{theorem_expansion}, we can fix $\bL\in\NSDS(Z)$ such that $\bL^{(1)}=\bG$. It follows from Proposition \ref{proposition_expansion_differences} and Corollary \ref{corollary_expansion_order_isomorphism} that $\Diff_n(\bS^0_1(Z))\subseteq\bL$ for each $n$, hence $\Diff_\omega(\bS^0_1(Z))\subseteq\bL$ or $\Diffc_\omega(\bS^0_1(Z))\subseteq\bL$ by Theorem \ref{theorem_complete_analysis_delta^0_2}. To conclude the proof, apply Proposition \ref{proposition_expansion_differences} and Corollary \ref{corollary_expansion_order_isomorphism} one more time.
\end{proof}

Next, we will show that good Wadge classes are well-behaved with respect to relativization.

\begin{lemma}\label{lemma_good_relativization}
Let $\bS$ be a nice topological pointclass, and assume that $\Det(\bS(\omega^\omega))$ holds. Let $Z$ and $W$ be uncountable zero-dimensional Polish spaces, and let $\bG\in\NSDS(\omega^\omega)$. Then the following conditions are equivalent:
\begin{itemize}
\item $\bG(Z)$ is a good Wadge class,
\item $\bG(W)$ is a good Wadge class.
\end{itemize}
\end{lemma}
\begin{proof}
First observe that $\bG(Z)$ and $\bG(W)$ are non-selfdual by Theorem \ref{theorem_relativization_uncountable}. Furthermore, Theorem \ref{theorem_level_relativization} guarantees that $\ell(\bG(Z))\geq 1$ iff $\ell(\bG(W))\geq 1$. Therefore, by Proposition \ref{proposition_good_equivalent}, it remains to show that $\Diff_n(\bS^0_2(Z))\subseteq\bG(Z)$ for every $n\in\omega$ iff $\Diff_n(\bS^0_2(W))\subseteq\bG(W)$ for every $n\in\omega$. In order to see this, use the methods of \cite[\S9]{carroy_medini_muller_constructing} and \cite[Lemma 21.5]{carroy_medini_muller_constructing} to obtain $\bL_n\in\NSDS(\omega^\omega)$ for $n\in\omega$ such that each $\bL_n(Z)=\Diff_n(\bS^0_2(Z))$ and each $\bL_n(W)=\Diff_n(\bS^0_2(W))$. Then, the desired equivalence will follow from Theorem \ref{theorem_order_isomorphism}.
\end{proof}

The remainder of this subsection is devoted to obtaining the fundamental closure properties of good Wadge classes (see Corollary \ref{corollary_closure_good}), as a particular case of a result of independent interest (see Theorem \ref{theorem_closure_borel}). The following result first appeared as \cite[Lemma 12.3]{carroy_medini_muller_homogeneous}, and it generalizes \cite[Lemma 3.6.a]{andretta_hjorth_neeman}. However, we will give a new proof, which relies on separated differences and Theorem \ref{theorem_sd_main} instead of \cite{andretta_hjorth_neeman}. Corollary \ref{corollary_closure_good} first appeared as \cite[Theorem 12.4]{carroy_medini_muller_homogeneous}.

\begin{lemma}[Carroy, Medini, M\"uller]\label{lemma_closure_closed_open}
Let $\bS$ be a nice topological pointclass, and assume that $\Det(\bS(\omega^\omega))$ holds. Let $Z$ be an uncountable zero-dimensional Polish space, and let $\bG\in\NSDS(Z)$. Assume that $\Diff_n(\bS^0_1(Z))\subseteq\bG$ whenever $1\leq n<\omega$. Then:
	\begin{itemize}
	\item $\bG$ is closed under intersections with $\bP^0_1$ sets,
	\item $\bG$ is closed under unions with $\bS^0_1$ sets.
	\end{itemize}
\end{lemma}
\begin{proof}
It will be enough to prove the first statement, since the second one will follow by considering the dual class. By Lemma \ref{lemma_closure_level}, we can assume without loss of generality that $\ell(\bG)=0$. Fix $\eta$, $\bD$, $\bG^\ast$ and $\bG_n$ for $n\in\omega$ as in Theorem \ref{theorem_sd_main}, so that $\bG=\SD_\eta(\bD,\bG^\ast)$. Pick $A\in\bG$ and $C\in\bP^0_1(Z)$. Write
$$
A=\SD_\eta((U_{\mu,n}:\mu<\eta,n\in\omega),(A_{\mu,n}:\mu<\eta,n\in\omega),A^\ast),
$$
where the $U_{\mu,n}$, $A_{\mu,n}$ and $A^\ast$ are as in Definition \ref{definition_sd}.

First assume that $\bD=\{\varnothing\}\cup\{Z\}$, hence either $\bG^\ast=\{\varnothing\}$ or $\bG^\ast=\{Z\}$. Using Lemma \ref{lemma_sd_differences} and the fact that the difference hierarchy is strictly increasing (see \cite[Exercise 22.26]{kechris}), one sees that $\eta\geq\omega$. In particular, the following definitions make sense:
\begin{itemize}
\item $U'_{0,0}=Z\setminus C$,
\item $U'_{0,n}=\varnothing$ for $1\leq n<\omega$,
\item $A'_{0,n}=\varnothing$ for $n\in\omega$,
\item $U'_{1+\mu,n}=U_{\mu,n}$ for $\mu<\eta$ and $n<\omega$,
\item $A'_{1+\mu,n}=A_{\mu,n}$ for $\mu<\eta$ and $n<\omega$.
\end{itemize}
It is easy to realize that
$$
A\cap C=\SD_\eta((U'_{\mu,n}:\mu<\eta,n\in\omega),(A'_{\mu,n}:\mu<\eta,n\in\omega),A^\ast),
$$
which shows that $A\cap C\in\bG$.

Now assume that $\bD\supsetneq\{\varnothing\}\cup\{Z\}$. In particular, since each $\ell(\bG_n)\geq 1$, one can use Lemma \ref{lemma_closure_level} to obtain that $\bD$ is closed under intersections with $\bP^0_1$ sets. By Theorem \ref{theorem_reduction}, it is possible to find $U'_{0,n}\in\bS^0_1(Z)$ for $n\in\omega$ such that the following conditions hold:
\begin{itemize}
\item $U'_{0,0}\subseteq Z\setminus C$,
\item $U'_{0,1+n}\subseteq U_{0,n}$ for $n\in\omega$,
\item $U'_{0,m}\cap U'_{0,n}=\varnothing$ whenever $m\neq n$,
\item $\bigcup_{n\in\omega}U'_{0,n}=(Z\setminus C)\cup\bigcup_{n\in\omega}U_{0,n}$.
\end{itemize}
Also make the following definitions:
\begin{itemize}
\item $A'_{0,0}=\varnothing$,
\item $A'_{0,1+n}=A_{0,n}\cap C$ for $n<\omega$,
\item $U'_{\mu,n}=U_{\mu,n}$ for $1\leq\mu<\eta$ and $n<\omega$,
\item $A'_{\mu,n}=A_{\mu,n}$ for $1\leq\mu<\eta$ and $n<\omega$.
\end{itemize}
It is easy to realize that
$$
A\cap C=\SD_\eta((U'_{\mu,n}:\mu<\eta,n\in\omega),(A'_{\mu,n}:\mu<\eta,n\in\omega),A^\ast),
$$
which concludes the proof.
\end{proof}

\begin{theorem}\label{theorem_closure_borel}
Let $\bS$ be a nice topological pointclass, and assume that $\Det(\bS(\omega^\omega))$ holds. Let $Z$ be an uncountable zero-dimensional Polish space, let $\xi<\omega_1$, and let $\bG\in\NSDS(Z)$. Assume that $\ell(\bG)\geq\xi$ and $\Diff_n(\bS^0_{1+\xi}(Z))\subseteq\bG$ whenever $1\leq n<\omega$. Then:
\begin{itemize}
\item $\bG$ is closed under intersections with $\bP^0_{1+\xi}$ sets,
\item $\bG$ is closed under unions with $\bS^0_{1+\xi}$ sets.
\end{itemize}
\end{theorem}
\begin{proof}
By Theorem \ref{theorem_expansion}, we can fix $\bL\in\NSDS(Z)$ such that $\bL^{(\xi)}=\bG$. Notice that
$$
\Diff_n(\bS^0_1(Z))^{(\xi)}=\Diff_n(\bS^0_{1+\xi}(Z))\subseteq\bG=\bL^{(\xi)}
$$
for every $n\in\omega$, where the first equality holds by Proposition \ref{proposition_expansion_differences}. It follows from Corollary \ref{corollary_expansion_order_isomorphism} that $\Diff_n(\bS^0_1(Z))\subseteq\bL$ for every $n\in\omega$. Therefore, Lemma \ref{lemma_closure_closed_open} guarantees that $\bL$ is closed under intersections with $\bP^0_1$ sets and unions with $\bS^0_1$ sets. To conclude the proof, apply Lemma \ref{lemma_preservation_basic_closure_under_expansion}.
\end{proof}

\begin{corollary}[Carroy, Medini, M\"uller]\label{corollary_closure_good}
Let $\bS$ be a nice topological pointclass, and assume that $\Det(\bS(\omega^\omega))$ holds. Let $Z$ be an uncountable zero-dimensional Polish space, and let $\bG\subseteq\bS(Z)$ be a good Wadge class in $Z$. Then:
\begin{itemize}
\item $\bG$ is closed under intersections with $\bP^0_2$ sets,
\item $\bG$ is closed under unions with $\bS^0_2$ sets.
\end{itemize}
\end{corollary}

\subsection{Preservation of small type under separated differences}\label{subsection_wadge_closure_preservation_type}

The aim of this subsection is to show that the operation $\SD_\eta(\bD,-)$ preserves the property of having small type for classes of level $0$ (see Theorem \ref{theorem_preservation_small_type_under_sd} for a precise statement). This is one of the ingredients needed to obtain the crucial closure properties of \S\ref{subsection_wadge_closure_main}. We begin with a technical lemma.

\begin{lemma}\label{lemma_sd_preserves_su}
Let $Z$ be a zero-dimensional space, let $1\leq\eta<\omega_1$, let $\bD\subseteq\PP(Z)$ be selfdual, and let $\bG_n^\ast\subseteq\PP(Z)$ for $n\in\omega$. Set $\bG_n=\SD_\eta(\bD,\bG_n^\ast)$ for $n\in\omega$. Then
$$
\SU_0\left(\bigcup_{n\in\omega}(\bG_n\cup\bGc_n)\right)=\SD_\eta\left(\bD,\SU_0\left(\bigcup_{n\in\omega}(\bG_n^\ast\cup\bGc_n^\ast)\right)\right).
$$
\end{lemma}
\begin{proof}
For notational convenience, set
$$
\bD\!^\ast=\bigcup_{n\in\omega}(\bG_n^\ast\cup\bGc_n^\ast).
$$
Since $\bG_n\subseteq\SD_\eta(\bD,\SU_0(\bD\!^\ast))$ for each $n$, it is clear that $\bG_n^\ast\subseteq\SU_0(\bD\!^\ast)$ for each $n$. Similarly, using Lemma \ref{lemma_sd_check} and the assumption that $\bD$ is selfdual, one sees that $\bGc_n\subseteq\SD_\eta(\bD,\SU_0(\bD\!^\ast))$ for each $n$. Since $\SD_\eta(\bD,\SU_0(\bD\!^\ast))$ is closed under $\SU_0$ by Lemmas \ref{lemma_su_closed_under_su} and \ref{lemma_preservation_su_under_sd}, the inclusion $\subseteq$ follows.

In order to prove the inclusion $\supseteq$, pick $A\in\SD_\eta(\bD,\SU_0(\bD\!^\ast))$. Fix $A_{\mu,n}\in\bD$, $U_{\mu,n}\in\bS^0_1(Z)$ and $A^\ast\in\SU_0(\bD\!^\ast)$ as in Definition \ref{definition_sd} such that
$$
A=\SD_\eta((U_{\mu,n}:\mu<\eta,n\in\omega),(A_{\mu,n}:\mu<\eta,n\in\omega),A^\ast).
$$
Also fix $A_k^\ast\in\bD\!^\ast$ and pairwise disjoint $V_k\in\bS^0_1(Z)$ for $k\in\omega$ such that
$$
A^\ast=\bigcup_{k\in\omega}(A_k^\ast\cap V_k).
$$
Set $V_{-1}=\bigcup\{U_{\mu,n}:\mu<\eta\text{ and }n\in\omega\}$. By Theorem \ref{theorem_reduction}, there exist $V'_k\in\bS^0_1(Z)$ for $-1\leq k<\omega$ satisfying the following conditions:
\begin{itemize}
\item $V'_k\subseteq V_k$ for $-1\leq k<\omega$,
\item $V'_j\cap V'_k=\varnothing$ whenever $j\neq k$,
\item $\bigcup_{-1\leq k<\omega}V'_k=\bigcup_{-1\leq k<\omega}V_k$.
\end{itemize}
Also choose an arbitrary $A_{-1}^\ast\in\bD\!^\ast$. It is straightforward to check that
$$
A=\bigcup_{-1\leq k<\omega}(\SD_\eta((U_{\mu,n}:\mu<\eta,n\in\omega),(A_{\mu,n}:\mu<\eta,n\in\omega),A_k^\ast)\cap V'_k).
$$
Finally, using Lemma \ref{lemma_sd_check} and the assumption that $\bD$ is selfdual, one sees that the right-hand side belongs to $\SU_0(\bigcup_{n\in\omega}(\bG_n\cup\bGc_n))$, as desired.
\end{proof}

\begin{theorem}\label{theorem_preservation_small_type_under_sd}
Let $\bS$ be a nice topological pointclass, and assume that $\Det(\bS(\omega^\omega))$ holds. Let $Z$ be an uncountable zero-dimensional Polish space, let $1\leq\eta<\omega_1$, let $\bL_n\in\NSD(Z)$ for $n\in\omega$ be such that each $\ell(\bL_n)\geq 1$, and let $\bG^\ast\in\NSD(Z)$ be such that $\bG^\ast\subseteq\bD$, where $\bD=\bigcup_{n\in\omega}(\bL_n\cup\bLc_n)$. Set $\bG=\SD_\eta(\bD,\bG^\ast)$, and assume that $\bG\subseteq\bS(Z)$. If $\ell(\bG^\ast)=0$ and $t(\bG^\ast)\in\{1,2\}$ then $t(\bG)\in\{1,2\}$.
\end{theorem}
\begin{proof}
Assume that $\ell(\bG^\ast)=0$ and $t(\bG^\ast)\in\{1,2\}$. By Theorem \ref{theorem_characterization_small_type}, there exist $\bG_n^\ast\in\NSDS(Z)$ for $n\in\omega$ such that
$$
\bG^\ast=\SU_0\left(\bigcup_{n\in\omega}(\bG^\ast_n\cup\bGc^\ast_n)\right)\text{ or }\bGc^\ast=\SU_0\left(\bigcup_{n\in\omega}(\bG^\ast_n\cup\bGc^\ast_n)\right).
$$
Set $\bG_n=\SD_\eta(\bD,\bG_n^\ast)$ for $n\in\omega$. Observe that Theorem \ref{theorem_sd_main} guarantees that each $\bG_n\in\NSDS(Z)$. Furthermore, it guarantees that $\bG\in\NSDS(Z)$ and $\ell(\bG)=0$. Since
$$
\bG=\SU_0\left(\bigcup_{n\in\omega}(\bG_n\cup\bGc_n)\right)\text{ or }\bGc=\SU_0\left(\bigcup_{n\in\omega}(\bG_n\cup\bGc_n)\right)
$$
by Lemmas \ref{lemma_sd_check} and \ref{lemma_sd_preserves_su}, the desired conclusion follows from Theorem \ref{theorem_characterization_small_type}.
\end{proof}

\subsection[Preservation of the separation property]{Preservation of the separation property under separated differences}\label{subsection_wadge_closure_preservation_separation}

The aim of this subsection is to show that the operation $\SD_\eta(\bD,-)$ preserves the separation property (see Theorem \ref{theorem_preservation_separation_under_sd} for a precise statement). This is one of the ingredients needed to obtain the crucial closure properties of \S\ref{subsection_wadge_closure_main}.

\begin{lemma}\label{lemma_sd_versus_su}
Let $Z$ be a zero-dimensional space, let $1\leq\eta<\omega_1$, and let $\bD\subseteq\PP(Z)$ be selfdual with $Z\in\bD$. Then there exists a selfdual $\bD'\subseteq\PP(Z)$ such that
$$
\SD_\eta(\bD,\{\varnothing\})=\SU_0(\bD').
$$
\end{lemma}
\begin{proof}
If $\eta=1$, just set $\bD'=\bD$. Next, assume that $\eta=\eta'+1$, where $1\leq\eta'<\omega_1$. Set $\bD'=\SD_{\eta'}(\bD,\bD)$, and observe that $\bD'$ is selfdual by Lemma \ref{lemma_sd_check}. To prove the inclusion $\subseteq$, pick $A\in\SD_\eta(\bD,\{\varnothing\})$. Fix $A_{\mu,n}\in\bD$ and $U_{\mu,n}\in\bS^0_1(Z)$ as in Definition \ref{definition_sd} such that
$$
A=\SD_\eta((U_{\mu,n}:\mu<\eta,n\in\omega),(A_{\mu,n}:\mu<\eta,n\in\omega),\varnothing).
$$
Set $A_k=\SD_{\eta'}((U_{\mu,n}:\mu<\eta',n\in\omega),(A_{\mu,n}:\mu<\eta',n\in\omega),A_{\eta',k})$ and
$$
V_k=U_{\eta',k}\cup\bigcup_{\substack{\mu<\eta'\\n\in\omega}}U_{\mu,n}
$$
for $k\in\omega$. By Theorem \ref{theorem_reduction}, it is possible to find $V'_k\in\bS^0_1(Z)$ for $k\in\omega$ satisfying the following conditions:
\begin{itemize}
\item $V'_k\subseteq V_k$ for $k\in\omega$,
\item $V'_j\cap V'_k=\varnothing$ whenever $j\neq k$,
\item $\bigcup_{k\in\omega}V'_k=\bigcup_{k\in\omega}V_k$.
\end{itemize}
It is easy to realize that $A=\bigcup_{n\in\omega}(A_k\cap V'_k)$, which shows that $A\in\SU_0(\bD')$, as desired. In order to prove the inclusion $\supseteq$, first observe that $\bD'\subseteq\SD_\eta(\bD,\{\varnothing\})$. It follows that
$$
\SU_0(\bD')\subseteq\SU_0(\SD_\eta(\bD,\{\varnothing\}))=\SD_\eta(\bD,\{\varnothing\}),
$$
where the equality holds by Lemma \ref{lemma_preservation_su_under_sd}, since $\{\varnothing\}$ is obviously closed under $\SU_0$.

To conclude the proof, assume that $\eta$ is a limit ordinal. Set
$$
\bD'=\bigcup_{1\leq\eta'<\eta}(\SD_{\eta'}(\bD,\{\varnothing\})\cup\SDc_{\eta'}(\bD,\{\varnothing\})),
$$
and observe that $\bD'$ is selfdual. To prove the inclusion $\subseteq$, pick $A\in\SD_\eta(\bD,\{\varnothing\})$. Fix $A_{\mu,n}\in\bD$ and $U_{\mu,n}\in\bS^0_1(Z)$ as in Definition \ref{definition_sd} such that
$$
A=\SD_\eta((U_{\mu,n}:\mu<\eta,n\in\omega),(A_{\mu,n}:\mu<\eta,n\in\omega),\varnothing).
$$
Set $V_{\eta'}=\bigcup\{U_{\mu,n}:\mu<\eta'\text{ and }n\in\omega\}$ for $1\leq\eta'<\eta$. By Theorem \ref{theorem_reduction}, it is possible to find $V'_{\eta'}\in\bS^0_1(Z)$ for $1\leq\eta'<\eta$ satisfying the following conditions:
\begin{itemize}
\item $V'_{\eta'}\subseteq V_{\eta'}$ for $\eta'<\eta$,
\item $V'_{\eta'}\cap V'_{\eta''}=\varnothing$ whenever $\eta'\neq\eta''$,
\item $\bigcup_{1\leq\eta'<\eta}V'_{\eta'}=\bigcup_{1\leq\eta'<\eta}V_{\eta'}$.
\end{itemize}
Finally, set
$$
A_{\eta'}=\SD_{\eta'}((U_{\mu,n}:\mu<\eta',n\in\omega),(A_{\mu,n}:\mu<\eta',n\in\omega),\varnothing)
$$
for $1\leq\eta'<\eta$. It is easy to realize that
$$
A=\bigcup_{1\leq\eta'<\eta}(A_{\eta'}\cap V'_{\eta'}),
$$
which shows that $A\in\SU_0(\bD')$, as desired. In order to prove the inclusion $\supseteq$, the argument given above will work, provided that $\bD'\subseteq\SD_\eta(\bD,\{\varnothing\})$. To see this, use the fact that
$$
\SDc_{\eta'}(\bD,\{\varnothing\})\subseteq\SD_{\eta'+1}(\bD,\{\varnothing\})
$$
for each $\eta'$, which follows from Lemma \ref{lemma_sd_check} and the assumption that $Z\in\bD$.
\end{proof}

While Lemma \ref{lemma_sd_versus_su} will be sufficient for most applications, in the proof of Theorem \ref{theorem_closure_main_type} we will need the following more precise version. However, their proofs are essentially the same.

\begin{lemma}\label{lemma_sd_versus_su_precise}
Let $\bS$ be a nice topological pointclass, and assume that $\Det(\bS(\omega^\omega))$ holds. Let $Z$ be an uncountable zero-dimensional Polish space, let $1\leq\eta<\omega_1$, and let $\bG_n\in\NSD(Z)$ for $n\in\omega$ be such that each $\ell(\bG_n)\geq 1$. Set $\bD=\bigcup_{n\in\omega}(\bG_n\cup\bGc_n)$, and assume that $\SD_\eta(\bD,\{\varnothing\})\subseteq\bS(Z)$. Then there exist $\bG'_n\in\NSDS(Z)$ for $n\in\omega$ such that
$$
\SD_\eta(\bD,\{\varnothing\})=\SU_0(\bD'),
$$
where $\bD'=\bigcup_{n\in\omega}(\bG'_n\cup\bGc'_n)$.
\end{lemma}
\begin{proof}
If $\eta=1$, set $\bG'_n=\bG_n$ for $n\in\omega$. If $\eta=\eta'+1$, where $1\leq\eta'<\omega_1$, set $\bG'_n=\SD_{\eta'}(\bD,\bG_n)$ for $n\in\omega$. If $\eta$ is a limit ordinal,  set $\bG'_{\eta'}=\SD_{\eta'}(\bD,\{\varnothing\})$ for every $\eta'$ such that $1\leq\eta'<\eta$. To verify that $\SD_\eta(\bD,\{\varnothing\})=\SU_0(\bD')$, proceed as in the proof of Lemma \ref{lemma_sd_versus_su}. To see that each $\bG'_n\in\NSDS(Z)$ and each $\bG'_{\eta'}\in\NSDS(Z)$, apply Theorem \ref{theorem_sd_main}.
\end{proof}

\begin{theorem}\label{theorem_preservation_separation_under_sd}
Let $\bS$ be a nice topological pointclass, and assume that $\Det(\bS(\omega^\omega))$ holds. Let $Z$ be an uncountable zero-dimensional Polish space, let $1\leq\eta<\omega_1$, let $\bL_n\in\NSD(Z)$ for $n\in\omega$ be such that each $\ell(\bL_n)\geq 1$, and let $\bG^\ast\in\NSD(Z)$ be such that $\bG^\ast\subseteq\bD$, where $\bD=\bigcup_{n\in\omega}(\bL_n\cup\bLc_n)$. Set $\bG=\SD_\eta(\bD,\bG^\ast)$, and assume that $\bG\subseteq\bS(Z)$. If $\bG^\ast$ has the separation property then $\bG$ has the separation property.
\end{theorem}
\begin{proof}
Assume that $\bG^\ast$ has the separation property. Pick disjoint $A_0,A_1\in\bG$. Given $i<2$, fix $U_{\mu,2k+i}\in\bS^0_1(Z)$ and $A_{\mu,2k+i}\in\bD$ for $\mu<\eta$ and $k\in\omega$ as in Definition \ref{definition_sd} such that
$$
A_i=\SD_\eta((U_{\mu,2k+i}:\mu<\eta,k\in\omega),(A_{\mu,2k+i}:\mu<\eta,k\in\omega),A_i^\ast).
$$
Set $U_i=\bigcup\{U_{\mu,2k+i}:\mu<\eta\text{ and }k\in\omega\}$ for $i<2$.

First assume that $\Diff_n(\bS^0_1(Z))\subseteq\bG^\ast$ whenever $1\leq n<\omega$. Then, by Lemma \ref{lemma_closure_closed_open}, we can assume without loss of generality that $A_i^\ast\cap U_i=\varnothing$ for each $i$. It follows that $A_0^\ast\cap A_1^\ast=\varnothing$. Therefore, since $\bG^\ast$ has the separation property, we can fix $B^\ast\in\Delta(\bG^\ast)$ such that $A_0^\ast\subseteq B^\ast\subseteq Z\setminus A_1^\ast$. Given $\mu<\eta$, by Theorem \ref{theorem_reduction} there exist $U'_{\mu,n}\in\bS^0_1(Z)$ for $n\in\omega$ satisfying the following conditions:
\begin{itemize}
\item $U'_{\mu,n}\subseteq U_{\mu,n}$ for each $n$,
\item $U'_{\mu,m}\cap U'_{\mu,n}=\varnothing$ whenever $m\neq n$,
\item $\bigcup_{n\in\omega}U'_{\mu,n}=\bigcup_{n\in\omega}U_{\mu,n}$.
\end{itemize}
Also set $B_{\mu,2k}=A_{\mu,2k}$ and $B_{\mu,2k+1}=Z\setminus A_{\mu,2k+1}$ for $\mu<\eta$ and $k\in\omega$, then define
$$
B=\SD_\eta((U'_{\mu,n}:\mu<\eta,n\in\omega),(B_{\mu,n}:\mu<\eta,n\in\omega),B^\ast).
$$
Since $B\in\Delta(\bG)$ by Lemma \ref{lemma_sd_check}, it remains to verify that $A_0\subseteq B\subseteq Z\setminus A_1$. We will only show that $A_0\subseteq B$, as $A_1\subseteq Z\setminus B$ can be verified using a similar argument. So pick $x\in A_0$. Given $\mu\leq\eta$, set
$$
V_\mu=\bigcup_{\substack{\mu'<\mu\\n\in\omega}}U_{\mu',n}=\bigcup_{\substack{\mu'<\mu\\n\in\omega}}U'_{\mu',n}.
$$
\noindent\textbf{Case 1.} $x\in A_0^\ast\setminus V_\eta$.

\noindent In this case, it is clear that $x\in B^\ast\setminus V_\eta$, hence $x\in B$.

\noindent\textbf{Case 2.} $x\in A_0^\ast\cap V_\eta$.

\noindent Fix the minimal $\mu<\eta$ such that $x\in\bigcup_{n\in\omega}U'_{\mu,n}$, then fix $n\in\omega$ such that $x\in U'_{\mu,n}$. If we had $n=2k$ for some $k\in\omega$, then we would have $x\in U'_{\mu,2k}\subseteq U_{\mu,2k}$, contradicting the assumption that $A_0^\ast\cap U_0=\varnothing$. Therefore, we can fix $k\in\omega$ such that $x\in U'_{\mu,2k+1}$. Notice that $x\notin A_{\mu,2k+1}$, otherwise we would have $x\in A_1$ by the minimality of $\mu$, contradicting the fact that $A_0\cap A_1=\varnothing$. At this point, it is clear that $x\in B$.

\noindent\textbf{Case 3.} $x\notin A_0^\ast$.

\noindent This means that we can fix $\mu<\eta$ and $k\in\omega$ such that
$$
x\in A_{\mu,2k}\cap U_{\mu,2k}\setminus\bigcup_{\substack{\mu'<\mu\\j\in\omega}}U_{\mu',2j}.
$$
First assume that $x\notin V_\mu$. Since
$$
x\in U_{\mu,2k}\subseteq\bigcup_{n\in\omega}U_{\mu,n}=\bigcup_{n\in\omega}U'_{\mu,n},
$$
we can fix $n\in\omega$ such that $x\in U'_{\mu,n}$. If $n=2j$ for some $j\in\omega$ then $j=k$, because otherwise $U_{\mu,2j}\cap U_{\mu,2k}=\varnothing$. On the other hand, if $n=2j+1$ for some $j\in\omega$ then $x\notin A_{\mu,n}$, otherwise it would follow that $x\in A_1$. In either case, it is clear that $x\in B$, as desired. Now assume that $\mu'<\mu$ is minimal such that $x\in\bigcup_{n\in\omega}U'_{\mu',n}$. Fix $n\in\omega$ such that $x\in U'_{\mu',n}$, and observe that $n=2j+1$ for some $j\in\omega$. Furthermore, we must have $x\notin A_{\mu',2j+1}$, otherwise it would follow that $x\in A_1$. Therefore $x\in B$ in this case as well.

Finally, assume that $\Diff_n(\bS^0_1(Z))\nsubseteq\bG^\ast$, where $1\leq n<\omega$. Then $\bG^\ast\subseteq\bD^0_2(Z)$ by Lemma \ref{lemma_wadge}. Therefore, by Theorem \ref{theorem_complete_analysis_delta^0_2}, one of the following cases must hold:
\begin{itemize}
\item $\bG^\ast=\{\varnothing\}$ or $\bG^\ast=\{Z\}$,
\item $\bG^\ast=\Diff_m(\bS^0_1(Z))$ or $\bG^\ast=\Diffc_m(\bS^0_1(Z))$, where $1\leq m<\omega$.
\end{itemize}
However, the case $\bG^\ast=\{\varnothing\}$ contradicts the assumption that $\bG^\ast$ has the separation property. Furthermore, since each $\Diffc_m(\bS^0_1(Z))$ has the separation property by Lemmas \ref{lemma_sd_differences}, \ref{lemma_su_dual_has_separation} and \ref{lemma_sd_versus_su}, the case $\bG^\ast=\Diff_m(\bS^0_1(Z))$ would lead to a contradiction by Theorem \ref{theorem_separation_generalized}. By Lemma \ref{lemma_su_dual_has_separation}, in order to conclude the proof, it will be enough to show that $\bG=\SUc_0(\bD')$ for some selfdual $\bD'\subseteq\PP(Z)$. In the case $\bG^\ast=\{Z\}$, this clearly follows from Lemmas \ref{lemma_sd_check} and \ref{lemma_sd_versus_su}. Now assume that $\bG^\ast=\Diffc_m(\bS^0_1(Z))$, where $1\leq m<\omega$. Observe that $t(\bG^\ast)=1$ by Theorem \ref{theorem_complete_analysis_delta^0_2}. Furthermore, it follows from Lemma \ref{lemma_sd_differences} and Theorem \ref{theorem_sd_main} that $\ell(\bG^\ast)=0$. By Theorems \ref{theorem_characterization_small_type} and \ref{theorem_preservation_small_type_under_sd}, this means that there exists a selfdual $\bD'\subseteq\PP(Z)$ such that either $\bG=\SU_0(\bD')$ or $\bG=\SUc_0(\bD')$. Assume, in order to get a contradiction, that $\bG=\SU_0(\bD')$. Notice that $\Diff_m(\bS^0_1(Z))$ is closed under $\SU_0$ by Lemmas \ref{lemma_sd_differences}, \ref{lemma_su_closed_under_su} and \ref{lemma_sd_versus_su}. It follows from Lemmas \ref{lemma_sd_check} and \ref{lemma_preservation_su_under_sd} that $\bGc$ is closed under $\SU_0$. Hence
$$
\bG=\SU_0(\bD')\subseteq\SU_0(\SUc_0(\bD'))=\SU_0(\bGc)=\bGc,
$$
which contradicts the fact that $\bG$ is non-selfdual.
\end{proof}

\subsection{Closure properties: the main results}\label{subsection_wadge_closure_main}

This subsection is the culmination (and the conclusion) of the purely Wadge-theoretic portion of this article.

\begin{theorem}\label{theorem_closure_main_separation}
Let $\bS$ be a nice topological pointclass, and assume that $\Det(\bS(\omega^\omega))$ holds. Let $Z$ be an uncountable zero-dimensional Polish space, let $\xi<\omega_1$, and let $\bG\in\NSDS(Z)$. Assume that the following conditions hold:
\begin{itemize}
\item $\ell(\bG)=\xi$,
\item $\bG$ does not have the separation property.
\end{itemize}
Then $\bG$ is closed under $\SU_\xi$.
\end{theorem}
\begin{proof}
Using the methods of \S\ref{subsection_wadge_fundamental_expansions} and Lemma \ref{lemma_preservation_closure_su_under_expansion}, one sees that the case $\xi=0$ implies the general case. So assume, in order to get a contradiction, that the desired result does not hold in this particular case. By Theorem \ref{theorem_well-founded}, we can assume that $\bG$ is a $\subseteq$-minimal counterexample (in other words, $\ell(\bG)=0$ and $\bG$ does not have the separation property, but $\bG$ is not closed under $\SU_0$). Fix $\eta$, $\bD$ and $\bG^\ast$ as given by Theorem \ref{theorem_sd_main} such that $\bG=\SD_\eta(\bD,\bG^\ast)$. Notice that $\bG^\ast$ does not have the separation property, otherwise $\bG$ would have the separation property by Theorem \ref{theorem_preservation_separation_under_sd}. Similarly, Lemma \ref{lemma_preservation_su_under_sd} shows that $\bG^\ast$ is not closed under $\SU_0$. Furthermore, as $\bD$ is selfdual and $\bG^\ast\subseteq\bD\subseteq\bG$, one sees that $\bG^\ast\subsetneq\bG$. It follows that $\ell(\bG^\ast)\geq 1$, otherwise the minimality of $\bG$ would be contradicted. In particular, $\bG^\ast$ is closed under $\SU_0$ (notice that $\bG^\ast\neq\{Z\}$ because $\bG^\ast$ does not have the separation property), which is a contradiction.
\end{proof}

\begin{theorem}\label{theorem_closure_main_type}
Let $\bS$ be a nice topological pointclass, and assume that $\Det(\bS(\omega^\omega))$ holds. Let $Z$ be an uncountable zero-dimensional Polish space, let $\xi<\omega_1$, and let $\bG\in\NSDS(Z)$. Assume that the following conditions hold:
\begin{itemize}
\item $\ell(\bG)=\xi$,
\item $t(\bG)=3$.
\end{itemize}
Then $\bG$ is closed under $\SU_\xi$.
\end{theorem}
\begin{proof}
Using the methods of \S\ref{subsection_wadge_fundamental_expansions}, Corollary \ref{corollary_type_transfer_expansion} and Lemma \ref{lemma_preservation_closure_su_under_expansion}, one sees that the case $\xi=0$ implies the general case. So assume, in order to get a contradiction, that the desired result does not hold in this particular case. By Theorem \ref{theorem_well-founded}, we can assume that $\bG$ is a $\subseteq$-minimal counterexample (in other words, $\ell(\bG)=0$ and $t(\bG)=3$, but $\bG$ is not closed under $\SU_0$). Fix $\eta$, $\bD$ and $\bG^\ast$ as given by Theorem \ref{theorem_sd_main} such that $\bG=\SD_\eta(\bD,\bG^\ast)$. Observe that $t(\bG^\ast)\neq 0$, otherwise we would have $t(\bG)\in\{1,2\}$ by Lemma \ref{lemma_sd_check}, Theorem \ref{theorem_characterization_small_type} and Lemma \ref{lemma_sd_versus_su_precise}. To conclude the proof, we will show that the cases $\ell(\bG^\ast)=0$ and $\ell(\bG^\ast)\geq 1$ both lead to a contradiction.

First assume that $\ell(\bG^\ast)=0$. Notice that this implies that $t(\bG^\ast)=3$, otherwise we would have $t(\bG)\in\{1,2\}$ by Theorem \ref{theorem_preservation_small_type_under_sd}. Similarly, Lemma \ref{lemma_preservation_su_under_sd} shows that $\bG^\ast$ is not closed under $\SU_0$. Since $\bG^\ast\subsetneq\bG$, this contradicts the minimality of $\bG$. Finally, assume that $\ell(\bG^\ast)\geq 1$. In particular, $\bG^\ast$ is closed under $\SU_0$ (notice that $\bG^\ast\neq\{Z\}$ because $t(\bG^\ast)\neq 0$). By Lemma \ref{lemma_preservation_su_under_sd}, this contradicts the assumption that $\bG$ is not closed under $\SU_0$.
\end{proof}

\section{Zero-dimensional homogeneous spaces}\label{section_homogeneous}

\subsection{The complexity of a zero-dimensional space}\label{subsection_homogeneous_complexity}

In this subsection, we will discuss how Wadge classes can be used to describe the \emph{absolute} topological complexity of a zero-dimensional space (as opposed to its \emph{relative} complexity as a subset of some ambient space). In the spirit of Definition \ref{definition_relativization}, our canonical choice of zero-dimensional Polish space will be $\omega^\omega$. Intuitively, a $\bS$ space is one of complexity \emph{at most} $\bS$.

\begin{definition}
Let $\bS$ be a topological pointclass, and let $X$ be a zero-dimensional space. We will say that $X$ is a \emph{$\bS$ space} if there exists an embedding $j:X\longrightarrow\omega^\omega$ such that $j[X]\in\bS(\omega^\omega)$.
\end{definition}

Given $\bG\in\NSD(\omega^\omega)$, as we have already observed in \S\ref{subsection_wadge_fundamental_relativization}, the assignment obtained by setting $\bS(Z)=\bG(Z)$ for every space $Z$ yields a topological pointclass. In this case, which is the most important case for our purposes, we will simply talk about $\bG$ spaces. Another example that is very relevant in the present context is obtained by setting $\bS(Z)=\Delta(\Diff_\omega(\bS^0_2(Z)))$ for every space $Z$. In this case, we will simply talk about $\Delta(\Diff_\omega(\bS^0_2))$ spaces. By setting $\bS(Z)=\Borel(Z)$ (respectively $\bS(Z)=\bS^1_n(Z)$ or $\bS(Z)=\bP^1_n(Z)$, where $1\leq n<\omega$) for every space $Z$, one obtains the usual notion of Borel (respectively $\bS^1_n$ or $\bP^1_n$) space.

\begin{definition}\label{definition_exact_complexity}
Let $\bG$ be a good Wadge class in $\omega^\omega$, and let $X$ be a zero-dimensional space. We will say that $X$ has \emph{exact complexity $\bG$} if there exists an embedding $j:X\longrightarrow\omega^\omega$ such that $j[X]\wc=\bG$.
\end{definition}

To motivate the assumptions in the above definition, consider $X=\omega^\omega$: it is easy to realize that there exist embeddings $i,j:X\longrightarrow\omega^\omega$ such that $i[X]\wc=\bP^0_1(\omega^\omega)$ and $j[X]\wc=\bP^0_2(\omega^\omega)$. So, if we dropped ``good'' from Definition \ref{definition_exact_complexity}, there would be spaces with more than one exact complexity, which is highly undesirable. In fact, the point of the following results is that the above notions are well-behaved if we restrict the attention to good Wadge classes. Proposition \ref{proposition_characterization_complexity} is ultimately inspired by \cite[Lemma 4.2.16]{van_engelen_thesis}.

\begin{proposition}\label{proposition_characterization_complexity}
Let $\bS$ be a nice topological pointclass, and assume that $\Det(\bS(\omega^\omega))$ holds. Let $\bG\subseteq\bS(\omega^\omega)$ be a good Wadge class in $\omega^\omega$, and let $X$ be a zero-dimensional space. Then the following conditions are equivalent:
\begin{enumerate}
\item $X$ is a $\bG$ space,
\item $j[X]\in\bG(Z)$ for every zero-dimensional Polish space $Z$ and embedding $j:X\longrightarrow Z$,
\item $j[X]\in\bG(Z)$ for some zero-dimensional Polish space $Z$ and embedding $j:X\longrightarrow Z$.
\end{enumerate}
\end{proposition}
\begin{proof}
In order to prove the implication $(1)\rightarrow (2)$, fix an embedding $i:X\longrightarrow\omega^\omega$ such that $i[X]\in\bG$. Now pick a zero-dimensional Polish space $Z$ and an embedding $j:X\longrightarrow Z$. Set $h=j\circ i^{-1}$, and observe that $h:i[X]\longrightarrow j[X]$ is a homeomorphism. Therefore, by \cite[Theorem 3.9]{kechris}, we can fix $G\in\bP^0_2(\omega^\omega)$, $H\in\bP^0_2(Z)$ and a homeomorphism $\widetilde{h}:G\longrightarrow H$ such that $h\subseteq\widetilde{h}$. Since $i[X]\in\bG(G)$ by Lemma \ref{lemma_relativization_subspace}, it is clear that $j[X]\in\bG(H)$. A further application of Lemma \ref{lemma_relativization_subspace} yields $A\in\bG(Z)$ such that $A\cap H=j[X]$. Since $\bG(Z)$ is closed under intersection with $\bP^0_2$ sets by Lemma \ref{lemma_good_relativization} and Corollary \ref{corollary_closure_good}, it follows that $j[X]=A\cap H\in\bG(Z)$.

The implication $(2)\rightarrow (3)$ is clear. In order to prove the implication $(3)\rightarrow (1)$, fix a zero-dimensional Polish space $Z$ and an embedding $i:X\longrightarrow Z$ such that $i[X]\in\bG(Z)$. Let $j:Z\longrightarrow\omega^\omega$ be an embedding, and observe that $j[Z]\in\bP^0_2(\omega^\omega)$ by \cite[Theorem 3.11]{kechris}. Furthermore, it is clear that $j[i[X]]\in\bG(j[Z])$ because $j:Z\longrightarrow j[Z]$ is a homeomorphism. Therefore, by Lemma \ref{lemma_relativization_subspace}, there exists $A\in\bG$ such that $A\cap j[Z]=j[i[X]]$. Since $\bG$ is closed under intersection with $\bP^0_2$ sets by Corollary \ref{corollary_closure_good}, it follows that $j[i[X]]=(j\circ i)[X]\in\bG$.
\end{proof}

\begin{corollary}\label{corollary_characterization_exact_complexity}
Let $\bS$ be a nice topological pointclass, and assume that $\Det(\bS(\omega^\omega))$ holds. Let $\bG\subseteq\bS(\omega^\omega)$ be a good Wadge class in $\omega^\omega$, and let $X$ be a zero-dimensional space. Then the following conditions are equivalent:
\begin{enumerate}
\item $X$ has exact complexity $\bG$,
\item $X$ is a $\bG$ space but not a $\bGc$ space,
\item $j[X]\wc=\bG(Z)$ for every zero-dimensional Polish space $Z$ and embedding $j:X\longrightarrow Z$,
\item $j[X]\wc=\bG(Z)$ for some uncountable zero-dimensional Polish space $Z$ and embedding $j:X\longrightarrow Z$.
\end{enumerate}
\end{corollary}
\begin{proof}
We will only prove the implication $(4)\rightarrow (1)$, in order to make clear where the uncountability assumption is used. The rest of the proof is left to the reader. So fix an uncountable zero-dimensional Polish space $Z$ and an embedding $i:X\longrightarrow Z$ such that $i[X]\wc=\bG(Z)$. Let $j:Z\longrightarrow\omega^\omega$ be an embedding. By Lemma \ref{lemma_wadge}, it will be enough to show that $(j\circ i)[X]\in\bG\setminus\bGc$. The fact that $(j\circ i)[X]\in\bG$ is clear by Proposition \ref{proposition_characterization_complexity}. Now assume, in order to get a contradiction, that $(j\circ i)[X]\in\bGc$. Then $i[X]\in\bGc(Z)$ by Proposition \ref{proposition_characterization_complexity}, which shows that $\bG(Z)$ is selfdual. Since $Z$ is uncountable, this contradicts Theorem \ref{theorem_relativization_uncountable}.
\end{proof}

We conclude this subsection by showing that complexity is hereditary with respect to sufficiently simple subspaces.

\begin{proposition}\label{proposition_complexity_hereditary}
Let $\bS$ be a nice topological pointclass, and assume that $\Det(\bS(\omega^\omega))$ holds. Let $\bG\subseteq\bS(\omega^\omega)$ be a good Wadge class in $\omega^\omega$, and let $X$ be a zero-dimensional $\bG$ space. Then every subspace $Y$ of $X$ such that $Y\in\bP^0_2(X)$ is a $\bG$ space.
\end{proposition}
\begin{proof}
Fix an embedding $j:X\longrightarrow\omega^\omega$ such that $j[X]\in\bG$. Pick $Y\in\bP^0_2(X)$. Let $\widetilde{Y}\in\bP^0_2(\omega^\omega)$ be such that $\widetilde{Y}\cap j[X]=j[Y]$. It follows from Corollary \ref{corollary_closure_good} that $j[Y]\in\bG$, which concludes the proof.
\end{proof}

\subsection{Uniqueness}\label{subsection_homogeneous_uniqueness}

In this subsection we will give a precise statement of the uniqueness part of the classification (see Theorem \ref{theorem_uniqueness_embeddings}). According to this result, exact complexity and Baire category are sufficient to identify a zero-dimensional homogeneous space up to homeomorphism. In other words, for every good Wadge class $\bG$ in $\omega^\omega$, there exist at most two zero-dimensional homogeneous spaces of exact complexity $\bG$: one meager, and one Baire.\footnote{\,Recall that, by Proposition \ref{proposition_homogeneous_dichotomy}, every homogeneous space is either meager or Baire.}

First, we will state two preliminary results, which are essentially the same as \cite[Lemma 14.2 and Theorem 14.4]{carroy_medini_muller_homogeneous}. Lemma \ref{lemma_open_implies_whole} will be useful in the proofs of Corollary \ref{corollary_h-homogeneously_embedded} and Theorem \ref{theorem_existence_negative_meager}, while Theorem \ref{theorem_homogeneous_spaces_have_good_complexity} justifies restricting the attention to good Wadge classes in our classification results.

\begin{lemma}[Carroy, Medini, M\"uller]\label{lemma_open_implies_whole}
Let $\bS$ be a nice topological pointclass, and assume that $\Det(\bS(\omega^\omega))$ holds. Let $\bG\subseteq\bS(\omega^\omega)$ be a good Wadge class in $\omega^\omega$, and let $X$ be a zero-dimensional homogeneous space. If $X$ has a non-empty open $\bG$ subspace then $X$ is a $\bG$ space.
\end{lemma}

\begin{theorem}[Carroy, Medini, M\"uller]\label{theorem_homogeneous_spaces_have_good_complexity}
Let $\bS$ be a nice topological pointclass, and assume that $\Det(\bS(\omega^\omega))$ holds. Let $X$ be a zero-dimensional homogeneous $\bS$ space that is not a $\Delta(\Diff_\omega(\bS^0_2))$ space. Then there exists a good Wadge class $\bG\subseteq\bS(\omega^\omega)$ such that $X$ has exact complexity $\bG$.
\end{theorem}

We finally come to the uniqueness result, which is essentially the same as \cite[Theorem 15.2]{carroy_medini_muller_homogeneous}. We decided to state this result in the following form (that is, in terms of embeddings into $2^\omega$) for two reasons. First, this will allow us to derive Corollary \ref{corollary_h-homogeneously_embedded}, which will be useful in the proofs of Corollary \ref{corollary_existence_negative_baire} and Theorem \ref{theorem_characterization_semifilters}. Second, Theorem \ref{theorem_uniqueness_embeddings} shows that, given a zero-dimensional homogeneous space $X$ of sufficiently high complexity, there is essentially a unique way to embed $X$ into $2^\omega$ densely, which seems to be of independent interest.

\begin{theorem}[Carroy, Medini, M\"{u}ller]\label{theorem_uniqueness_embeddings}
Let $\bS$ be a nice topological pointclass, and assume that $\Det(\bS(\omega^\omega))$ holds. Let $\bG\subseteq\bS(\omega^\omega)$ be a good Wadge class in $\omega^\omega$, and let $X$ and $Y$ be zero-dimensional homogeneous spaces that satisfy the following conditions:
\begin{itemize}
\item $X$ and $Y$ both have exact complexity $\bG$,
\item $X$ and $Y$ are either both meager spaces or both Baire spaces.
\end{itemize}
Then $X$ and $Y$ are homeomorphic. More precisely, if $i:X\longrightarrow 2^\omega$ and $j:Y\longrightarrow 2^\omega$ are dense embeddings, then there exists a homeomorphism $h:2^\omega\longrightarrow 2^\omega$ such that $h[i[X]]=j[Y]$.
\end{theorem}

\begin{corollary}\label{corollary_h-homogeneously_embedded}
Let $\bS$ be a nice topological pointclass, and assume that $\Det(\bS(\omega^\omega))$ holds. Let $\bG\subseteq\bS(\omega^\omega)$ be a good Wadge class in $\omega^\omega$, and let $X$ be a dense homogeneous subspace of $2^\omega$ of exact complexity $\bG$. If $U\in\bD^0_1(2^\omega)$ is non-empty, then there exists a homeomorphism $h_U:2^\omega\longrightarrow U$ such that $h_U[X]=X\cap U$. In particular, both $X$ and $2^\omega\setminus X$ are strongly homogeneous.
\end{corollary}
\begin{proof}
Pick a non-empty $U\in\bD^0_1(2^\omega)$. Set $Y=X\cap U$, and observe that $Y$ is a zero-dimensional homogeneous space. It is clear that if $X$ is a meager space then $Y$ is a meager space, and that if $X$ is a Baire space then $Y$ is  Baire space. Notice that $Y$ is a $\bG$ space by Proposition \ref{proposition_complexity_hereditary}. On the other hand, if $Y$ were a $\bGc$ space, then $X$ would be a $\bGc$ space by Lemma \ref{lemma_open_implies_whole}, contradicting Corollary \ref{corollary_characterization_exact_complexity}. Therefore, by a further application of Corollary \ref{corollary_characterization_exact_complexity}, one sees that $Y$ has exact complexity $\bG$. In conclusion, since $U\approx 2^\omega$, the desired homeomorphism is given by Theorem \ref{theorem_uniqueness_embeddings}. The strong homogeneity of $X$ and $2^\omega\setminus X$ then follows from Lemma \ref{lemma_terada}.
\end{proof}

\subsection{Existence: the positive results}\label{subsection_homogeneous_existence_positive}

In this subsection, we will show that certain assumptions on the level, type, and separation property of a good Wadge class $\bG$ permit the construction of a zero-dimensional homogeneous space of exact complexity $\bG$. Our method of proof will naturally yield semifilters (recall that these are homogeneous by Corollary \ref{corollary_semifilter_homogeneous}).

While the proof of Theorem \ref{theorem_meager_semifilter} is essentially the same as the proof of \cite[Theorem 12.4]{medini_semifilters} (which in turn is essentially due to van Engelen), we have isolated more clearly the key closure property of $\bG$ needed for it to go through. We begin with two technical results.

\begin{lemma}\label{lemma_closed_subset}
Let $\bS$ be a nice topological pointclass, and assume that $\Det(\bS(\omega^\omega))$ holds. Let $Z$ be a zero-dimensional Polish space, and let $\bG\in\NSDS(Z)$ be closed under intersections with $\bP^0_1$ sets and under unions with $\bS^0_1$ sets. Assume that $\bG=A\wc$ and $B\in\bG$ is such that $A$ is a closed subset of $B$. Then $\bG=B\wc$.
\end{lemma}
\begin{proof}
We will use $\cl$ to denote closure in $Z$. Assume, in order to get a contradiction, that $A\nleq B$. Then $Z\setminus B\leq A$ by Lemma \ref{lemma_wadge}. By the closure properties of $\bG$, it follows that
$$
Z\setminus A=((Z\setminus B)\cap\cl(A))\cup (Z\setminus\cl(A))\in\bG,
$$
which contradicts the fact that $\bG$ is non-selfdual.
\end{proof}

\begin{lemma}\label{lemma_closed_decomposition}
Let $Z$ be a space, let $B\subseteq Z$, let $\bG\subseteq\PP(Z)$ be closed under $\SU_1$, and let $A_n\in\bG$ for $n\in\omega$. If each $A_n$ is a closed subset of $B$ and $\bigcup_{n\in\omega}A_n=B$ then $B\in\bG$.
\end{lemma}
\begin{proof}
We will use $\cl$ to denote closure in $Z$. Assume that each $A_n\in\bP^0_1(B)$. Set
$$
U_n=\cl(A_n)\setminus\bigcup_{k<n}\cl(A_k)
$$
for $n\in\omega$, and observe that each $U_n\in\bS^0_2(Z)$. It is straightforward to check that
$$
B=\bigcup_{n\in\omega}(A_n\cap U_n).
$$
Since $\bG$ is closed under $\SU_1$, it follows that $B\in\bG$.
\end{proof}

\begin{theorem}\label{theorem_meager_semifilter}
Let $\bS$ be a nice topological pointclass, and assume that $\Det(\bS(\omega^\omega))$ holds. Let $\bG\subseteq\bS(\omega^\omega)$ be a good Wadge class in $\omega^\omega$ that is closed under $\SU_1$. Then there exists a meager semifilter of exact complexity $\bG$.
\end{theorem}

\begin{proof}
We begin by making the following definitions:
\begin{itemize}
\item $\Omega=2^{<\omega}$,
\item $Z=2^\Omega$,
\item $K=\{\{x\re n:n\in\omega\}:x\in 2^\omega\}$,
\item $\widehat{S}=\{z\in Z:z\subseteq z'\text{ for some }z'\in S\}$ whenever $S\subseteq K$,
\item $\ddot{S}=\{z\in\widehat{S}:z\text{ is infinite}\}$ whenever $S\subseteq K$,
\item $Z_e=\{z\in Z:z\cap e=\varnothing\}$ whenever $e\in\Fin(\Omega)$.
\end{itemize}
Observe that both $K$ and $\widehat{K}$ are compact and crowded. Furthermore, it is easy to check that $\ddot{K}\in\bP^0_2(Z)$, and it is clear that each $Z_e\in\bD^0_1(Z)$. We will use $\cl$ to denote closure in $Z$. Fix $X\subseteq K$ such that $X\wc=\bG(K)$, then set $X'=X\cup (K\setminus\cl(X))$.

\noindent\textbf{Claim 1.} $X'\wc=\bG(K)$.

\noindent\textit{Proof of Claim 1.} Notice that $\bG(K)$ is a good Wadge class by Lemma \ref{lemma_good_relativization}. Therefore, using Lemma \ref{lemma_closure_closed_open}, one sees that $X'\in\bG(K)$. An application of Lemma \ref{lemma_closed_subset} concludes the proof. $\blacksquare$

By Claim 1, we can assume without loss of generality that $X$ is dense in $K$. Notice that $\bG(Z)=X\wc$ by Corollary \ref{corollary_characterization_exact_complexity}. Furthermore, Lemma \ref{lemma_good_relativization} shows that $\bG(Z)$ is a good Wadge class. Set $Y=\widehat{X}$. Throughout the remainder of this proof, we will use $\leq$ exclusively to denote Wadge-reduction in $Z$.

\noindent\textbf{Claim 2.} $Y\in\bG(Z)$.

\noindent\textit{Proof of Claim 2.} Consider the unique function $\varphi:\ddot{K}\longrightarrow K$ such that $x\subseteq\varphi(x)$ for every $x\in\ddot{K}$, and observe that $\varphi$ is continuous. By Lemma \ref{lemma_relativization_basic}, it follows that $\ddot{X}=\varphi^{-1}[X]\in\bG(\ddot{K})$. Therefore, by Lemma \ref{lemma_relativization_subspace}, there exists $\widetilde{X}\in\bG(Z)$ such that $\ddot{X}=\widetilde{X}\cap\ddot{K}$. Finally, Corollary \ref{corollary_closure_good} ensures that
$$
Y=(\widetilde{X}\cap\ddot{K})\cup(\Fin(\Omega)\cap Y)\in\bG(Z),
$$
as desired. $\blacksquare$

Define
$$
\RR=\{y\cup e:y\in Y\text{ and }e\in\Fin(\Omega)\},
$$
and observe that $\RR$ is a semiideal on $\Omega$. To conclude the proof, it will be enough to show that $\RR$ has exact complexity $\bG$ and that $\RR$ is meager. The first statement will follow from Claim 3 and Corollary \ref{corollary_characterization_exact_complexity}, while the second statement is the content of Claim 4.

\noindent\textbf{Claim 3.} $\RR\wc=\bG(Z)$.

\noindent\textit{Proof of Claim 3.} Define
$$
\RR_e=\RR\cap\widehat{K}\cap Z_e
$$
for $e\in\Fin(\Omega)$, and observe that each $\RR_e$ is a closed subset of $\RR$. It is easy to realize that $Y\cap\Fin(\Omega)=\widehat{K}\cap\Fin(\Omega)$ by the density of $X$ in $K$. Using this fact, one sees that each $\RR_e=Y\cap Z_e$. Therefore, each $\RR_e\in\bG(Z)$ by Claim 2 and Lemma \ref{lemma_closure_clopen}.

Define $\psi_e:Z\longrightarrow Z$ for $e\in\Fin(\Omega)$ by setting $\psi_e(x)=x\cup e$ for $x\in Z$. It is clear that each $\psi_e$ is continuous, hence closed by compactness. Since each $\psi_e\re Z_e$ is injective, it is an embedding by compactness.  In particular, each $\psi_e[\RR_e]\approx\RR_e$. Furthermore, each $\psi_e\re\RR:\RR\longrightarrow\RR$ is closed by Lemma \ref{lemma_closed_function}. Let $\{A_n:n\in\omega\}$ be an enumeration of
$$
\{\psi_e[\RR_e]:e\in\Fin(\Omega)\}.
$$
Observe that each $A_n\in\bG(Z)$ by Proposition \ref{proposition_characterization_complexity}, and that each $A_n$ is a closed subset of $\RR$. Furthermore, it is clear that $\RR=\bigcup_{n\in\omega}A_n$. Since $\bG(Z)$ is closed under $\SU_1$ by Lemma \ref{lemma_preservation_closure_su_under_relativization}, it follows from Lemma \ref{lemma_closed_decomposition} that $\RR\in\bG(Z)$. To complete the proof, simply observe that $\RR\cap K=X$, then apply Lemmas \ref{lemma_closure_closed_open} and \ref{lemma_closed_subset}. $\blacksquare$

\noindent\textbf{Claim 4.} $\RR$ is meager.

\noindent\textit{Proof of Claim 4.} Notice that $\RR$ is either meager or comeager by Proposition \ref{proposition_finite_modifications_meager_comeager}. Assume, in order to get a contradiction, that $\RR$ is comeager. Then, as in the proof of Corollary \ref{corollary_filters_meager}, there exist $x,y\in\RR$ such that $x\cup y=\Omega$. It is easy to realize that this contradicts the fact that $\widehat{K}$ has the finite union property. $\blacksquare$
\end{proof}

\begin{corollary}\label{corollary_meager_semifilter}
Let $\bS$ be a nice topological pointclass, and assume that $\Det(\bS(\omega^\omega))$ holds. Let $\bG\subseteq\bS(\omega^\omega)$ be a good Wadge class in $\omega^\omega$ that satisfies at least one of the following conditions:
\begin{enumerate}
\item $\ell(\bG)\geq 2$,
\item $\ell(\bG)=1$ and $\bG$ does not have the separation property,
\item $\ell(\bG)=1$ and $t(\bG)=3$.
\end{enumerate}
Then there exists a meager semifilter of exact complexity $\bG$.
\end{corollary}
\begin{proof}
By Theorem \ref{theorem_meager_semifilter}, it will be enough to show that $\bG$ is closed under $\SU_1$. If condition $(1)$ holds, this follows from the definition of level, since $\SU_1(\bG)\subseteq\PU_2(\bG)$. If condition $(2)$ holds, apply Theorem \ref{theorem_closure_main_separation}. Finally, if condition $(3)$ holds, apply Theorem \ref{theorem_closure_main_type}.
\end{proof}

\begin{corollary}\label{corollary_baire_semifilter}
Let $\bS$ be a nice topological pointclass, and assume that $\Det(\bool\bS(\omega^\omega))$ holds. Let $\bG\subseteq\bS(\omega^\omega)$ be a good Wadge class in $\omega^\omega$ that satisfies at least one of the following conditions:
\begin{enumerate}
\item $\ell(\bG)\geq 2$,
\item $\ell(\bG)=1$ and $\bG$ has the separation property,
\item $\ell(\bG)=1$ and $t(\bG)=3$.
\end{enumerate}
Then there exists a Baire semifilter of exact complexity $\bG$.
\end{corollary}
\begin{proof}
The key observation is that Corollary \ref{corollary_meager_semifilter} applies to $\bGc$. This is straightforward to check, except for condition $(2)$, which requires Theorem \ref{theorem_separation_generalized}. So there exists a meager semifilter $\Ss$ of exact complexity $\bGc$. Using Corollary \ref{corollary_characterization_exact_complexity}, it is easy to realize that $\TT=c[2^\omega\setminus\Ss]$ is a Baire semifilter of exact complexity $\bG$.
\end{proof}

\subsection{Existence: the negative results}\label{subsection_homogeneous_existence_negative}

Next, we will prove that the results of \S\ref{subsection_homogeneous_existence_positive} are sharp. More precisely, as Theorem \ref{theorem_existence_negative_meager} and Corollary \ref{corollary_existence_negative_baire} show, there are no other zero-dimensional homogeneous spaces whose exact complexity is a good Wadge class. Theorem \ref{theorem_existence_negative_meager} is based on \cite[Lemma 4.3.7]{van_engelen_thesis}.

\begin{theorem}\label{theorem_existence_negative_meager}
Let $\bS$ be a nice topological pointclass, and assume that $\Det(\bS(\omega^\omega))$ holds. Let $\bG\subseteq\bS(\omega^\omega)$ be a good Wadge class in $\omega^\omega$ that satisfies the following conditions:
\begin{itemize}
\item $\ell(\bG)=1$,
\item $t(\bG)\in\{1,2\}$,
\item $\bG$ does not have the separation property.
\end{itemize}
Then every zero-dimensional homogeneous space of exact complexity $\bG$ is meager.
\end{theorem}
\begin{proof}
Let $X$ be a zero-dimensional homogeneous space of exact complexity $\bG$. Assume without loss of generality that $X$ is a dense subspace of $2^\omega$. Notice that $X\wc=\bG(2^\omega)$ by Corollary \ref{corollary_characterization_exact_complexity}. By Lemma \ref{lemma_su_dual_has_separation} and Theorem \ref{theorem_characterization_small_type}, it is possible to find $\bG_n\in\NSD(2^\omega)$ for $n\in\omega$ such that each $\ell(\bG_n)\geq 1$ and $\bG(2^\omega)=\SU_1(\bD)$, where
$$
\bD=\bigcup_{n\in\omega}(\bG_n\cup\bGc_n).
$$

Fix $A_n\in\bD$ and pairwise disjoint $U_n\in\bS^0_2(2^\omega)$ for $n\in\omega$ such that
$$
X=\bigcup_{n\in\omega}(A_n\cap U_n).
$$
Let $C_{n,k}\in\bP^0_1(2^\omega)$ for $n,k\in\omega$ be such that each $U_n=\bigcup_{k\in\omega}C_{n,k}$. Observe that each $A_n\cap C_{n,k}\in\bD$ by Lemma \ref{lemma_closure_level}.

To conclude the proof, it will be enough to show that each $A_n\cap C_{n,k}=X\cap C_{n,k}$ is closed nowhere dense in $X$. Assume, in order to get a contradiction, that $X\cap V\subseteq X\cap C_{n,k}$ for some non-empty $V\in\bD^0_1(2^\omega)$ and $n,k\in\omega$. Using Lemma \ref{lemma_closure_clopen}, one sees that $X\cap V=A_n\cap C_{n,k}\cap V\in\bD$. Notice that $\bD\subseteq\bGc(2^\omega)$ because $\bD\subseteq\bG(2^\omega)$ and $\bD$ is selfdual. It follows from Proposition \ref{proposition_characterization_complexity} that $X\cap V$ is a $\bGc$ space, hence $X$ is a $\bGc$ space by Lemma \ref{lemma_open_implies_whole}. This contradicts Corollary \ref{corollary_characterization_exact_complexity}.
\end{proof}

\begin{corollary}\label{corollary_existence_negative_baire}
Let $\bS$ be a nice topological pointclass, and assume that $\Det(\bS(\omega^\omega))$ holds. Let $\bG\subseteq\bS(\omega^\omega)$ be a good Wadge class in $\omega^\omega$ that satisfies the following conditions:
\begin{itemize}
\item $\ell(\bG)=1$,
\item $t(\bG)\in\{1,2\}$,
\item $\bG$ has the separation property.
\end{itemize}
Then every zero-dimensional homogeneous space of exact complexity $\bG$ is Baire.
\end{corollary}
\begin{proof}
Let $X$ be a zero-dimensional homogeneous space of exact complexity $\bG$. Assume without loss of generality that $X$ is a dense subspace of $2^\omega$. Set $Y=2^\omega\setminus X$, and observe that $Y$ is homogeneous by Corollaries \ref{corollary_strongly_homogeneous} and \ref{corollary_h-homogeneously_embedded}. Furthermore, using Corollary \ref{corollary_characterization_exact_complexity}, one sees that $Y$ has exact complexity $\bGc$. Since $\bGc$ does not have the separation property by Theorem \ref{theorem_separation_generalized}, it follows from Theorem \ref{theorem_existence_negative_meager} that $Y$ is a meager space, hence a meager subset of $2^\omega$ by density. In conclusion, being a comeager subspace of $2^\omega$, the space $X$ is Baire.
\end{proof}

\subsection{The classification of the zero-dimensional homogeneous spaces}\label{subsection_homogeneous_classification}

As the following theorem shows, the previous three subsections can be combined to yield a complete classification of the zero-dimensional homogeneous spaces of sufficiently high complexity under $\AD$. This classification is complete because, by Theorem \ref{theorem_homogeneous_spaces_have_good_complexity}, a zero-dimensional homogeneous space that is not $\bD(\Diff_\omega(\bS^0_2))$ has exact complexity $\bG$ for some good Wadge class $\bG$. The case $\bS=\Borel$ is due to van Engelen (see \cite[Theorem 4.4.4]{van_engelen_thesis}), although his results are expressed using the more complicated language of \cite{louveau_article}.

\begin{theorem}\label{theorem_classification_homogeneous}
Let $\bS$ be a nice topological pointclass, and assume that $\Det(\bS(\omega^\omega))$ holds. Let $\bG\subseteq\bS(\omega^\omega)$ be a good Wadge class in $\omega^\omega$.
\begin{itemize}
\item If either $\ell(\bG)\geq 2$ or $(\ell(\bG)=1\text{ and }t(\bG)=3)$  then, up to homeomorphism, there exist exactly two zero-dimensional homogeneous spaces of exact complexity $\bG$. One of these spaces is Baire, and the other is meager.
\item If $\ell(\bG)=1$, $t(\bG)\in\{1,2\}$ and $\bG$ does not have the separation property then, up to homeomorphism, there exists exactly one zero-dimensional homogeneous space of exact complexity $\bG$. This space is meager.
\item If $\ell(\bG)=1$, $t(\bG)\in\{1,2\}$ and $\bG$ has the separation property then, up to homeomorphism, there exists exactly one zero-dimensional homogeneous space of exact complexity $\bG$. This space is Baire.
\end{itemize}
\end{theorem}
\begin{proof}
The existence claims follow from Corollaries \ref{corollary_meager_semifilter} and \ref{corollary_baire_semifilter}, since every semifilter is homogeneous by Corollary \ref{corollary_semifilter_homogeneous}. The constraints on Baire category follow from Theorem \ref{theorem_existence_negative_meager} and Corollary \ref{corollary_existence_negative_baire}. Finally, the claims about uniqueness up to homeomorphism follow from Proposition \ref{proposition_homogeneous_dichotomy} and Theorem \ref{theorem_uniqueness_embeddings}.
\end{proof}

\subsection{A purely topological characterization of semifilters}\label{subsection_homogeneous_characterization}

The following ``self-strengthening'' result is a natural byproduct of our classification. In particular, the equivalence $(1)\leftrightarrow (3)$ can be viewed as a purely topological characterization of semifilters. In \S\ref{subsection_filters_characterization}, we will give an analogous characterization of filters. Theorem \ref{theorem_characterization_semifilters} simultaneously generalizes \cite[Corollary 4.4.6]{van_engelen_thesis} (which gives $(1)\leftrightarrow (2)$ in the case $\bS=\Borel$), \cite[Theorem 1.1]{medini_semifilters} (which gives $(1)\leftrightarrow (3)$ in the case $\bS=\Borel$), and \cite[Theorem 1.1]{carroy_medini_muller_homogeneous} (which gives $(1)\leftrightarrow (2)$ in the case $\bS=\PP$). Furthermore, it answers \cite[Question 13.4]{medini_semifilters} in the affirmative.

\begin{theorem}\label{theorem_characterization_semifilters}
Let $\bS$ be a nice topological pointclass, and assume that $\Det(\bS(\omega^\omega))$ holds. Assume that $X$ is a zero-dimensional $\bS$ space that is not locally compact. Then the following conditions are equivalent:
\begin{enumerate}
\item $X$ is homogeneous,
\item $X$ is strongly homogeneous,
\item $X$ is homeomorphic to a semifilter.
\end{enumerate}
\end{theorem}
\begin{proof}
The implication $(2)\rightarrow (1)$ is given by Corollary \ref{corollary_strongly_homogeneous}, while the implication $(3)\rightarrow (1)$ is given by Corollary \ref{corollary_semifilter_homogeneous}. Next, assume that $X$ is a $\bD(\Diff_\omega(\bS^0_2))$ space. For the implication $(1)\rightarrow (2)$, see \cite[Corollary 8.4]{van_engelen_ambiguous} or \cite[Theorems 3.4.13 and 3.6.2]{van_engelen_thesis}. For the implication $(1)\rightarrow (3)$, see \cite[Theorem 8.4]{medini_semifilters}.

From now on, assume that $(1)$ holds and that $X$ is not a $\bD(\Diff_\omega(\bS^0_2))$ space. By Theorem \ref{theorem_homogeneous_spaces_have_good_complexity}, there exists a good Wadge class $\bG\subseteq\bS(\omega^\omega)$ such that $X$ has exact complexity $\bG$. It follows from Corollary \ref{corollary_h-homogeneously_embedded} that $(2)$ holds. Finally, we will show that $(3)$ holds. If $X$ is meager then, using Corollary \ref{corollary_existence_negative_baire}, one sees that $\bG$ must satisfy the requirement of Corollary \ref{corollary_meager_semifilter}, which guarantees the existence of a meager semifilter of exact complexity $\bG$. By a similar argument, if $X$ is Baire then Corollary \ref{corollary_baire_semifilter} and Theorem \ref{theorem_existence_negative_meager} yield a Baire semifilter of exact complexity $\bG$. In either case, an application of Theorem \ref{theorem_uniqueness_embeddings} will conclude the proof.
\end{proof}

\section{Filters}\label{section_filters}

\subsection{Closure under squares}\label{subsection_filters_squares}

We begin by discussing a useful closure property, which first appeared in \cite{van_engelen_ideals} in the case $Z=2^\omega$, under the name of ``closure under products.''

\begin{definition}
Let $Z$ be a space, let $2\leq n<\omega$, and let $\bG\subseteq\PP(\omega^\omega)$. We will say that $\bG(Z)$ is \emph{closed under $n$-th powers} if $A^n\in\bG(Z^n)$ for every $A\in\bG(Z)$. When $n=2$, we will simply say that $\bG(Z)$ is \emph{closed under squares}.
\end{definition}

\begin{proposition}\label{proposition_characterization_squares}
Let $\bS$ be a nice topological pointclass, and assume that $\Det(\bS(\omega^\omega))$ holds. Let $\bG\in\NSDS(\omega^\omega)$. Then the following conditions are equivalent:
\begin{enumerate}
\item There exists an uncountable zero-dimensional Polish space $Z$ such that $\bG(Z)$ is closed under squares,
\item $\bG(Z)$ is closed under squares for every zero-dimensional Polish space $Z$,
\item There exists an uncountable zero-dimensional Polish space $Z$ such that $\bG(Z)$ is closed under $n$-th powers whenever $2\leq n<\omega$,
\item $\bG(Z)$ is closed under $n$-th powers whenever $Z$ is a zero-dimensional Polish space $Z$ and $2\leq n<\omega$.
\end{enumerate}
\end{proposition}
\begin{proof}
To see that $(1)\rightarrow (2)$, fix an uncountable zero-dimensional Polish space $Z$ such that $\bG(Z)$ is closed under squares. Let $W$ be a zero-dimensional Polish space and pick $A\in\bG(W)$. Since $Z$ is uncountable, we can assume without loss of generality that $W$ is a subspace of $Z$. By Lemma \ref{lemma_relativization_subspace}, there exists $\widetilde{A}\in\bG(Z)$ such that $A=\widetilde{A}\cap W$. Since $A\times A=(\widetilde{A}\times\widetilde{A})\cap (W\times W)$ and $\widetilde{A}\times\widetilde{A}\in\bG(Z)$, a further application of Lemma \ref{lemma_relativization_subspace} shows that $A\times A\in\bG(W)$, as desired.

The proof of the implication $(3)\rightarrow (4)$ is similar to the one that we have just given, so we will omit it. The implications $(2)\rightarrow (1)$ and $(4)\rightarrow (3)\rightarrow (2)$ are trivial. To conclude the proof, we will show that $(2)\rightarrow (4)$. So assume that $(2)$ holds. Pick a zero-dimensional Polish space $Z$, pick $A\in\bG(Z)$ and $2\leq n<\omega$. Assume without loss of generality that $A\neq\varnothing$, and fix $x\in A$. Using $(2)$, it is easy to realize that $A^m\in\bG(Z^m)$ whenever $m=2^k$ for some $k\in\omega$. Now fix $k\in\omega$ such that $2^k\geq n$, then set $m=2^k$. Define
$$
W=\{z\in Z^m:z(i)=x\text{ whenever }n\leq i<m\},
$$
and observe that $A^m\cap W\in\bG(W)$ by Lemma \ref{lemma_relativization_subspace}. Define $h:Z^n\longrightarrow W$ by setting $h(z_0,\ldots,z_{n-1})=(z_0,\ldots,z_{n-1},x,\ldots,x)$, and observe that $h$ is a homeomorphism. It follows that
$$
A^n=h^{-1}[A^m\cap W]\in\bG(Z^n)
$$
by Lemma \ref{lemma_relativization_basic}, as desired.
\end{proof}

\subsection{The topological classification of filters}\label{subsection_filters_classification}

In this subsection, we will give the promised classification of filters up to homeomorphism (see Theorem \ref{theorem_classification_filters}). Naturally, this classification will build on our previous results. The new existence result that we will need is Theorem \ref{theorem_existence_filter}, whose proof is taken almost verbatim from the proof of \cite[Lemma 3.3]{van_engelen_ideals}. However, as in the statement of Theorem \ref{theorem_meager_semifilter}, we have isolated more clearly the required closure properties. More importantly, in the proof of Theorem \ref{theorem_classification_filters}, we will use the results of \S\ref{subsection_wadge_closure_main}.

\begin{theorem}\label{theorem_existence_filter}
Let $\bS$ be a nice topological pointclass, and assume that $\Det(\bS(\omega^\omega))$ holds. Let $\bG\subseteq\bS(\omega^\omega)$ be a good Wadge class in $\omega^\omega$ that is closed under squares and closed under $\SU_1$. Then there exists a filter of exact complexity $\bG$.
\end{theorem}

\begin{proof}
The setup of this proof is exactly the same as the setup of the proof of Theorem \ref{theorem_meager_semifilter}. We repeat it here for convenience, but we will not be as thorough with some of the subsequent details. Recall the following definitions:
\begin{itemize}
\item $\Omega=2^{<\omega}$,
\item $Z=2^\Omega$,
\item $K=\{\{x\re n:n\in\omega\}:x\in 2^\omega\}$,
\item $\widehat{S}=\{z\in Z:z\subseteq z'\text{ for some }z'\in S\}$ whenever $S\subseteq K$,
\item $\ddot{S}=\{z\in\widehat{S}:z\text{ is infinite}\}$ whenever $S\subseteq K$,
\item $Z_e=\{z\in Z:z\cap e=\varnothing\}$ whenever $e\in\Fin(\Omega)$.
\end{itemize}
Fix $X\subseteq K$ such that $\bG(K)=X\wc$, and assume without loss of generality that $X$ is dense in $K$. Notice that $\bG(Z)=X\wc$, and that $\bG(Z)$ is a good Wadge class. Set $Y=\widehat{X}$. For the remainder of this proof, we will use $\leq$ exclusively to denote Wadge-reduction in $Z$. 

\noindent\textbf{Claim 1.} $Y\in\bG(Z)$.

\noindent\textit{Proof of Claim 1.} This is proved like Claim 2 in the proof of Theorem \ref{theorem_meager_semifilter}. $\blacksquare$

Given $1\leq n<\omega$, $s\in\Omega$ and $\vec{s}\in\Omega^n$, we will use the following notation:
\begin{itemize}
\item Define $\phi_n:(2^\Omega)^n\longrightarrow 2^\Omega$ by setting $\phi_n(z_0,\ldots, z_{n-1})=z_0\cup\cdots\cup z_{n-1}$,
\item $M_n=\{\vec{s}\in\Omega^n:|\vec{s}(i)|=|\vec{s}(j)|\text{ for all }i,j<n\text{ and }\vec{s}(i)\neq\vec{s}(j)\text{ if }i\neq j\}$,
\item $F_s=\{z\in\widehat{K}: s\subseteq t\text{ for every }t\in z\}$,
\item $F_{\vec{s}}=\prod_{i< n}F_{\vec{s}(i)}$.
\end{itemize}
Observe that each $\phi_n$ is continuous, and that each $F_s$ is compact. Set
$$
Y'=\bigcup_{1\leq n<\omega}\bigcup_{\vec{s}\in M_n}\phi_n[F_{\vec{s}}\cap Y^n],
$$
then define
$$
\II=\{y\cup e:y\in Y'\text{ and }e\in\Fin(\Omega)\}.
$$

\noindent\textbf{Claim 2.} $\II$ is an ideal on $\Omega$.

\noindent\textit{Proof of Claim 2.} It is trivial to check that $\varnothing\in\II$, while $\Omega\notin\II$ because $K$ has the finite union property. Next, we will show that $\II$ is downward-closed. Define $\theta:\ddot{K}\longrightarrow 2^\omega$ by setting $\theta(z)=\bigcup z$ for $z\in\ddot{K}$. Pick
$$
z=y_0\cup\cdots\cup y_{n-1}\cup e,
$$
where $1\leq n<\omega$, $(y_0,\ldots, y_{n-1})\in F_{\vec{s}}\cap Y^n$ for some $\vec{s}\in M_n$, and $e\in\Fin(\Omega)$. Let $z'\subseteq z$, and assume without loss of generality that $z'\notin\Fin(\Omega)$. Write
$$
z'=y'_0\cup\cdots\cup y'_{m-1}\cup e',
$$
where $1\leq m\leq n$, each $y'_i\in\ddot{X}$, $\theta(y'_i)\neq\theta(y'_j)$ whenever $i\neq j$, and $e'\in\Fin(\Omega)$. In particular, it is possible to find $k\in\omega$ big enough so that $\theta(y_i)\re k\neq\theta(y_j)\re k$ whenever $i\neq j$. Define $\vec{t}:m\longrightarrow\Omega$ by setting $\vec{t}(i)=\theta(y_i)\re k$ for $i<m$, and observe that $\vec{t}\in M_m$. Notice that setting $y''_i=y'_i\setminus 2^{<k}$ for $i<m$ yields $(y''_0,\ldots, y''_{m-1})\in F_{\vec{t}}\cap Y^n$. Also set $e''=e'\cup\bigcup_{i<m}(y'_i\cap 2^{<k})$. It is clear that 
$$
z'=y''_0\cup\cdots\cup y''_{m-1}\cup e'',
$$
which shows that $z'\in\II$, as desired. Using a similar argument, one sees that $\II$ is closed under finite unions. $\blacksquare$

By Claim 2, to conclude the proof, it will be enough to show that $\II$ has exact complexity $\bG$. This will follow from Claim 5 and Corollary \ref{corollary_characterization_exact_complexity}. The following two claims will be used in the proof of Claim 5.

\noindent\textbf{Claim 3.} $\phi_n[F_{\vec{s}}\cap Y^n]$ is closed in $\II$ whenever $1\leq n<\omega$ and $\vec{s}\in M_n$.

\noindent\textit{Proof of Claim 3.} Let $1\leq n<\omega$ and $\vec{s}\in M_n$. It will be enough to show that
$$
\phi_n[F_{\vec{s}}\cap Y^n]=\phi_n[F_{\vec{s}}]\cap\II.
$$
The inclusion $\subseteq$ is clear by the definitions of $Y'$ and $\II$. In order to prove the other inclusion, pick $\vec{x}\in F_{\vec{s}}$ such that $\bigcup_{i<n}\vec{x}(i)\in\II$. We will show that each $\vec{x}(i)\in Y$. So pick $i<n$. Since $Y\cap\Fin(\Omega)=\widehat{K}\cap\Fin(\Omega)$ by the density of $X$ in $K$, we can assume without loss of generality that $\vec{x}(i)$ is infinite. Let $1\leq m<\omega$, $\vec{t}\in M_m$, $\vec{y}\in F_{\vec{t}}\cap Y^n$ and $e\in\Fin(\Omega)$ be such that
$$
\bigcup_{j<n}\vec{x}(j)=\bigcup_{j<m}\vec{y}(j)\cup e.
$$
Pick $j<m$ such that $\vec{x}(i)\cap\vec{y}(j)$ is infinite. Since $\vec{x}(i),\vec{y}(j)\in\widehat{K}$, this means that there exists $z\in K$ such that $\vec{x}(i),\vec{y}(j)\subseteq z$. However, the fact that $\vec{y}(j)\in Y$ implies that $z\in X$. Hence $\vec{x}(i)\in Y$ as well. $\blacksquare$

\noindent\textbf{Claim 4.} $\phi_n[F_{\vec{s}}\cap Y^n]\in\bG(Z)$ whenever $1\leq n<\omega$ and $\vec{s}\in M_n$.

\noindent\textit{Proof of Claim 4.} Begin by observing that $\phi_n\re F_{\vec{s}}$ is injective whenever $1\leq n<\omega$ and $\vec{s}\in M_n$. Therefore, each $\phi_n\re F_{\vec{s}}$ is an embedding by compactness. Notice that $Y^n\in\bG(Z^n)$ whenever $1\leq n<\omega$ by Claim 1 and Proposition \ref{proposition_characterization_squares}, since $\bG$ is closed under squares. Furthermore, each $\bG(Z^n)$ is a good Wadge class by Lemma \ref{lemma_good_relativization}. It follows that each $F_{\vec{s}}\cap Y^n\in\bG(Z^n)$ by Lemma \ref{lemma_closure_closed_open}, hence each $\phi_n[F_{\vec{s}}\cap Y^n]\in\bG(Z)$ by Proposition \ref{proposition_characterization_complexity}.

\noindent\textbf{Claim 5.} $\II\wc=\bG(Z)$.

\noindent\textit{Proof of Claim 5.} Define $\psi_e:Z\longrightarrow Z$ for $e\in\Fin(\Omega)$ by setting $\psi_e(x)=x\cup e$ for $x\in Z$. It is clear that each $\psi_e$ is continuous, hence closed by compactness. Since each $\psi_e\re Z_e$ is injective, it is an embedding by compactness. Furthermore, each $\psi_e\re\II:\II\longrightarrow\II$ is closed by Lemma \ref{lemma_closed_function}. Let $\{A_k:k\in\omega\}$ be an enumeration of
$$
\{\psi_e[\phi_n[F_{\vec{s}}\cap Y^n]\cap Z_e]:1\leq n<\omega,\vec{s}\in M_n\text{ and }e\in\Fin(\Omega)\}.
$$
Using Claim 4, Lemma \ref{lemma_closure_clopen} and Proposition \ref{proposition_characterization_complexity}, one sees that each $A_k\in\bG(Z)$. Furthermore, each $A_k\in\bP^0_1(\II)$ by Claim 3. Since $\II=\bigcup_{k\in\omega}A_k$, it follows from Lemmas \ref{lemma_preservation_closure_su_under_relativization} and \ref{lemma_closed_decomposition} that $\II\in\bG(Z)$. To complete the proof, simply observe that $\II\cap K=X$, then apply Lemmas \ref{lemma_closure_closed_open} and \ref{lemma_closed_subset}. $\blacksquare$
\end{proof}

Now we have all the tools needed to completely classify filters up to homeomorphism. For the classification of filters that are $\Delta(\Diff_\omega(\bS^0_2))$ spaces, see \cite[Lemmas 2.9 and 3.3]{van_engelen_ideals}. On the other hand, by Theorem \ref{theorem_homogeneous_spaces_have_good_complexity}, the following result takes care of every filter that is not a $\Delta(\Diff_\omega(\bS^0_2))$ space. We also remark that the filter of exact complexity $\bG$ mentioned below is unique up to homeomorphism by Corollary \ref{corollary_filters_meager} and Theorem \ref{theorem_uniqueness_embeddings}.

\begin{theorem}\label{theorem_classification_filters}
Let $\bS$ be a nice topological pointclass, and assume that $\Det(\bS(\omega^\omega))$ holds. Let $\bG\subseteq\bS(\omega^\omega)$ be a good Wadge class in $\omega^\omega$. Then the following conditions are equivalent:
\begin{enumerate}
\item There exists a filter of exact complexity $\bG$,
\item $\bG$ is closed under squares and either $\ell(\bG)\geq 2$, $(\ell(\bG)=1\text{ and }t(\bG)=3)$, or $(\ell(\bG)=1\text{ and }\bG\text{ does not have the separation property})$.
\end{enumerate}
\end{theorem}
\begin{proof}
In order to prove the implication $(1)\rightarrow (2)$, let $\FF$ be a filter of exact complexity $\bG$. Notice that $\FF$ is meager by Corollary \ref{corollary_filters_meager}. Therefore, by Theorem \ref{theorem_classification_homogeneous}, it will be enough to show that $\bG$ is closed under squares. By Proposition \ref{proposition_characterization_squares}, it will be enough to show that $\bG(2^\omega)$ is closed under squares. So pick $A\in\bG(2^\omega)$, and observe that $A\leq\FF$ in $2^\omega$ by Corollary \ref{corollary_characterization_exact_complexity}. This trivially implies that $A\times A\leq\FF\times\FF$ in $2^\omega\times 2^\omega$. Since $\FF\times\FF\in\bG(2^\omega\times 2^\omega)$ by Theorem \ref{theorem_filters_square} and Proposition \ref{proposition_characterization_complexity}, it follows that $A\times A\in\bG(2^\omega\times 2^\omega)$, as desired. Now assume that $\bG$ satisfies condition $(2)$. As in the proof of Corollary \ref{corollary_meager_semifilter}, one sees that $\bG$ is closed under $\SU_1$. Therefore, condition $(1)$ holds by Theorem \ref{theorem_existence_filter}.
\end{proof}

\subsection{A purely topological characterization of filters}\label{subsection_filters_characterization}

As we discussed in \S\ref{subsection_preliminaries_filters} and \S\ref{subsection_preliminaries_homogeneity}, the combinatorial structure of filters imposes strong constraints on their topology. But is it possible to go the other way? In other words, are there natural conditions on a space $X$ which will guarantee that $X$ is homeomorphic to a filter? Van Engelen showed that, in the Borel context, this question has a very elegant positive answer (see \cite[Theorem 3.4]{van_engelen_ideals}). The following result shows that, under suitable determinacy assumptions, his characterization extends beyond the Borel realm.

\begin{theorem}\label{theorem_characterization_filters}
Let $\bS$ be a nice topological pointclass, and assume that $\Det(\bS(\omega^\omega))$ holds. Let $X$ be a zero-dimensional $\bS$ space that is not locally compact. Then the following conditions are equivalent:
\begin{enumerate}
\item $X$ is homeomorphic to a filter,
\item $X$ is homogeneous, meager, and homeomorphic to $X^2$.
\end{enumerate}
\end{theorem}
\begin{proof}
The implication $(1)\rightarrow (2)$ follows from the discussion in \S\ref{subsection_preliminaries_filters}, Theorem \ref{theorem_filters_square}, and Corollary \ref{corollary_filters_meager}. In order to prove the implication $(2)\rightarrow (1)$, assume that condition $(2)$ holds. If $X$ is a $\Delta(\Diff_\omega(\bS^0_2))$ space, then the desired result follows from \cite[Lemma 3.3]{van_engelen_ideals}, so assume that this is not the case. By Theorem \ref{theorem_homogeneous_spaces_have_good_complexity}, there exists a good Wadge class $\bG\in\NSDS(\omega^\omega)$ such that $X$ has exact complexity $\bG$. By Theorems \ref{theorem_classification_homogeneous} and \ref{theorem_classification_filters}, it will be enough to show that $\bG$ is closed under squares. To see this, proceed as in the proof of Theorem \ref{theorem_classification_filters}.
\end{proof}

\subsection{Filters versus topological groups}\label{subsection_filters_vs_groups}

As Theorem \ref{theorem_characterization_semifilters} shows, the lower-right portion of the diagram from \S\ref{subsection_preliminaries_homogeneity} ``collapses'' under $\AD$. Therefore, it seems natural to ask whether the rest of the diagram could collapse as well. More precisely, one might ask: does $\AD$ imply that every zero-dimensional topological group is homemorphic to a filter?

As usual, the trivial case of locally compact spaces provides several counterexamples (see the spaces mentioned by Proposition \ref{proposition_locally_compact}). And even after excluding these spaces, the question remains too naive, as $\omega^\omega$ is a zero-dimensional topological group that is not homeomorphic to a filter (by Corollary \ref{corollary_filters_meager}). However, adding the requirement that the topological group is meager yields an open question. In fact, even the Borel version of this problem is still open. It was first formulated at the very end of \cite{van_engelen_low}, where van Engelen solved the problem for spaces of low complexity (see \cite[Theorem 3.4]{van_engelen_low}). The precise statements are as follows.\footnote{\,We added the assumption ``not locally compact'' in the statement of Theorem \ref{theorem_groups_low_complexity} because our definition of filter is slightly more restrictive than that of van Engelen. However, this does not result in any loss of generality (see \cite[Proposition 2]{medini_zdomskyy_between}).}

\begin{theorem}[van Engelen]\label{theorem_groups_low_complexity}
Let $X$ be a zero-dimensional $\bD^0_3$ space that is not locally compact. Then the following conditions are equivalent:
\begin{itemize}
\item $X$ is a meager topological group,
\item $X$ is homeomorphic to a filter.
\end{itemize}
\end{theorem}

\begin{conjecture}[van Engelen]\label{conjecture_van_engelen}
Theorem \ref{theorem_groups_low_complexity} holds for all zero-dimensional Borel spaces that are not locally compact.
\end{conjecture}

While we were unable to make any progress on the above conjecture, we suspect that if it were true then the following stronger conjecture would be true as well.

\begin{conjecture}\label{conjecture_generalized_van_engelen}
Let $\bS$ be a nice topological pointclass, and assume that $\Det(\bS(\omega^\omega))$ holds. Then Theorem \ref{theorem_groups_low_complexity} holds for all zero-dimensional $\bS$ spaces that are not locally compact.
\end{conjecture}

We remark that a proof of Conjecture \ref{conjecture_generalized_van_engelen} would yield a complete classification of the zero-dimensional topological groups under $\AD$. In fact, by Proposition \ref{proposition_homogeneous_dichotomy} and the following well-known result, restricting the attention to meager topological groups does not result in any loss of generality. Therefore, the desired classification would follow from Theorem \ref{theorem_classification_filters} and the remarks that precede it.

\begin{proposition}\label{proposition_baire_group}
Let $\bS$ be a nice topological pointclass, and assume that $\Det(\bS(\omega^\omega))$ holds. Let $X$ be a zero-dimensional $\bS$ space that is not locally compact. If $X$ is a Baire topological group then $X\approx\omega^\omega$.
\end{proposition}
\begin{proof}
Assume that $X$ is a Baire topological group. Since $X$ is homogeneous and not locally compact, it is clear that no non-empty open subspace of $X$ is compact. Therefore, by \cite[Theorem 7.7]{kechris}, it will be enough to show that $X$ is Polish.

By \cite[Theorem 3.6.10, Proposition 3.6.20 and Proposition 4.3.8]{arhangelskii_tkachenko}, it is possible to find a Polish topological group $Z$ such that $X$ is a dense topological subgroup of $Z$. Assume, in order to get a contradiction, that $X\subsetneq Z$. Pick $z\in Z\setminus X$, and observe that the cosets $X$ and $zX$ are disjoint. On the other hand, Corollary \ref{corollary_baire_dense_subspace} shows that both $X$ and $zX$ are comeager in $Z$, contradicting the fact that $Z$ is Polish.
\end{proof}

However, Conjecture \ref{conjecture_generalized_van_engelen} would be false if we altogether dropped the determinacy assumption. In fact, working in $\ZFC$, it is a simple exercise to construct an uncountable zero-dimensional topological group $X$ in which $2^\omega$ does not embed. It is easy to realize that the space $\QQQ\times X$ will be a meager topological group that is not homeomorphic to a filter.

We conclude this subsection by showing that, although the diagram from \S\ref{subsection_preliminaries_homogeneity} might partially collapse, under no circumstance will a complete collapse happen. We will combine ideas from \cite[Proposition 5.4]{medini_semifilters} and \cite[Theorem 3.6.5]{van_engelen_thesis}.

\begin{proposition}\label{proposition_counterexample_semifilter}
There exists a semifilter that is not a topological group.
\end{proposition}
\begin{proof}
Fix infinite sets $\Omega_1$ and $\Omega_2$ such that $\Omega_1\cup\Omega_2=\omega$ and $\Omega_1\cap\Omega_2=\varnothing$. Define
$$
\TT=\{x_1\cup x_2:x_1\subseteq\Omega_1\text{, }x_2\subseteq\Omega_2\text{, and }(x_1\notin\Fin(\Omega_1)\text{ or }x_2\in\Cof(\Omega_2))\},
$$
and observe that $\TT$ is a semifilter. Furthermore, it is clear that $\TT$ is the union of $X=\{x\subseteq\omega:x\cap\Omega_1\notin\Fin(\Omega_1)\}$ and $Y=\{x_1\cup x_2:x_1\in\Fin(\Omega_1)\text{ and }x_2\in\Cof(\Omega_2)\}$. Since $X$ is Polish and $Y$ is countable, it follows that $\TT$ is a Borel space (in fact, it is a $\bD^0_3$ space).

Now assume, in order to get a contradiction, that $\TT$ is a topological group. Observe that $\TT$ is a Baire space because $X$ is dense in $\TT$. Therefore, $\TT$ is Polish by Proposition \ref{proposition_baire_group}. Since $\Cof(\Omega_2)$ is a countable closed subspace of $\TT$ that is crowded, this is contradiction.
\end{proof}

We remark that the semifilter $\TT$ described in the above proof is homeomorphic to the notable space $\mathbf{T}$ introduced by van Douwen (see \cite{van_engelen_van_mill} and \cite[Proposition 5.4]{medini_semifilters}). Since $\QQQ\times\mathbf{T}$ is not a topological group by \cite[Corollary 4.3]{van_engelen_groups}, and multiplication by $\QQQ$ preserves the property of being homeomorphic to a semifilter by \cite[Lemma 8.2]{medini_semifilters}, this shows that that the example given by Proposition \ref{proposition_counterexample_semifilter} could be made meager as well.

\section*{List of symbols and terminology}

The following is a list of most of the symbols and terminology used in this article, organized by the subsection in which they are defined.

\begin{itemize}
\item[\S\ref{subsection_introduction_history}:] homogeneous, Wadge class.
\item[\S\ref{subsection_preliminaries_terminology}:] power-set $\PP(\Omega)$, $\Fin(\Omega)$, $\Cof(\Omega)$, $\Omega^{<\omega}$, $\Ne_s$, monotone, image $f[S]$, inverse image $f^{-1}[S]$, identity function $\id_X$, retraction, base, local base, $\pi$-base, clopen, zero-dimensional, meager subset, comeager subset, meager space, Baire space, embedding, dense embedding, closed function, open function, $\Borel(Z)$, Borel space, $\bS^0_\xi$-measurable function, Borel function, differences operation $\Diff_\eta$, class of differences $\Diff_\eta(\bS^0_\xi(Z))$.
\item[\S\ref{subsection_preliminaries_determinacy}:] play, game, strategy, winning strategy, determinacy assumption $\Det(\bS)$, Axiom of Determinacy ($\AD$), axiom of Projective Determinacy ($\PD$), principle of Dependent Choices ($\DC$), boolean closure $\bool\bS$, topological pointclass, nice topological pointclass, $\bS^1_n(Z)$ for an arbitrary space $Z$, Baire property.
\item[\S\ref{subsection_preliminaries_filters}:] closed under finite modifications, closed under finite intersections, closed under finite unions, upward-closed, downward-closed, semifilter, filter, semiideal, ideal, finite intersection property, finite union property, homeomorphism $h_F$, complement homeomorphism $c=h_\omega$.
\item[\S\ref{subsection_preliminaries_homogeneity}:] strongly homogeneous.
\item[\S\ref{subsection_wadge_fundamental_basics}:] $\bGc$, selfdual class, $\Delta(\bG)$, Wadge-reduction $\leq$, strict Wadge-reduction $<$, selfdual set, Wadge class $A\wc$, continuously closed, $\NSD(Z)$, $\NSDS(Z)$.
\item[\S\ref{subsection_wadge_fundamental_relativization}:] relativized class $\bG(Z)$.
\item[\S\ref{subsection_wadge_fundamental_level}:] class of partitioned unions $\PU_\xi(\bG)$, level $\ell(\bG)$.
\item[\S\ref{subsection_wadge_fundamental_expansions}:] expansion $\bG^{(\xi)}$.
\item[\S\ref{subsection_wadge_fundamental_hausdorff}:] Hausdorff operation $\HH_D(A_0,A_1,\ldots)$, Hausdorff class $\bG_D(Z)$, $W$-universal set, vertical section $S_y$, relativizable pointclass.
\item[\S\ref{subsection_wadge_fundamental_clarifying}:] expanded Hausdorff class $\bG^{(\xi)}_D(Z)$.
\item[\S\ref{subsection_wadge_fundamental_separation}:] (first) separation property, second separation property.
\item[\S\ref{subsection_wadge_fundamental_sd}:] separated differences operation $\SD_\eta$, class of separated differences $\SD_\eta(\bD,\bG^\ast)$.
\item[\S\ref{subsection_wadge_type_basics}:] minimal, successor, countable limit, $\NSDxi(Z)$, $\NSDxiS(Z)$, type $t(\bG)$.
\item[\S\ref{subsection_wadge_type_su}:] class of separated unions $\SU_\xi(\bG)$.
\item[\S\ref{subsection_wadge_closure_preliminaries}:] closed under intersections with $\bP_{1+\xi}$ sets, closed under unions with $\bS_{1+\xi}$ sets, closed under $\SU_\xi$.
\item[\S\ref{subsection_wadge_closure_good}:] good Wadge class.
\item[\S\ref{subsection_homogeneous_complexity}:] $\bS$ space, $\bG$ space, $\Delta(\Diff_\omega(\bS^0_2))$ space, exact complexity.
\item[\S\ref{subsection_filters_squares}:] closed under $n$-th powers, closed under squares.
\end{itemize}

\end{document}